\documentclass{amsart}
\usepackage[english]{babel}
\usepackage{amssymb}
\usepackage{amsmath}
\usepackage{amsthm}
\usepackage{hyperref}
\usepackage{url}
\usepackage{graphicx}
\usepackage[all]{xy}
\usepackage{tikz}

\DeclareMathOperator{\spam}{span}

\DeclareMathOperator{\End}{End}
\DeclareMathOperator{\Aut}{Aut}
\DeclareMathOperator{\GL}{GL}
\DeclareMathOperator{\PGL}{PGL}
\DeclareMathOperator{\PSL}{PSL}

\DeclareMathOperator{\Id}{Id}

\DeclareMathOperator{\im}{Im}
\DeclareMathOperator{\Gr}{Gr}
\DeclareMathOperator{\chr}{char}

\DeclareMathOperator{\Tr}{Tr}

\DeclareMathOperator{\Mat}{Mat}

\DeclareSymbolFont{txsyc}{U}{txsyc}{m}{n}
\DeclareMathSymbol{\coloneqq}{\mathrel}{txsyc}{66}
\DeclareMathSymbol{\eqqcolon}{\mathrel}{txsyc}{67}

\DeclareMathOperator{\Ass}{Ass}

\newtheorem{theorem}{Theorem}[section]
\newtheorem{lemma}[theorem]{Lemma}
\newtheorem{proposition}[theorem]{Proposition}
\newtheorem{definition/proposition}[theorem]{Definition/Proposition}
\newtheorem{corollary}[theorem]{Corollary}

\newtheorem{question}[theorem]{Question}

\newtheorem*{claim1}{Claim 1}
\newtheorem*{claim2}{Claim 2}
\newtheorem*{claim3}{Claim 3}
\newtheorem*{theorem*}{Theorem}
\newtheorem*{lemma*}{Lemma}
\newtheorem*{proposition*}{Proposition}

\theoremstyle{definition}

\newtheorem{definition}[theorem]{Definition}
\newtheorem{example}[theorem]{Example}
\newtheorem{convention}[theorem]{Convention}

\newtheorem{remark}[theorem]{Remark}

\newtheorem{notation}[theorem]{Notation}

\newcommand{\ZZ}{\mathbb{Z}}
\newcommand{\QQ}{\mathbb{Q}}
\newcommand{\RR}{\mathbb{R}}
\newcommand{\CC}{\mathbb{C}}
\newcommand{\HH}{\mathbb{H}}
\newcommand{\OO}{\mathbb{O}}

\newcommand{\tensor}{\otimes}
\newcommand{\isom}{\cong}
\newcommand{\nin}{\not\in}

\newcommand{\lamba}{\lambda}

\newcommand{\Prod}{\prod}
\newcommand{\union}{\cup}

\newcommand{\qbinom}[2]{\genfrac{[}{]}{0pt}{}{#1}{#2}_q}
\title[$q$-Fano planes over infite fields]{Cayley algebras give rise to $q$-Fano planes over certain infinite fields and $q$-covering designs over others}
\author{Vincent van der Noort}

\begin{document}
\begin{abstract}
Let $F$ be a field. A $2$-$(7, 3, 1)_F$-subspace design or \emph{$q$-Fano plane} over $F$, is a $7$-dimensional vector space $V$ over $F$ together with a collection $\mathfrak{B}$ of $3$-dimensional subspaces of $V$ such that every two-dimensional subspace of $V$ is contained in exactly one element $B$ of $\mathfrak{B}$. The question of existence of such a subspace design over any field has been open since the 1970s and has attracted considerable attention in the special case that $F$ is finite. Here we show the existence of $2$-$(7, 3, 1)_F$-subspace designs over a collection of \emph{infinite} fields $F$, including (among others) the fields $\mathbb{Q}$ and $\mathbb{R}$ and the function fields $\mathbb{F}_q(x, y, z)$ with $q$ odd. 

The space $V$ is taken to be the 7-dimensional space of imaginary elements in a non-associative, 8-dimensional Cayley division algebra $O$ over $F$ and the collection $\mathfrak{B}$ consists of the intersections with $V$ of all 4-dimensional (quaternion) subalgebras of $O$. We will present all relevant facts about quaternion and Cayley algebras in a nearly self-contained fashion.  

The second part of the paper studies what happens if we apply the same procedure to the split Cayley algebra over $F$ rather than a Cayley division algebra. (This is relevant since over a large collection of fields, including $\CC$, $\mathbb{F}_{p^n}$ and $\QQ_p$ for all $p$, only split Cayley algebras exist.) By identifying all four-dimensional subalgebras of these algebras, we show that in this case our construction still yields an inclusion minimal $(7, 3, 2)$ $q$-covering design. That is: every two-dimensional subspace of $V$ is contained in at least one element of the resulting set $\mathfrak{B}$ of three dimensional subspaces of $V$ and no proper subset of $\mathfrak{B}$ has this property. However none of these $q$-covering designs are $q$-Fano planes. In the case that $F$ is finite we compute the number of elements of $\mathfrak{B}$. 

For general $F$ of characteristic $\neq 2$, we also give a purely combinatorial `direct' construction (not mentioning the underlying algebra structure) of our $q$-Fano planes and $q$-covering designs for an abstract 7-dimensional $F$-vector space $V$ by identifying the collection $\mathfrak{B}$ as a subvariety of the Grassmanian $\Gr_3(V)$ defined entirely in terms of the classical Fano plane. 
\end{abstract}

\maketitle 
\tableofcontents
%\section{introduction}
\section{Introduction}\label{introduction}
\subsection{Subspace designs, $q$-covering designs and $q$-Fano planes}
A $t$-$(v, k, \lambda)$ \emph{subspace design} over a field $F$ is a collection $\mathcal{B}$ of $k$-dimensional linear subspaces (called \emph{blocks}) of a $v$-dimensional $F$-vector space $V$, with the property that every $t$-dimensional linear subspace of $V$ is contained in exactly $\lambda$ elements of the set $\mathcal{B}$. They are considered the $q$-analogue of the more familiar (combinatorial) \emph{designs} which are defined in a similar way but with sets replacing $F$-vector spaces. Subspace designs over $F$ with $\lambda = 1$ are called \emph{$q$-Steiner systems} over $F$, by analogy with the ordinary Steiner systems.

According to Kiermaier and Laue \cite{KiermaierLaue}, subspace designs were first introduced in 1973 by Cameron \cite{cameron} but the first explicit construction of non-trivial subspace designs appeared only in 1989 in the work of Thomas \cite{thomasdesigns}. Meanwhile the existence of $q$-Steiner systems remained open until the construction, in 2013, of various $2$-$(13, 3, 1)$ subspace designs over $\mathbb{F}_2$ by Braun, Etzion, Ostergard, Vardy and Wasserman \cite{BEOVqSteiner}.

The existence of a $2$-$(7, 3, 1)$ subspace design, also known as a \emph{$q$-Fano plane} over any finite field is unknown. Severe restrictions on a hypothetical $q$-Fano plane over the field $\mathbb{F}_2$ have been proposed by various authors. In particular, from results of Braun, Kiermaier, Kurz, Nakic and Wasserman \cite{BKNqFano} \cite{KiermaierKurzWassermann} we know that the order of the automorphism group of such a design would be at most 2.

Kiermaier \cite{Kiermaiertalk} has called the existence of a $q$-Fano plane over any finite field the `most important open problem in $q$-analogs of designs'.

In this paper we will show the existence of $q$-Fano planes over certain \emph{infinite} fields, most notably all subfields of the field $\mathbb{R}$ of real numbers and the function fields $K(\alpha, \beta, \gamma)$ where $K$ is any field of characteristic unequal to 2.

A concept closely related to that of subspace designs is that of a $q$-covering design. A $q$-covering design with parameters $(v, k, t)$ on a $v$-dimensional vector space $V$ over a field $F$ is a set $\mathfrak{B}$ of $k$-dimensional subspaces of $V$ such that every $t$-dimensional subspace of $V$ is contained in \emph{at least} one elements of $\mathfrak{B}$. Since the set $\Gr_k(V)$ of all $k$-dimensional subspaces of $V$ is a rather trivial example of a $q$-covering design it stands to reason to consider the $q$-covering designs more interesting when they have a smaller number of blocks. By this measure $t$-$(v, k, 1)$-subspace designs are the most interesting $(v, k, t)$-$q$-covering designs. It is easy to see that over a finite field no $q$-covering design can have a smaller number of blocks than a hypothetical $q$-Steiner system over the same field with the same parameters would have: this number can be computed as the number of $t$-dimensional subspaces of $V$ divided by the number of $t$-dimensional subspaces in a given block and neither number depends on the chosen set of blocks. However since for most parameters the existence of subspace designs over finite fields is unknown, it is an open question what is the minimum number $\mathcal{C}_q$ of blocks in a $q$-covering design over a field with $q$ elements. $q$-Covering designs have application in network design as described by Lambert \cite{Lambert}, who also gives upper and lower bounds on the numbers $\mathcal{C}_q$.

\subsection{Statement of the main result and some limitations}\label{introresults}
In the current paper we will show the existence of $q$-Fano planes (that is $2$-$(7, 3, 1)$-subspace designs) over certain infinite fields, including the function fields $K(\alpha, \beta, \gamma)$ where $K$ is any field of characteristic unequal to 2, and all subfields of the field $\mathbb{R}$ of real numbers, so in particular $\RR$ itself. In this introduction we focus on the latter case as it is the easiest to work with.

Formulated in purely combinatorial terms we prove:

\begin{theorem}\label{mainintro} Let $F$ be field of characteristic $\neq 2$ and let $(\mathcal{V}, \mathcal{L})$ be a Fano-plane with vertex set $\mathcal{V}$ and line set $\mathcal{L} \subset \mathcal{V}^3$. We fix an automorphism $\phi$ of $(\mathcal{V}, \mathcal{L})$ of order 7 and choose a cyclic ordering on the three elements of $l$ for each $l \in \mathcal{L}$. The orderings must be `compatible' in the sense that application of $\phi$ will preserve the orderings on the lines. A rather standard way to achieve this is to label the elements of $\mathcal{V}$ $v_0, \ldots, v_6$, and define the ordered lines to be the (cyclically ordered) triples $(v_n, v_{n+1}, v_{n+3})$ where the indices are read modulo 7. (Here the automorphism $\phi$ amounts to cyclically permuting the indices.)

We define $V$ to be the seven-dimensional $F$-vector space with basis $\mathcal{V}$. Let $W = \bigwedge^3 V$, so $W$ is thirtyfive-dimensional. Let $\Delta = v_0 \wedge \phi(v_0) \wedge \phi^2(v_0) \ldots \wedge \phi^6(v_0) \in \bigwedge^7 V \isom F$ where $v_0$ is some element of $\mathcal{V}$. Note that, since 7 is an odd number, $\Delta$ does not depend on the choice of $v_0 \in \mathcal{V}$. Similarly, for each $l \in \mathcal{L}$ let $w_l \in W$ be the wedge-product of the three points in $l$ in the given cyclic order. (Again this is well-defined due to the oddness of 3.) These data in turn define, for every $v \in \mathcal{V}$ a linear functional $\eta_v \colon W \to F$ by $v \wedge (\sum_{l \in \mathcal{L}} w_l) \wedge w = \eta_v(w)\Delta$.

Let $\Gr_3(V)$ be the set of three-dimensional subspaces of $V$ and $\Gr_1(W)$ the set of one-dimensional subspaces of $W$. Let $\psi \colon \Gr_3(V) \to \Gr_1(W)$ be the Pl\"{u}cker embedding. Recall that $\psi$ sends a three-dimensional subsapce $B \subset V$ to the line $F(b_1 \wedge b_2 \wedge b_3) \in W$ where $\{b_1, b_2, b_3\}$ is \emph{any} basis of $B$; this assignment is well-known to be both well defined and injective.

Then the set $\mathfrak{B} = \{B \in \Gr_3(V) \colon \psi(B) \subset \ker \eta_v \textnormal{ for all } v \in \mathcal{V} \}$ is an inclusion-minimal $q$-covering design with parameters $(7, 3, 2)$ on $V$. In other words: every two-dimensional subspace of $V$ is contained in at least one element of $\mathfrak{B}$ and no proper subcollection of $\mathfrak{B}$ has this property. 

Moreover, if $F$ is a subfield of the field of real numbers $\mathbb{R}$ (such as $\mathbb{R}$ itself) then $\mathfrak{B}$ is a $q$-Fano plane, that is: if $F \subset \mathbb{R}$ we have that every two-dimensional subspace of $V$ is contained in \emph{exactly one} element of $\mathfrak{B}$. 

Further, if $F$ is a finite field with $q$ elements (recall that we required $\chr(F)$ and hence $q$ to be odd) then the $q$-covering design has $\qbinom{6}{2}$ blocks. 
\end{theorem}
Thus, in the finite field case, the number of blocks in the subspace design has the same leading term ($q^8$) as the number $\qbinom{7}{2}/\qbinom{3}{2}$ of blocks in a hypothetical $q$-Fano plane over $F$, but exceeds this number by $q[7]_q + q^4$.

We stress that the failure of the above construction to produce a $q$-Fano plane over finite fields is really due to the finiteness of the field and not to it having positive characteristic. As remarked above a different but closely related construction yields a $q$-Fano plane over fields $F = K(\alpha, \beta, \gamma)$ which, in case $K$ has a finite (but odd) number of elements, is an infinite field of positive characteristic. A more general version of Theorem \ref{mainintro}, which encompasses both cases as well as a number of others is given in Section \ref{combinatorics} (Thm. \ref{maincombinatorially}). 

Besides finite fields there is a another well known class of fields which are necessarily infinite but over which our construction fails to produce a $q$-Fano plane: fields that are algebraically closed. A precise criterion for over which fields our construction gives a $q$-Fano plane is given in Section \ref{Serre}.

Working over infinite fields robs us of the opportunity of comparing which of two $q$-covering designs is smaller by counting the number of blocks. It still seems fair however to say that a $q$-covering design $\mathfrak{B}'$ is smaller than a $q$-covering design $\mathfrak{B}$ with the same parameters when $\mathfrak{B}' \subset \mathfrak{B}$. In that sense it is encouraging that for the $q$-covering designs given by Theorem \ref{mainintro} no such `subdesigns' $\mathfrak{B}'$ exist. At the same time however, our construction shows that this `inclusion-minimality' is of limited use as a means to compare $q$-covering designs: the more general version of Thm \ref{mainintro} (Thm \ref{maincombinatorially}) implies that inclusion-minimal $(7, 3, 2)$-$q$-covering designs that are \emph{not} $q$-Fano planes can be constructed over \emph{any} field $F$, including fields $F$ where $q$-Fano planes do exist, such as $\mathbb{R}$. Also, in the case of $F = \mathbb{F}_q$ the fact that the $q$-covering design of Thm \ref{mainintro} is inclusion minimal does not by itself prove that the upper bound of  $\qbinom{6}{2} = q^8 + q^7 + 2q^6 + 2q^5 + 3q^4 + 2q^3 + 2q^2 + q + 1$ on $\mathcal{C}_q(7, 3, 2)$ is sharp. In fact it is not: carrying out the construction in Section 4.4.2 of Lambert's thesis \cite{Lambert}  for parameters $(7, 3, 2)$ one arrives at a lower upper bound of $q^8 + 2q^6 + 3q^4 + q^3 + 2q^2 + q + 1$.

In the present work, we will not discuss fields of characteristic 2 nor $q$-covering designs with different parameters than $(7, 3, 2)$.

\subsection{Real numbers, complex numbers, quaternions, octonions}
Theorem \ref{mainintro} was formulated in combinatorial terms, starting with the classical Fano plane and `quantizing' it into the $q$-Fano plane. This was done partly order to obtain a self-contained statement and partly in the hope that it would allow my future self or the reader to guess the right generalization to designs with other parameters than $2$-$(7, 3, 1)$. However most of the proof (and of the current paper) does not refer to the classical Fano plane at all and instead revolves around \emph{Cayley algebras}, the generalization to arbitrary fields of an extension of the real numbers called \emph{octonions}, named so since they occupy an 8-dimensional real vector space. It is with pleasure that I quote Baez \cite{baez}:

\begin{quote}
`There are exactly four normed division algebras: the real numbers ($\RR$), complex numbers ($\CC$), quaternions ($\HH$), and octonions ($\OO$). The real numbers are the dependable breadwinner of the family, the complete ordered field we all rely on. The complex numbers are a slightly flashier but still respectable younger brother: not ordered, but algebraically complete. The quaternions, being noncommutative, are the eccentric cousin who is shunned at important family gatherings. But the octonions are the crazy old uncle nobody lets out of the attic: they are nonassociative.'
\end{quote}

It is because of this non-associativity (and subsequent obscurity) that we spend a sizable portion of the paper on presenting the preliminaries. Section \ref{preliminaries} will recall in detail the theory of quaternion and Cayley algebras over general fields of characteristic $\neq 2$. However in this introduction we will first discuss in a more concrete fashion the real algebras $\HH$ and $\OO$ that stand at the cradle of this theory. We'll spend more time at the quaternions as they are easier to work with and already illustrate the most important concepts. 

\subsubsection{Complex numbers and quaternions}
The octonions $\OO$ were independently discovered by Graves in 1843 and Cayley in 1845, both building on Hamilton's 1843 discovery of the quaternions $\HH$. (\cite{baez}). Hamilton was motivated by geometry (\cite{baez}): understanding how rotations and translations in two-dimensional geometry could be understood as multiplication and addition of complex numbers, Hamilton sought an algebraic structure that would give a similar description for rotations and translations in three dimensions. The quaternions achieve this exceptionally well (\cite{baez}, \cite{ConwaySmith}) and their construction from the reals is very similar to that of the complex numbers. Indeed, where the complex numbers are obtained by attaching a square root $i$ of $-1$ to $\RR$, the quaternions are obtained by attaching \emph{three} square roots of $-1$, called $i, j, k$ and which are related by
$$ij = k; \qquad ji = -k$$
and cyclic permutions of those equations, so
$$jk = i; \qquad kj = -i; \qquad ki = j; \qquad ik = -j.$$
A \emph{quaternion} is any element of the four-dimensional vectorspace spanned by $1, i, j, k$ with multiplication implied by the above equations and $i^2 = j^2 = k^2 = -1$. Before moving on to the octonions we discuss a few properties of the quaternions which will be important in the sequel.

First, as mentioned before, the quaternions form a \emph{division algebra}, which for the context of this paper can be understood to mean that there are no zero-divisors:
\begin{equation}\label{zerodivisors}
ab = 0 \Rightarrow a = 0 \textrm{ or } b = 0.
\end{equation}
(The precise relation between this property and what one would ordinarily call division is discussed in Section \ref{divisionalg}.) A standard way to see that this property holds in the complex numbers is to show that the Euclidian norm $\|.\|$ on $\CC \isom \RR^2$ is multiplicative: for complex numbers $x, y$ we have that $\|xy\| = \|x\|\|y\|$, assigning to any hypothetical counterexample to (\ref{zerodivisors}) in $\CC$ a corresponding example in $\RR$ where we already know that (\ref{zerodivisors}) holds. Multiplicativity of the norm in turn is achieved by expressing the square of the norm as the product of a complex number $a = \alpha + \beta i$ by its complex conjugate $a^* \coloneqq \alpha - \beta i$ and noticing that taking complex conjugates is an automorphism of $\CC$: $(xy)^* = x^*y^* = y^*x^*$ so that
\begin{equation}\label{normmult}
\|xy\|^2 = xyy^*x^* = \|y\|^2xx^* = \|y\|^2\|x\|^2 = \|x\|^2\|y\|^2
\end{equation}
Interestingly the proof in the quaternion case is almost exactly the same. Here we define the analogue of complex conjugation $x \mapsto x^*$ by 
\begin{equation}\label{quaternionconj}
(\alpha + \beta i + \gamma j + \delta k)^* = \alpha - \beta i - \gamma j - \delta k
\end{equation}
(which reduces to the ordinary complex conjugation when $\delta = \gamma = 0$) and we notice that again $xx^* = \|x\|^2$, where $\|.\|$ denotes the standard Euclidean norm on $\RR^4 \isom \HH$. Here the non-commutativity in the definition of the multiplication nicely helps us to get rid of the cross-terms and end up with a sum of squares. At the same time this non-commutativity prevents the conjugation map from being an automomorphism this time. However, it does make it into an \emph{anti-automorphsim}, that is: we have
$$(xy)^* = y^*x^*$$
and hence (\ref{normmult}) goes through unmodified and we conclude that indeed $\HH$ is a division algebra.

The second feature of $\HH$ we want to stress, before moving on to discussing octonions, is its richness in subalgebras isomorphic to $\CC$. We `created' the quaternions by attaching three squareroots ($i, j, k$) of $-1$ to $\mathbb{R}$. But in doing so we introduced infinitely many more. Every element $i'$ of Euclidean norm 1 in the three-dimensional space $\im(\mathbb{H}) = \spam(i, j, k)$ is a squareroot of $-1$. (This follows directly from the claim above that $xx^* = \|x\|^2$: for every $x \in \im(\HH)$ we have by (\ref{quaternionconj}) that $x^* = -x$ and hence $x^2 = -\|x\|^2 \in \mathbb{R}_{\leq 0}$.) Elements of $\im(\mathbb{H})$ are called \emph{imaginary quaternions}. 

For every imaginary norm-one quaternion $i'$, the two-dimensional subspace $\CC' \coloneqq \spam(1, i') \subset \HH$ is isomorphic (not just as a field, but also as an $\mathbb{R}$-algebra) to $\CC$; where complex conjugation is just the restriction to $\CC'$ of the quaternion conjugation (\ref{quaternionconj}) and where the one-dimensional space $\im(\CC')$ (the scalar multiples of $i'$) is just the intersection of $\CC'$ with $\im(\HH)$. 

More generally, the vector space decomposition $\HH = \mathbb{R}1 \oplus \im(\mathbb{H})$ is very similar to the decompostion $\CC = \RR 1 \oplus \im(\CC)$: in both cases the first summand consists of all elements whose square is a positive real number and the second of all elements whose square is a negative real number and in both cases the linear map ${}^*$ which acts as $1$ on the first summand and as $-1$ on the second is an anti-automorphism that reproduces the euclidean norm through $x^*x = \|x\|^2$. Of course, setting $\im(\mathbb{\RR}) = \{0\}$, the same facts hold in $\RR$ as well. 

\subsubsection{Octonions and the classical Fano plane}\label{octonions}
Given the above the following will not come as a surprise. The octonions $\OO$ are formed by attaching to $\RR$ no less than $7$ squareroots of $-1$ (following Conway-Smith \cite{ConwaySmith} we will call them $e_0, \ldots, e_6$) that, for $i \neq j$ satisfy $e_ie_j = -e_je_i \in \im(\OO)$ where $\im(\OO) = \spam(e_0, \ldots, e_6)$. From here it follows as before that $\OO = \mathbb{R}1 \oplus \im(\OO)$ and that by defining ${}^*$ as the linear involution which acts as the identity on the first summand and as minus the identity on the second we have that ${}^*$ is an anti-automorphism satisfying $x^*x = \|x\|^2$ where $\|.\|$ denotes the Euclidean norm on $\RR^8$. What all this does not tell us is how exactly the $e_i$ interact. The multiplication table is given in Table \ref{multtabclassic}, but the structure becomes a bit clearer in the following description taken from \cite{ConwaySmith}:

$$e_ie_{i+1} = e_{i + 3} \qquad e_{i+1}e_i = -e_{i + 3}$$
Where indices are read $\mod 7$, together with cyclic permutations of these equations, so
$$e_{i+1}e_{i +3} = e_i; \qquad e_{i+3}e_{i +1} = -e_i; \qquad e_{i+3}e_{i} = e_{i+1}; \qquad e_{i}e_{i +3} = -e_{i+1};.$$

It is well known that the cyclic group on seven elements together with the set of all seven triples $(i, i + 1, i+ 3)$ from that group forms a model of the classical Fano plane. Indeed the Steiner triple system property of the Fano plane is why the above equations (together with $e_i^2 = -1$) are enough to determine the multiplication on any pair of generators (and hence on all of $\OO$). The Fano plane structure also sheds some light on the infamous non-associativity of $\OO$. For every `line' $(e_i, e_{i+1}, e_{i + 3})$ we have that products of elements in their span \emph{is} associative: in fact we see from the above that $\spam(1, e_i, e_{i+1}, e_{i+3})$ is a subalgebra of $\OO$ isomorphic to $\mathbb{H}$. On the other hand, for triples of basis elements \emph{not} on a Fano-line, associativity fails (mildly): for instance for the triple $(e_0, e_1, e_2)$ we have that $(e_0e_1)e_2 = e_3e_2 = -e_5$ while $e_0(e_1e_2) = e_0e_4 = e_5$. 

While the seven Fano-lines each determine a quaternion subalgebra of $\OO$, the total number of such subalgebras is much larger. Analogous to how each line in $\im(\HH)$ generates a subalgebra of $\HH$ isomorphic to $\mathbb{C}$ we have in $\OO$ the following result:

\begin{proposition}\label{HinO}
Each two-dimensional subspace of $\im(\mathbb{O})$ generates a four-dimensional subalgebra of $\OO$ isomorphic to $\mathbb{H}$.
\end{proposition}
The proof of this proposition is non-trivial, but we won't discuss it here as we will prove a more general result down the line. 

We \emph{will} discuss here however a number of interesting consequences of Proposition \ref{HinO}. First we notice:
\begin{corollary}\label{altO}
Every subalgebra of $\OO$ generated by 2 elements is associative.
\end{corollary}
Algebras with this property are called \emph{alternative}. Since it is not hard to see that the subalgebras generated by 2 elements are closed under the involution ${}^*$ as well, we conclude from Corollary \ref{altO} (and the already estabilished properties of ${}^*$) that (\ref{normmult}) holds for all $x, y$ in spite of $\OO$ not being associative. Hence we conclude:
\begin{corollary}
The multiplication of $\OO$ satisfies $\|xy\| = \|x\|\|y\|$ where $\|.\|$ is the Euclidean norm on $\RR^8$ and hence $\OO$ is a division algebra.
\end{corollary}
Of more interest to us, however, is the following corollary to Proposition \ref{HinO}:
\begin{theorem}\label{Oqfano}
Let $\mathcal{H}$ be the collection of all four-dimensional subalgebras of $\OO$ that are isomorphic to $\mathbb{H}$ and let $\mathcal{B} = \{\im(H) \colon H \in \mathcal{H}\}$ be the collection of their three-dimensional subspaces of imaginary elements. Then each $B \in \mathcal{B}$ is a three-dimensional subspace of the seven-dimensional real vector space $\im(\OO)$ and the pair $(\im(\OO), \mathcal{B})$ is a $q$-Fano plane over $\RR$.
\end{theorem}
Doubtless, Proposition \ref{HinO} has been known to experts in the field for a long time, but to the best of my knowledge nobody has openly made the connection with $q$-Fano planes before. The reason that the current article is not just two lines long, is that we want to understand what happens over fields different from $\RR$.

\begin{table}\caption{Multiplication table of $\OO$ with respect to the basis $e_0, \ldots, e_6$ described in Section \ref{octonions} \label{multtabclassic}}
\begin{tabular}{l|rrrrrrrr}
      &  $1$  &    $e_0$ &   $e_1$ &   $e_2$ &   $e_3$ &   $e_4$ &   $e_5$ &   $e_6$  \\ \hline
  $1$ &  $1$  &    $e_0$ &   $e_1$ &   $e_2$ &   $e_3$ &   $e_4$ &   $e_5$ &   $e_6$ \\
$e_0$ & $e_0$ &     $-1$ &   $e_3$ &   $e_6$ &  $-e_1$ &   $e_5$ &  $-e_4$ &  $-e_2$ \\
$e_1$ & $e_1$ &   $-e_3$ &    $-1$ &   $e_4$ &   $e_0$ &   $e_2$ &   $e_6$ &  $-e_5$ \\
$e_2$ & $e_2$ &   $-e_6$ &  $-e_4$ &    $-1$ &   $e_5$ &   $e_1$ &  $-e_3$ &   $e_0$ \\
$e_3$ & $e_3$ &    $e_1$ &  $-e_0$ &  $-e_5$ &    $-1$ &   $e_6$ &   $e_2$ &  $-e_4$ \\
$e_4$ & $e_4$ &   $-e_5$ &   $e_2$ &  $-e_1$ &  $-e_6$ &    $-1$ &   $e_0$ &   $e_3$ \\
$e_5$ & $e_5$ &    $e_4$ &  $-e_6$ &   $e_3$ &   $e_2$ &  $-e_0$ &    $-1$ &   $e_1$ \\
$e_6$ & $e_6$ &    $e_2$ &   $e_5$ &  $-e_6$ &   $e_4$ &  $-e_3$ &  $-e_1$ &    $-1$
\end{tabular}
\end{table}

\subsection{Overview of the article}
In 1923 Dickson generalized the procedure of obtaining $\CC$ from $\mathbb{R}$, $\HH$ from $\CC$ and $\OO$ from $\HH$ into what is now known as the `Cayley-Dickson Procedure', a recipe of creating a family of $2^n$-dimensional algebras over a given field $F$ of characteristic $\neq 2$. The eight-dimensional members of this family are called \emph{Cayley algebras} and share with $\OO$ the important property of being alternative (cf Cor. \ref{altO} above). Also the decomposition $\mathcal{O} = F1 \oplus \im(\mathcal{O})$ makes sense over these algebras (as well as their lower-dimensional counterparts) and behaves largely the same as in the real cases described above. However not all Cayley algebras are division algebras. Cayley algebras that contain zero-divisors are called \emph{split}.

Section \ref{preliminaries} covers the preliminaries from non-associative algebra, focussing towards the structure of the 1, 2, 4 and 8-dimensional algebras generated in the Cayley-Dickson process over arbitrary fields of characteristic not 2. In Section \ref{qfano} we then prove our first main result:

\begin{theorem*}[\ref{quantumfano}]
Let $\mathcal{O}$ be a Cayley division algebra (i.e. a Cayley algebra that is also a division algebra) over $F$ and let $V = \im(\mathcal{O})$. Let $\mathcal{H}$ be the collection of four-dimensional associative subalgebras of $\mathcal{O}$ and $\mathcal{B} \coloneqq \{\im(H) \colon H \in \mathcal{H}\}$. Then the pair $(V, \mathcal{B})$ is a $q$-Fano plane over $F$.
\end{theorem*} 

Theorem \ref{quantumfano} generalizes Theorem \ref{Oqfano} and hence it is not surprising that it is derived from a more general version of Proposition \ref{HinO} above:

\begin{proposition*}[\ref{main}]
In the setting of Thm \ref{quantumfano} above, every two-dimensional subspace of $V$ generates a four-dimensional associative subalgebra of $\mathcal{O}$.
\end{proposition*}

Section \ref{split} focusses on understanding the role of the requirement that $\mathcal{O}$ is a division algebra in the statement and proof of Proposition \ref{main}, by studying what would happen in case $\mathcal{O}$ is split. To this end we divise a classification of all two-dimensional subspaces of $\im(\mathcal{O})$ into six types, based on the properties of the subalgebras they generate. Here we are greatly helped by the fact that, quite contrary to the famous dictum by Tolstoy, each Cayley division algebra is a divsion algebra in its own way, but all split Cayley algebras are alike. `Alike' here means concretely that not only are all split Cayley algebras over a given field $F$ isomorphic (Theorem \ref{splitunique}), but even stronger: they can all be obtained by tensoring $F$ over $\mathbb{Z}$ with the same non-associative ring (Lemma \ref{tab}). Based on our classification of two-dimensional subspaces we conclude that, with notation as in Thm. \ref{quantumfano} above, no subset of $\mathcal{B}$ can be a $q$-Fano plane on $V$ in case $\mathcal{O}$ contains zero divisors (Corrolary \ref{hopeless}).

In Section \ref{qcover} we then put a positive spin on the results of the previous section as we prove the second main result of the paper:
\begin{theorem*}[\ref{main2}]
Let $\mathcal{O}$ be a Cayley algebra over a field $F$ of characteristic $\neq 2$ and let $V = \im(\mathcal{O})$. Let $\mathcal{H}$ be the collection of four-dimensional subalgebras of $\mathcal{O}$ and $\mathcal{B} \coloneqq \{\im(H) \colon H \in \mathcal{H}\}$. Then the pair $(V, \mathcal{B})$ is a $q$-covering design over $F$ with parameters  $(7, 3, 2)$, which is inclusion-minimal in the sense that no proper subset of $\mathcal{B}$ is has the $q$-covering design property. If moreover $\mathcal{O}$ is a division algebra, then $(V, \mathcal{B})$ is a $q$-Fano plane over $F$.
\end{theorem*} 

The next three sections are dedicated to reformulating Theorem \ref{main2} into more beautiful forms. To see what we mean by that, we return to the appearance of the classical Fano-plane inside the real octonions $\OO$. Having settled on taking the basis elements $e_0, \ldots, e_6$ as the point set of our Fano plane, Theorem \ref{main2} suggests a way of describing the collection of lines: a triple $\{e_i, e_j, e_k\}$ forms a line if and only if the eight element set $\{\pm 1, \pm e_i, \pm e_j, \pm e_k \}$ is closed under multiplication. Indeed it is easy to verify from the definition of octonion multiplication in Section \ref{octonions} that this reproduces the same collection of lines we met in that section. However this description of the Fano plane is also a bit ugly, involving the 9 additional elements $\pm 1, -e_0, \ldots, -e_6$ that are not needed as points in the Fano-plane itself. In Section \ref{octonions} we met a much prettier description of the line set: $\{e_i, e_j, e_k\}$ is a line in the Fano plane if and only if $(e_ie_j)e_k = e_i(e_je_k)$. In Section \ref{associative} we show that this description of the Fano-plane has a $q$-analog as well: the set $\mathcal{B}$ of Theorem \ref{main2} can equivalently be described as 

\begin{align*}
\mathcal{B} 
& = \{B \in \Gr_3(V) \colon (ab)c = a(bc) \textrm{ for all } a, b, c \in B\} \\
& = \{\spam(a, b, c) \colon a, b, c \in V, (ab)c = a(bc) \}
\end{align*}

Here the Grassmanian $\Gr_3(V)$ is the set of all three-dimensional subspaces of $V$.

As a (non-obvious) consequence of this equality we moreover derive (in Section \ref{geometric}) that the set $\mathcal{B}$ allows a more geometric description: embedding $\Gr_3(V)$ into $\mathbb{P}(W)$ where $W \coloneqq \bigwedge^3V$ by the Pl\"ucker embedding $\psi$, we find that the subset $\mathcal{B}$ is described as the intersection of $ \psi(\Gr_3(V))$ with the image under the projection map of a 28-dimensional \emph{linear} subspace $K$ of the 35-dimensional space $W$.

In Section \ref{combinatorics} we then set out to describe this space $K$ in entirely combinatorial terms. The result is a slightly more general version of Theorem \ref{mainintro} above (Theorem \ref{maincombinatorially}) that provides a recipe for creating a $(7, 3, 2)$-$q$-covering design over a given field $F$ (with $\chr(F) \neq 2$) from the classical Fano plane in an `algebra-free' way. The motivation for including such a result is the hope that upon reading it someone might `guess' the correct way to create $q$-versions of Steiner triple systems at different parameters, where the underlying algebra is not available. Needless to say, the proof of such a result would have to be quite different from the proof of Theorem \ref{maincombinatorially} (which relies on Thm \ref{main2} above) and we won't speculate on it here. 

Although similar in form, Thm \ref{maincombinatorially} is more general than Theorem \ref{mainintro} above, encompassing among others the $q$-Fano planes (not covered by Thm \ref{mainintro} but mentioned just below it) over the function fields $F = K(\alpha, \beta, \gamma)$ of transedence degree three over any field $K$ with $\chr(K) \neq 2$. All statements in Theorem \ref{mainintro} above follow as straightforward special cases of Theorem \ref{maincombinatorially} except for the last, more quantative part, where it states that the number of blocks in the $q$-covering design over a field with $q$ elements is $q^8 + q^7 + 2q^6 + 2q^5 + 3q^4 + 2q^3 + 2q^2 + q + 1$.

The verification of this number is the content of the final section of the paper, Section \ref{counting}. Assisted by the classic fact (discussed in Section \ref{preliminaries}) that all Cayley algebras over a finite field contain zero divisors and hence are split, we use our classification of Section \ref{split} to count the number of four-dimensional subalgebras of such an algebra and hence the minimum number of blocks in a $(7, 3, 2)$-$q$-covering design over a finite field obtainable by our method. 

\subsection{Acknowledgements}
Section 2 of the current paper is loosely based on Chapters 1 and 2 of my (unpublished) master's thesis \cite{scriptie}. Although the objective of that thesis was perpendicular to that of the current paper (defining $q$-analogs of the octonions rather than using the octonions to define $q$-analogs of something else) it was the research done for that project that enabled me, more than a decade later, to recognize the truth of Theorem \ref{Oqfano} upon learning the definition of a $q$-Fano plane. I'd like to thank my thesis advisors, Tom Koornwinder and Eric Opdam for suggesting the topic and for their guidance in those days. 

Most of all I'd like to thank Relinde Jurrius for introducing me to the notions of $q$-Fano plane and $q$-covering design, for careful comments on an earlier version of this paper and for being the go-to expert on all things related to finite geometry over the years.

%\section{Preliminaries from algebra}\label{preliminaries}
\section{Preliminaries from algebra}\label{preliminaries}
\subsection{Conventions and definitions}\label{divisionalg}%raar label, heeft te maken met vooruitwijzing in introductie
\begin{convention}
Throughout the paper, let $F$ be a field of characteristic unequal to 2. We will denote elements of $F$ (scalars) by lower case Greek letters, vector spaces over $F$ will be denoted by capital roman letters and elements of such vector spaces will be denoted by lower case roman letters. 
\end{convention}

\begin{definition}\label{algebra}
An \emph{algebra} over $F$ is a vector space $A$ over $F$ equiped with an bilinear multiplication map $A \times A \to A$. An algebra $A$ is called \emph{unital} if there exists an element $1 \in F$ such that $1a = a1 = a$ for all $a \in A$. It is called associative when $(ab)c = a(bc)$ for all $a, b, c$ in $A$. 
\end{definition}

\begin{convention}
Throughout this paper, by an \emph{algebra} we will denote a finite-dimensional, unital, but not necessarily associative algebra over $F$.
\end{convention}

Let $A$ be an algebra. The unit element of $A$ will be denoted $1$ and the one-dimensional subalgebra of scalar multiples of 1 will be denoted $F1$. By bilinearity of the product on $A$ we have that %for all $\alpha \in F1 \subseteq A, a, b \in A$: 
%
%\begin{equation}\label{center}
%(\alpha a)b = \alpha(ab) = (a \alpha)b = a(\alpha b) =  (ab)\alpha = a(b\alpha).
%\end{equation}
%Conversely the set $Z(A)$ of elements $\alpha \in A$ satisfying (\ref{center}) for all $a, b$ in $A$ is called the center of $A$. An algebra is called \emph{central} if $Z(A) = F1$.
%
%An algebra that contains no non-trivial two-sided ideals is called \emph{simple}. A \emph{quaternion algebra} is defined as a four-dimensional, central simple associative algebra.
all elements of $F1$ commute and associate with all elements of $A$.

\begin{notation}
Let $A$ be an algebra. For any subset $B$ of $A$ we denote by $\langle B \rangle$ the subalgebra of $A$ generated by the elements of $B$. That is: the intersection of all subalgebras of $A$ containing $B$. When $B$ is finite, e.g. $B = \{u, v\}$, we will write $\langle u, v \rangle$ for $\langle \{u, v\}\rangle$.
\end{notation}
\begin{convention} Since we required all our algebras to be unital, we demand the same of their subalgebras. In particular we have that $F1 \subseteq \langle B \rangle$ for any subset $B$ of $A$. %(We can even make the convention $\langle \emptyset \rangle = F$ but we won't use it.) 
\end{convention}
The following elementary observation is stated explicitly, not because I think the reader is stupid but because I woke up one night in panick, believing that there were a hole in the proof amounting to the following fact being false. However it is true and promoted to a lemma so that I can refer back to it later.
\begin{lemma}\label{generated}
Let $A$ be an algebra, let $U \subseteq A$ be a linear subspace and let $u_1, \ldots, u_k$ be a basis of $U$. Then $\langle u_1, \ldots, u_k \rangle = \langle U \rangle$.
\end{lemma}
\begin{proof}
The inclusion $\langle u_1, \ldots, u_k \rangle \subseteq \langle U \rangle$ is obvious from the inclusion $\{ u_1, \ldots, u_k \} \subseteq  U$. For the converse inclusion we note that every algebra containing $\{u_1, \ldots, u_k\}$ contains $U$ and hence so does the intersection $\langle u_1, \ldots, u_k \rangle$ of all such algebras. But this means that $\langle u_1, \ldots, u_k \rangle$ is an algebra containing $U$ and hence it certainly contains the intersection of all such algebras, that is $\langle u_1, \ldots, u_k \rangle \supseteq \langle U \rangle$.
\end{proof}

\begin{definition}
An algebra is \emph{alternative} if every subalgebra generated by two elements is associative.
\end{definition}

Alternative algebras need not be associative, but many results that are familiar for associative algebras do hold for more general alternative algebras as well. We will see an example below in Lemma \ref{division}, many more examples can be found in the book by Schafer \cite{Schafer}. 

\begin{definition}
We say that an algebra $A$ is a \emph{division algebra} if for every non-zero $a \in A$ the linear maps $L_a \colon x \mapsto ax$ and $R_a \colon x \mapsto xa$ are invertible in $\End(A)$.
\end{definition}
\begin{definition}
A non-zero element $x$ in an algebra $A$ is called a \emph{zero-divisor} if there exist non-zero $y \in A$ such that $xy = 0$ or $yx = 0$.
\end{definition}
\begin{lemma}\label{division}
Let $A$ be a finite-dimensional alternative algebra. Then the following are equivalent:
\begin{enumerate}
\item $A$ is a division algebra
\item $A$ has no zero-divisors
\item For every non-zero $a \in A$ there exist an element $a^{-1} \in A$ such that $aa^{-1} = a^{-1}a = 1$.
\end{enumerate}
\end{lemma}
The proof is left as an exercise to the reader. A counterexample to the implications $2 \implies 1$ and $2 \implies 3$ when we drop the condition that $A$ is finite-dimensional is given by the algebra $F[x]$ of polynomials over $F$. A counterexample to the implication $3 \implies 2$ when we drop the condition that $A$ is alternative is given in Section \ref{CayleyDickson}.

\begin{definition}\label{quadraticdef}
An algebra is \emph{quadratic} if the subalgebra generated by the single element $a$ is two-dimensional for all $a \nin F1$.
\end{definition}

\begin{definition}
By an \emph{involution} $*$ of an algebra $A$ we denote a linear anti-automorphism of order dividing 2. In other words and involution is a linear map $a \mapsto a^*$ satisfying 

\begin{equation}\label{star}
(ab)^* = b^*a^* \qquad (a^*)^* = a
\end{equation}
for all $a, b \in A$. 
\end{definition}

By linearity $\alpha^* = \alpha$ for all $\alpha \in F1$. 

\begin{definition}\label{stronginv}
We say that an involution $*$ is a \emph{strong involution} if the converse to the above holds, i.e. if $a^* = a$ if and only if $a \in F1$.
\end{definition}

The most famous example of a strong involution is complex conjugation in the algebra $A = \mathbb{C}$ over $F = \mathbb{R}$. The quaternion conjugation (\ref{quaternionconj}) on the $\mathbb{R}$-algebra $\HH$ discussed in the Introduction is another example of a strong involution. Moreover $\CC$ and $\HH$ are both examples of quadratic algebras over $\mathbb{R}$. We will see in Section \ref{nice} that this is no coincidence.

\subsection{A theorem of Artin}
We defined an alternative algebra as an algebra in which the subalgebra generated by any two elements is associative. Alternativity can be viewed as a `controlled' form of non-associativity, where we bear in mind that all associative algebras are automatically alternative. Clearly for any $a, b$ in an alternative algebra we have that
\begin{eqnarray}
(aa)b = a(ab) \label{leftalt}\\
(ab)a = a(ba) \label{flexible} \\
(ba)a = b(aa) \label{rightalt}
\end{eqnarray}
It is an interesting result of Artin that the converse holds:
\begin{theorem}[\cite{Schafer}, Thm III.3.1]\label{associator}
Let $A$ be an algebra. The following are equivalent
\begin{enumerate}
\item $A$ is alternative
\item At least two of the three equations (\ref{leftalt}, \ref{flexible}, \ref{rightalt}) hold for all $a, b \in A$
\item The \emph{associator}, the trilinear map $(. , ., .) \colon A \times A \times A \to A$ defined by $(a, b, c) = (ab)c - a(bc)$ \emph{alternates}, that is: changes sign under odd permutations of its entries.
\end{enumerate}
\end{theorem}
It is the third item that explains the name alternative. Note that $(a, b, c)  = 0$ for all $a, b, c$ when $A$ is associative, thus providing a second way of seeing that all associative algebras are alternative.

\subsection{A particularly nice class of algebras}\label{nice}
Let $A$ be a quadratic algebra and $a \in A, a \nin F1$. By definition of quadratic algebra there are unique scalars $\tau(a), n(a) \in F$ such that 
\begin{equation}\label{taun}
a^2 = 2\tau(a)a - n(a)
\end{equation}
This defines maps $\tau, n$ from $A \backslash F1 \to F$ and we extend these maps to all of $A$ by setting $\tau(\alpha) = \alpha, n(\alpha) = \alpha^2$ for $\alpha \in F1$. Note that with this convention equation (\ref{taun}) holds for all $a \in A$.

It is an easy but important observation that the map $\tau \colon A \to F$ is linear. Moreover, and this is the reason for the appearance of the number 2 in (\ref{taun}), we have that $\tau$ is a linear projection operator onto the space $F1$, that is, we have that $\tau \circ \tau = \Id$. When $F = \mathbb{R}$ we think of $\tau$ as `taking the real part of an element'. 

We define $\im(A) = \ker \tau$. Since $\tau$ is a projection operator we obtain a decompostion (of linear spaces, not of algebras):
\begin{equation}\label{ReIm}
A = \tau(A) \oplus \im(A).
\end{equation} 

The elements of the co-dimension-1 linear space $\im(A)$ are called \emph{imaginary}. The reason for this terminology stems from the special case $F = \mathbb{R}, A = \mathbb{C}$ and becomes apparent from the following alternative characterization of $\im(A)$ (the equivalence of the two characterizations dates back to Frobenius and follows easily from (\ref{taun})):
\begin{equation}\label{im}
\im(A) = \{a \in A \colon a^2 \in F1 \textrm{ but } a \nin F1\} \union \{0\}.
\end{equation}
We will be interested in quadratic algebras satisfying the additional condition
\begin{equation}\label{tauab}
\tau(ab) = \tau(ba) \textnormal{ for all } a, b \in A.
\end{equation}
We have:
\begin{theorem}\label{quadratic}
Let $A$ be an algebra. The following are equivalent: 
\begin{enumerate}
\item $A$ is a quadratic algebra satisfying (\ref{tauab}), 
\item $A$ possesses a strong involution.
\end{enumerate}
\end{theorem}
\begin{proof}
`$1 \implies 2$': Given the quadratic algebra structure we write $\tau$ for the map $a \mapsto \tau(a)1$ (in other words, we reinterpret $\tau$ as mapping into the subalgebra $F1$ of $A$ rather than into the `external' field $F$.) We write $\im = \Id - \tau$ for the linear projection onto the space $\im(A)$ along the decompostion (\ref{ReIm}).

The involution $*$ is then given by 
\begin{equation}\label{starfromtau}
a^* = \tau(a) - \im(a)
\end{equation}
Let $u, v \in \im(A)$. From (\ref{im}) we have that $(u + v)^2$, $u^2$, and $v^2$ are all in $F1$ so that from expanding $(u + v)^2$ we find that $uv + vu \in F1$ as well. It follows that $\tau(uv + vu) = uv + vu$ and hence $\im(uv + vu) = 0$ so that $\im(uv) = -\im(vu)$. 

Combining this with (\ref{tauab}) and (\ref{starfromtau}) yields $(uv)^* = \tau(uv) - \im(uv) = \tau(vu) + \im(vu) = vu$. Now since $u, v \in \im(A)$, $v^* = -v$ and $u^* = -u$ by (\ref{starfromtau}) and hence the equation $(uv)^* = vu$ extends to $(uv)^* = v^*u^*$, showing that for $u, v \in \im(A)$ the first relation in (\ref{star}) holds. Writing $a = \alpha + u$, $b = \beta + v$ with $\alpha, \beta \in F1, u, v \in \im(A)$ we can easily extend this to general $a, b \in A$ exploiting the linearity of the $*$-operator.

The second equation in (\ref{star}) is obvious from (\ref{starfromtau}). This shows that ${}^*$ is an involution; the fact that it is strong follows directly from (\ref{starfromtau}).

`$2 \implies 1$': Given a strong involution $*$ we define the maps $\tau$ and $n$ by 
\begin{equation}\label{taustar}
\tau(a) = \frac{1}{2}(a + a^*)
\end{equation}
 and
\begin{equation}\label{n}
n(a) = a^*a
\end{equation}
Note that the strongness of the involution guarantees that $\tau$ and $n$ take values inside $F1$, allowing us to reinterpret them as taking values in $F$. It is then straightforward to verify from (\ref{star}) that (\ref{taun}) and (\ref{tauab}) hold. (\ref{taun}) in turn implies that $A$ is quadratic. 
\end{proof}

Theorem \ref{quadratic} provides a third characterization of $\im(A)$ for algebras $A$ with a strong involution:
\begin{equation}\label{im2}
\im(A) = \{a \in A \colon a^* = -a\}
\end{equation}
In other words: $\im(A)$ is the $(-1)$-eigenspace of the strong involution. Since by definition of `strongness' the $1$-eigenspace is one-dimensional and since linear transformations of order $\leq 2$ cannot have proper Jordan blocks, we obtain a second proof of a fact that we will now give its own number: 

\begin{lemma}\label{codim1}
The co-dimension of $\im(A)$ in an algebra $A$ equipped with a strong involution equals 1.
\end{lemma}

A very useful property of imaginary elements is
\begin{equation}\label{anticommutative}
uv = -vu \qquad \textrm{ for all } u, v \in \im(A)
\end{equation}
Which follows from combining (\ref{im2}) with (\ref{star}). 

Finally we note that from (\ref{n}) and (\ref{star}) if follows that in alternative algebras $A$ with a strong involution the function $n$ is multiplicative:
\begin{equation}\label{nmult}
n(ab) = n(a)n(b) \qquad \textnormal{for all } a, b \in A.
\end{equation}
For the $(F = \mathbb{R})$-algebras discussed in the introduction we have that the function $n$ equals the square of the Euclidean norm on $A$ and the proof of (\ref{nmult}) is spelled out in (\ref{normmult}).

\subsection{The Dickson Double}
\begin{definition}\label{DicksonDouble}
Let $A$ be an $F$-algebra with a strong involution $*$ and let $\gamma \in F \backslash \{0\}$ be a non-zero element of $F$. The \emph{Dickson double} $D_\gamma(A)$ of $A$ is an algebra with dimension twice the dimension of $A$ formed by attaching to $A$ an element $i$ satisfying 
\begin{equation}\label{i2}
i^2 = \gamma.
\end{equation} 
In more detail we have that as a vector space $D_\gamma(A)$ equals $A \oplus iA$ where elements of the subspace $iA$ are interpreted as the product of the `new' element $i = i1$ with elements of $A$. Multiplication is defined by the formula
\begin{equation}\label{eyesleft}
(a + ib)(c + id) = (ac + \gamma d b^*) + i(a^*d + cb)
\end{equation} 
The involution $*$ on $A$ is extended to $D_\gamma(A)$ as the unique involution agreeing with the original $*$ on the subalgebra $A$ and satisfying $i^* = -i$.
\end{definition}

\begin{lemma}\label{multiplication}
The multiplication (\ref{eyesleft}) can be equivalently be defined by the three equations 
\begin{eqnarray}
p(iq) = i(p^*q)  \label{pq1}\\
(pi)q = (pq^*)i \label{pq2} \\
(ip)(qi) = \gamma(pq)^* \label{pq3}
\end{eqnarray}
for all $p, q \in A$ or by the single equation
\begin{equation}
(a + bi)(c + di) = (ac + \gamma d^*b) + (da + bc^*)i.
\end{equation}
\end{lemma}
Particularly useful in computations is also the corollary
\begin{equation}\label{xi}
xi = ix^* \textnormal{ for all } x \in A
\end{equation}
from which in turn it follow that
\begin{equation}\label{xistar}
(xi)^* = -(xi) \textnormal{ for all } x \in A
\end{equation}

From (\ref{xistar}) we note:
\begin{corollary}\label{strong}
The involution $*$ on $D_\gamma(A)$ is strong.
\end{corollary}

We'll use the last corollary to try and avoid doing computations using (\ref{eyesleft} - \ref{xistar}) in this paper and instead rely on the properties of algebras with a strong involution derived in Section \ref{nice}.

\begin{remark}\label{allfromxi}
One of the things apparent from equations (\ref{pq1}, \ref{pq2}, \ref{pq3}) is that we cannot assume the Dickson double $D$ of an algebra $A$ to be associative, even if $A$ is. On the other hand it is easy to see that if we happen to know that $A$ is a subalgebra of an associative algebra $D$ and $i$ is an element of $D$ not contained in $A$ then (\ref{xi}), if true in $D$, implies equations (\ref{pq1}, \ref{pq2}, \ref{pq3}) to hold in $D$ as well.
\end{remark}

The following theorem, due to Albert, (which can be derived directly from equations (\ref{pq1}, \ref{pq2}, \ref{pq3}) in conjunction with Artin's theorem above) states however that for $A$ of dimension 1, 2 or 4 (and $D$ subsequently of dimension 2, 4 or 8) the non-associativity is controlled.

\begin{theorem}\label{Albert}
Let $A$ be a an $F$-algebra with a strong involution and let $D$ be its Dickson double. Then:

$D$ is commutative if and only if $A = F$.

$D$ is associative if and only if $A$ is commutative.

$D$ is alternative if and only if $A$ is associative.
\end{theorem}

Now that alternativity is on the table it is interesting to know that Remark \ref{allfromxi} in fact generalizes to the alternative case though with a much more complicated proof:

\begin{lemma}\label{allfromxi2}
Let $D$ be an alternative algebra with strong involution $*$, $A$ a subalgebra of $D$ that is closed under application of $*$, and let $i \in \im(D)$ be an element \emph{not} contained in $A$ satisfying (\ref{xi}). Define $\gamma \in F$ by $i^2 = \gamma1$. Then $(\ref{pq1}, \ref{pq2}, \ref{pq3})$ hold in $D$ for all $p, q \in A$. It follows that if $\gamma \neq 0$ then $D$ contains a copy of the Dickson double $D_\gamma(A)$ of $A$.
\end{lemma}
\begin{proof}
We'll start by writing out the proof of (\ref{pq2}).
$$(pi)q = p(iq) + (p, i, q) = p(q^*i) + (p, i, q) = (pq^*)i - (p, q^*, i) + (p, i, q)$$
where the first and third equality apply the definition of the associator $(.,.,.)$ given in Theorem \ref{associator} and the second equality applies (\ref{xi}). Hence equation (\ref{pq2}) follows as soon as we show that $(p, i, q) - (p, q^*, i) = 0$. But by alternativity and trilineairty of the associator (Thm \ref{associator}) we find that
$$(p, i, q) - (p, q^*, i) = (p, i, q) + (p, i, q^*) = (p, i, q + q^*) = (p, i, 2\tau(q)1)$$
where the last equality uses (\ref{taustar}). And since elements of $F1$ associate with all elements of $D$ we find that $(p, i, 2\tau(q)1) = 0$ as  desired.

Equality (\ref{pq1} can be derived from (\ref{pq2}) and (\ref{xi}) as follows. By (\ref{xi}) and (\ref{star}) we see that the right hand side of (\ref{pq1}) equals $(q^*p)i$. Applying (\ref{pq2}) with $q^*$ in the role of $p$ and $p^*$ in the role of $q$ we obtain that this equals $(q^*i)p^*$. We have thus obtained the equality $(pq^*)i = (q^*i)p^*$. 

Applying the $*$-operator to both sides of this last equality we find on te left hand side $-(pq^*)i$ by (\ref{xistar}) and on the right hand side $p(i^*q)$ by applying (\ref{star}) both inside and outside the brackets. Since $i^* = -i$ the right hand side simplifies to $-p(iq)$ and multiplying both sides with $-1$ we finally obtain $(pq^*)i = p(iq)$, which is (\ref{pq1}).

Finally to see that (\ref{pq3}) holds as well we first establish that (\ref{xi}) also holds with $qi$ in the role of $i$. 
Concretely, for any $x \in A$ we have
\begin{equation}\label{xqi}
x(qi) = \tau(x)(qi) + \im(x)(qi) = (qi)\tau(x) - (qi)\im(x) = (qi)x^*.
\end{equation}
Here the second equality uses (\ref{anticommutative}) and the fact that elements of $F1$ commute with everything and the third equality uses (\ref{taustar}).

Now we use this to attack (\ref{pq3}). Using the definition of the associator for the first and third equality and (\ref{xqi}) for the second we obtain:
\begin{equation}\label{ipqi1}
(ip)(qi) = i(p(qi)) + (i, p, qi) = i((qi)p^*) + (i, p, qi) = (i(qi))p^* - (i, qi, p^*) + (i, p, qi).
\end{equation}
As above we find that $-(i, qi, p^*) + (i, p, qi) = (i, 2\tau(p), qi) = 0$ so that (\ref{ipqi1}) reduces to
\begin{equation}\label{ipqi2}
(ip)(qi) = (i(qi))p^*.
\end{equation}
The left hand side of (\ref{ipqi2}) equals the left hand side of (\ref{pq1}). We look at the term $i(qi)$ within the outermost brackets of the right hand side of (\ref{ipqi2}). From (\ref{xi}) we see that it equals $i(iq^*)$ and since this term lives in the associative subalgebra generated by the elements $i$ and $q$ we can shift the brackets and conclude that it equals $i^2q^* = \gamma q^*$. The right hand side of (\ref{ipqi2}) then becomes $\gamma q^*p^*$ which needs no brackets as $\gamma \in F$. By (\ref{star}) this equals the right hand side of (\ref{pq3}).
\end{proof}

\subsection{The Cayley-Dickson algebras}\label{CayleyDickson}
We inductively define the class of Cayley-Dickson algebras. Each such algebra is a finite-dimensional $F$-algebra with a strong involution. 
\begin{definition}
The (unique) one-dimensional $F$-algebra $F$ is a Cayley-Dickson algebra where the strong involution is the identity.
An $n$-dimensional $F$-algebra $A$ for $n > 1$ is a Cayley-Dickson algebra if it is the Dickson Double of a Cayley-Dickson algebra $B$ of lower dimension.
\end{definition}

It follows that every Cayley-Dickson algebra has dimension equal to a power of 2. 

It also follows from Theorem \ref{quadratic} that all Cayley-Dickson algebras $A$ are quadratic algebras with a co-dimension 1 subspace $\im(A)$ of imaginary elements $u$ satisfying $u^2 \in F1$, $u^* = -u$. We note the similarity to equations (\ref{i2}) and (\ref{xistar}) implying that the `special' element $i$ used in the construction of $A$ from the smaller algebra $B$ is always contained in the space $\im(A)$. 

\begin{remark}\label{special} 
In fact the element $i$ and subalgebra $B$ are not \emph{that} special. Once we are presented with a Cayley-Dickson algebra $A$ of dimension $2^n \leq 8$, it is impossible to tell which of its many $2^{n-1}$-dimensional Cayley-Dickson subalgebras $B$ was used to construct $A$, in the sense that $A$ can be realized as a Dickson double of each of them. Various choices of the subalgebra $B$ will allow various elements $u \in \im(A)$ with $u^2 \neq 0$ to play  the role of the `special' element $i$. We will split the proof this remarkable fact into two parts. The proof in case that $A$ is a division algebra is presented here (Prop. \ref{doubling} below); the remaining cases will be dealt with in Theorem \ref{doubling2}. A different proof of the division algebra case in the special case that $\dim A = 4$ appears as part of the proof of Proposition \ref{main}.
\end{remark}

\begin{proposition}\label{doubling}
Let $A$ be a Cayley-Dickson division algebra of dimension $2^{n}$ for $n \in \{1, 2, 3\}$ and let $B$ be a $2^{n-1}$-dimensional subalgebra of $A$ closed under the action of the strong involution. Then there exists a $\gamma \in F^\times$ and $i \in \im(A)$, not contained in $B$ such that $A = B \oplus iB$ as a vector space and such that $(\ref{pq1}, \ref{pq2}, \ref{pq3})$ hold in $A$ for all $p, q \in B$. Consequently, $A$ is isomorphic to the Dickson double $D_\gamma(B)$.
\end{proposition}
\begin{proof}
Let $\{b_1, \ldots, b_{2^{n-1}}\}$ be a basis of $B$. Let the linear maps $R_{b_j}$ be defined as in Lemma \ref{division}. Since $\ker \tau$ has co-dimension 1 (Lemma \ref{codim1}), for each $j \in \{1, \ldots, 2^{n-1}\}$ the space $\ker (\tau \circ R_{b_j})$ has co-dimension at most 1. (In fact these co-dimensions are exactly 1 by Lemma \ref{division}.) It follows that the space $\bigcap_{j = 1}^{\dim B} \ker (\tau \circ R_{b_j})$ of elements $i$ such that $ib \in \im(A)$ for every $b \in B$ has co-dimension at most $2^{n-1}$ and hence in particular is non-zero. Pick any $i \neq 0$ from this space. We note that $i \in \im(A)$ and hence $i^2 = \gamma1$ for some $\gamma \in F$ where moreover $\gamma \neq 0$ since otherwise $i$ would be a divisor of zero.

Now since $A$ is a division algebra, $L_i$ is invertible (Lemma \ref{division}) and hence the space $iB$ is $2^{n-1}$-dimensional. We also have that $iB \cap B = \{0\}$ since if $ia = b$ for some $a, b \in B$ with $a \neq 0$ it follows that $i = ba^{-1} \in B$, and hence $i^{-1} = \frac{1}{\gamma} i \in B$. But this would imply that $1 = ii^{-1} \in \im(A)$ by our choice of $i$ -- a clear contradiction. We conclude that $B + iB$ is $2^n$-dimensional and hence equal to all of $A$. 

It remains to verify $i$ and $B$ satisfy (\ref{pq1}, \ref{pq2}, \ref{pq3}). For this we first note that $i$ satisfies (\ref{xistar}) for all $x \in B$ since, by our construction of $i$, $ix \in \im(A)$ for each $x \in B$. From (\ref{xistar}) we then conclude that (\ref{xi}) holds for all $x \in B$ as well and then deduce (\ref{pq1}, \ref{pq2}, \ref{pq3}) from (\ref{xi}) as in Lemma \ref{allfromxi2}. In this last step we exploit that, as $\dim(A) \leq 8$ we have that $A$ is alternative by Theorem \ref{Albert}.
\end{proof} 

\begin{definition}
A Cayley-Dickson algebra of dimension 8 is called a \emph{Cayley algebra} or an \emph{octonion algebra}. A Cayley-Dickson algebra of dimension 4 is called a \emph{quaternion algebra}.
\end{definition}

By Theorem \ref{Albert}, all Cayley algebras are alternative and all quaternion algebras are associative. In the literature on associative algebras a different, yet equivalent, definition of the term quaternion algebra is used, we'll come back to that issue in Section \ref{nilp}. 

By the same theorem, all two-dimensional Cayley-Dickson algebras are commutative. In particular, when they are division algebras, they are fields. And, as they are formed by attaching to $F$ a new squareroot (the element $i$ from the definition) of some $\alpha \in F$, they are quadratic field extensions of $F$. In particular:

\begin{lemma}\label{2dim}
Let $\alpha \in F^\times$. The two-dimensional Cayley-Dickson algebra $D_\alpha(F)$ is a quadratic field extension of $F$ if and only if $\alpha$ is a non-square in $F$ and contains divisors of zero if and only if $\alpha$ is a square in $F$. 
\end{lemma}
\begin{proof}
We only need to prove the `if' directions. In both cases we tacitly exploit that $D_\alpha(F)$ is commutative by Theorem \ref{Albert}. If $\alpha$ is a non-square the polynomial $x^2 - \alpha$ is irreducible over $F$ and attaching the root $i$ of this polynomial to $F$ yields a quadratic field extension by the standard argument. On the other hand when $\alpha = \beta^2$ for some $\beta \in F$ we find that $(\beta - i)(\beta + i)  = 0$ by the `strange product' formula.
\end{proof}

%dit verplaatsen naar sectie nilp of die misschien zelfs in tweeen splitsen en bovenstaande referentie aanpassen.
%It also follows from the same theorem that the four-dimensional Cayley-Dickson algebras are central (meaning that the only elements commuting and associating with all other elements are the elements of $F1$). We state here without proof that these four-dimensional central associative algebras are also simple. %, resulting in their being quaternion algebras. 
%Conversely every quaternion algebra (definined as a central simple associative algebra) is a Cayley-Dickson algebra. The proof of these facts can be found in any text on associative algebras, e.g. \cite{Pierce} or \cite{Vigneras}.
%

The following surprising fact forms the basis underlying all our considerations from Section \ref{split} onwards.

\begin{theorem}[Uniqueness theorem, \cite{Schafer}, p. 24-27]\label{splitunique} 
In each of the dimensions 2, 4, 8 there is up to isomorphism exactly one Cayley-Dickson algebra containing divisors of zero, called the \emph{split} Cayley Dickson algebra in that dimension.
\end{theorem}

The number of alternative Cayley-Dickson algebras \emph{not} containing divisors of zero (hence division algebras by Lemma \ref{division}) depends on the field $F$. Cayley-Dickson algebras of dimension 16 and greater (equivalently: non-alternative Cayley-Dickson algebras) will be of no concern of us, although we note that the 16-dimensional $\mathbb{R}$-algebra $D_{-1}(D_{-1}(D_{-1}(D_{-1}(\mathbb{R}))))$ (the sedenions) are interesting as an example of an algebra containing zero-divisors while at the same time every non-zero element has a multiplicative inverse (cf Lemma \ref{division}).

We collect some consequences of the uniqueness theorem.

\begin{corollary}\label{sqrt1}
An alternative Cayley-Dickson algebra is split if and only if it contains an imaginary square root of 1
\end{corollary}
\begin{proof}
For the `if' direction, let $e$ be a imaginary square-root of 1. Then $(1 + e) \neq 0, (1 - e) \neq 0, (1+e)(1-e) = 0$, showing the existence of zero-divisors. Conversely: $D_1(F)$ and hence $D_\beta(D_1(F))$ and $D_\alpha(D_\beta(D_1(F)))$ contain an imaginary square-root of 1 by construction and hence the `only if' direction follows from the uniqueness theorem.
\end{proof}

We can however be a bit more explicit.

\begin{proposition}\label{CD2}
The (by Theorem \ref{splitunique}) unique two-dimensional split Cayley-Dickson algebra is isomorphic to the algebra $F \oplus F$ of pairs of elements in $F$ with pointwise addition and multiplication. Here the $*$ operation is given by $(\alpha, \beta)^* = (\beta, \alpha)$
\end{proposition}
\begin{proof}
Taking uniqueness for granted we only need to show that this algebra is isomorphic to $D_1(F)$. Let $i$ be the element of that algebra used in the definition of the doubling process (so in this case $i^2 = 1$). Then $\{1, i\}$ is a basis of $D_1(F)$ and the isomorphism is given by $1 \mapsto (1, 1)$, $i \mapsto (1, -1)$.
\end{proof}
\begin{proposition}\label{Mat2F}
With $\tau, n, *$ as in Section \ref{nice}, the (by Theorem \ref{splitunique} unique) four-dimensional Cayley-Dickson algebra over $F$ containing divisors of zero is isomorphic to the matrix algebra $\Mat(2, F)$ with $1 = I$, $\tau(X) = \frac{1}{2} \Tr(X)$, $n(x) = \det(X)$ and $\begin{pmatrix}
a & b \\ c & d
\end{pmatrix}^* = \begin{pmatrix}
d & -b \\ -c & a
\end{pmatrix}$.
\end{proposition}
This relation to the matrix trace is also the reason that the operator $\tau$ is called $\tau$.  
\begin{proof} Taking uniqueness for granted, we only need to show an isomorphism between the matrix algebra and $D_1(D_{-1}(F))$. It is easy to verify that the subalgebra $\{\begin{pmatrix}
a & - b \\ b & a
\end{pmatrix} \colon a, b \in F\}$ of $\Mat(2, F)$ is isomorphic to $D_{-1}(F)$ with the involution being the restriction to this subalgebra of the the involution described in the proposition. Denoting this subalgebra by $B$ for the moment, it takes only a little more effort to verify that the algebra $\Mat(2, F)$ is isomorphic to the Dickson Double $D_1(B)$ of $B$ where the role of the element $i$ (the `new square-root of 1') is played by the matrix $i = \begin{pmatrix}
1 & 0 \\
0 & -1
\end{pmatrix}$.
\end{proof}

In a sense the split Cayley and quaternion algebras are the richest Cayley and quaternion algebras:

\begin{corollary}\label{rich}
Every isomorphism class of Cayley-Dickson algebras of dimension $2^{n-1}$ $(n = 1, 2, 3)$ is represented among the subalgebras of the split Cayley-Dickson algebra of dimension $2^n$.
\end{corollary}
\begin{proof}
Given a $2^{n-1}$-dimensional Cayley-Dickson algebra $A$ we can realize the split Cayley-Dickson algebra of dimension $2^n$ as $D_1(A)$.
\end{proof}

However, over certain important fields this richness is not too impressive, as the following two corollaries show.

\begin{corollary}
Every alternative Cayley-Dickson algebra over an algebraically closed field $F$ is split.
\end{corollary}
\begin{proof}
Let $u$ be a non-zero imaginary element in such an algebra. By (\ref{im}), $u^2 \in F$ and by algebraic closedness there exists an $\alpha \in F$ such that $\alpha^2 = u^2$. If $u^2  = 0$ we have that $u$ is a zero-divisor and we are done. Otherwise $(u/\alpha)$ is an imaginary square-root of 1 and we can apply Corollary \ref{sqrt1}.
\end{proof}

\begin{corollary}\label{splitfinite}
Over a finite field $F$ there is up to isomorphism only one quaternion algebra and only one Cayley algebra.
\end{corollary}
\begin{proof}
By Wedderburn's little theorem, every division algebra over $F$ is commutative. By Theorem \ref{Albert} this means that no 4-dimensional Cayley-Dickson algebra can be a division algebra and hence the last proposition implies that all 4-dimensional Cayley-Dickson algebras over $F$ are isomorphic to $\Mat(2, F)$. Moreover, since every Cayley algebra $C$ over $F$ is obtained by doubling a quaternion algebra $Q$ over $F$, which it contains as a subalgebra, any zero-divisor present in the quaternion algebra used to obtain $C$ will also be an element of $C$. As we just saw that the presence of zero-divisors in $Q$ is inevitable, the uniqueness theorem tells us that $C$ is unique up to isomorphism.
\end{proof}

To end this subsection on a happy note we recall that we already saw in the introduction an example of a Cayley algebra that \emph{is} a division algebra, the algebra $\OO \coloneqq D_{-1}(D_{-1}(D_{-1}(\mathbb{R})))$ over $F = \mathbb{R}$. The proof presented there that $\OO$ is a division algebra relied on the fact that the Euclidean norm $\|.\|$ on $\RR^8$ respects the octonion multiplication: $\|x\|^2\|y\|^2 = \|xy\|^2$. In the context of more general Cayley-Dickson algebras this is just equation (\ref{nmult}) combined with the fact that in the special case of $\OO$ the function $n$ equals the square of the Euclidean norm. This suggests a path forward to finding other fields over which a Cayley division algebra might exist.

\begin{lemma}\label{splitn0}
Let $A$ be a Cayley-Dickson algebra. Then $A$ contains a zero divisor if and only it contains a non-zero element $x$ such that $n(x) = 0$.
\end{lemma}
\begin{proof}
For the if-direction it is enough to recall that $xx^* = n(x)$ (equation (\ref{n})). For the only-if-direction let $a, b \in A$ such that $ab = 0$. We note that by (\ref{taun}) we have that $n(0) = 0$ so that the multiplicativity of $n$ (equation (\ref{nmult})) implies that $n(a) = 0$ or $n(b) = 0$. 
\end{proof}

Apparently the key to understanding the (non-)existence of zero divisors is understanding the quadratic form $n \colon A \to F$. We'll return to this in Section \ref{Serre}.

\subsection{Multiplication table and relation to the Fano plane}\label{OtoFano}
In Section \ref{introduction} we gave a multiplication table (Table \ref{multtabclassic}) of $\OO = D_{-1}(D_{-1}(D_{-1}(\RR)))$ with respect to the `standard' basis. We also indicated how, when viewing the seven imaginary elements among these standard basis elements as the points in a classical Fano plane, the lines in that Fano plane can be used to describe the multiplication. In this section we will provide the proof of those statements by providing a similar table, Table \ref{multtabgeneral}, for the general Cayley algebra $\mathcal{O} \coloneqq D_\gamma(D_\beta(D_\alpha(F)))$, where $F$ is any field and $\alpha, \beta, \gamma$ are elements of $F^\times$. A special case of interest (see also Section \ref{Serre}) occurs when $\alpha, \beta, \gamma$ are algebraically independent, transcedental elements w. r. t. a subfield $K$ of $F$, but this requirement is not necessary for the correctness of the table. In particular we see that when choosing $\alpha = \beta = \gamma = -1$ in Table \ref{multtabgeneral} we recover Table \ref{multtabclassic}, showing that Table \ref{multtabclassic} is indeed the correct multiplication table, not only of $\OO = D_{-1}(D_{-1}(D_{-1}(\RR)))$ but for the Cayley algebras $D_{-1}(D_{-1}(D_{-1}(F)))$ over any $F$.

Although Table \ref{multtabgeneral} is a bit less transparant than Table \ref{multtabclassic}, the underlying Fano plane structure is still visible when igorning the scalars. In particular: for $a, b, c \in \{e_0, \ldots, e_6\}$ we have that $abc \in F1$ (for both placements of the brackets) if and only if $\{a, b, c\}$ is a line in the standard Fano plane structure on the $e_i$. Since moreover for \emph{any} imaginary elements $a, b \in \im(\mathcal{O})$ with $a^2 \neq 0$ we have that $ab \in F1$ if and only if $b \in Fa$, we find that the product of two imaginary basis elements $e_i, e_j$ is always a scalar multiple of the third basis element $e_k$ lying on the line spanned by $e_i, e_j$ in the Fano plane structure on $\{e_0, \ldots, e_6\}$. 

The goal of this subsection is to prove these statements and compute the mentioned `scalar multiples', i.e. to verify the correctness of Table \ref{multtabgeneral}. In other words, our goal is to recover the points and lines of the Fano plane, together with the additional information in Table \ref{multtabgeneral}, from the Cayley algebra $\mathcal{O}$. The opposite direction, obtaining the Cayley algebra (and subsequent $q$-covering design) from the Fano plane together with some additonal information as in Theorem \ref{mainintro}, will be the subject of Section \ref{combinatorics}.

\begin{notation}
For the duration of this subsection let $\mathcal{O} = D_\gamma(D_\beta(D_\alpha(F)))$ and let $e_0, e_1, e_2 \in \im(\mathcal{O})$ denote the elements $i$ used in the first, second and third Dickson-doubling step respectively. Hence we have $e_0^2 = \alpha1$, $e_1^2 = \beta 1$ and $e_2^2 = \gamma 1$.  
\end{notation}
Altough the doubling process puts a clear `hierarchy' on the elements $e_0, e_1, e_2$, here we adopt a different perspective where the three elements $e_0, e_1, e_2$ are treated on equal footing, as \emph{generators} of $\mathcal{O}$; the remaining basis elements are then derived from them as products: $e_3 \coloneqq e_0e_1$, $e_4 \coloneqq e_1e_2$, $e_5 = e_0(e_1e_2)$ and $e_6 = e_0e_2$ respectively. 

What `allows' us to treat $e_0, e_1, e_2$ as of equal importance is the fact that both these elements and their product are imaginary: taking $x = e_0, i = e_2$ in (\ref{xi}) reads $e_0e_2 = -e_2e_0$ while taking $x = e_2, i = e_0$ reads $e_2e_0 = -e_0e_2$, which is clearly equivalent. The same holds for the pairs $(e_0, e_1)$ and $(e_1, e_2)$ and hence it follows that when presented with $\mathcal{O}$ there is no meaningful way to tell whether if it was created as $D_\gamma(D_\beta(D_\alpha(F)))$, as $D_\gamma(D_\alpha(D_\beta(F)))$, as $D_\beta(D_\gamma(D_\alpha(F)))$, as $D_\beta(D_\alpha(D_\gamma(F)))$, as $D_\alpha(D_\gamma(D_\beta(F)))$ or as $D_\alpha(D_\beta(D_\gamma(F)))$, even if the three elements $e_0, e_1, e_2$ used as $i$ in each doubling step are fixed and known.

\begin{lemma}\label{imbasis}
Let $e_3 = e_0e_1$, $e_4 = e_1e_2$, $e_6 = e_0e_2$ and $e_5 = -e_3e_2$ then each of $e_0 \ldots, e_6$ is imaginary and their squares are as given as in Table \ref{multtabgeneral}.
\end{lemma}
\begin{proof}
For the proof we return to the viewpoint of $\mathcal{O}$ as $D_\gamma(D_\beta(D_\alpha(F)))$. $e_0, e_1, e_2$ are imaginary with squares $\alpha, \beta, \gamma$ respectively by definition. $e_3 = e_0e_1$ appears in the four-dimensional subalgebra $D_\beta(D_\alpha(F))$. Viewing $D_\alpha(F)$ as $A$ in equation (\ref{xistar}) and $e_0$ as $x \in A$, $e_1$ as $i$, we find that $e_3^* = -e_3$ which by (\ref{im2}) means that $e_3 \in \im(\mathcal{O})$. For the remaining three products we use the same reasoning but with $A = D_\beta(D_\alpha(F))$ and $i = e_2$ to obtain imaginarity of $e_4, e_5, e_6$.

To prove the claim about squares we note that since we established that for $i < j \in \{0, 1, 2, 3\}$ we have that both $e_i$, $e_j$ and $e_ie_j$ are imaginary we obtain by alternativity and (\ref{anticommutative}) that $(e_ie_j)^2 = e_i(e_je_i)e_j =e_i(-e_ie_j)e_j = -e_i^2e_j^2$. Note that the right hand side is just an ordinary product of scalars in $F$.
\end{proof}
From the lemma and (\ref{anticommutative}) it also follows that $e_ie_j = - e_je_i$ whenever $i \neq j$. The usefulness of Table \ref{multtabgeneral} hinges on the following:
\begin{lemma}
The set $\{1, e_0, \ldots, e_6\}$ is a basis of $\mathcal{O}$.
\end{lemma}
\begin{proof}
It suffices to show that the set is linearly indepedent and for that it suffices to show linear independence of the same set but with $e_5$ replaced by $-e_5$. Viewing $\mathcal{O}$ as $D_\gamma(D_\beta(D_\alpha(F)))$ and using the shorthand $A = D_\beta(D_\alpha(F))$ again, we obtain a vectorspace decomposition $\mathcal{O} = A \oplus Ae_2$ as in Definition \ref{DicksonDouble}. By construction $\{1, e_0, e_1, e_3\} \subset A$ and $\{e_2 = 1e_2, e_4 = e_1e_2, -e_5 = e_3e_2, e_6 = e_0e_2\} \subset Ae_2$. The decomposition $\mathcal{O} = A \oplus Ae_2$ reduces linear independence of the set $\{e_0, \ldots e_6\}$ to linear independence of the two subsets $\{1, e_0, e_1, e_3\}$ and $\{e_2, e_4, -e_5, e_6\}$ but we can say something more. Since the linear map $R_{e_2}$ (right multiplication by $e_2$) is invertible (its inverse being $\frac{1}{\gamma}R_{e_2}$) we have that the set $\{e_2 = 1e_2, e_6 = e_0e_2, e_4 = e_1e_2, -e_5 = e_3e_2\}$ is linearly independent if and only if the set $\{1, e_0, e_1, e_3\}$ is. Using the decomposition $A = D_\alpha(F) \oplus (D_\alpha(F))e_1$ we can then use similar reasoning to show that $\{1, e_0, \ldots, e_6\}$ is linear independent if and only the two element set $\{1, e_0\}$ is, and this is the case by definition.
\end{proof}

Having identified the vertex set of our Fano plane we move on to recognizing the lines.

\begin{lemma}\label{linescalars}
Let $a, b \in \im(\mathcal{O})$ be two imaginary elements satisfying $ab \in \im(\mathcal{O})$ and let $\ell = \{a, b, ab\}$. Then for every two distinct elements $p, q \in \ell$ there is scalar $\lambda_{p, q} \in F$ such that $pq = \lambda_{p,q} r$ and $qp = -\lambda_{p,q}r$ where $r$ is the third element in the set (`line') $\ell$.
\end{lemma}
\begin{proof}
The relation between $pq$ and $qp$ follows from (\ref{anticommutative}) so we focus on estabishing the $pq = \lambda r$ for some $\lambda \in F$. The case where $p = a$ and $q = b$ or vice versa is trivial, so we may assume that $q = ab$. For $p = a$ we can use alternativity to move the brackets and conclude $a(ab) = (aa)b = \lambda b$ with $\lambda = a^2 \in F$ as desired. For $p = b$ we have $b(ab) = -b(ba) = -(bb)a = \lambda a$ with $\lambda = -b^2 \in F$.
\end{proof}
The lemma gives a different way of thinking about the Fano-lines. When we view every tripple $\{a, b, c\} \subset \im(\mathcal{O})$ as a `line' whenever it satisfies $ab = \lambda c$ for some $\lambda \in F^\times$ and some labeling of the elements of the triple as $a$, $b$ and $c$, then the lemma states that this definition does in fact not depend on the choice of labeling and that for any pair $\{p, q\}$ of imaginary elements we have that all $r$ making $\{p, q, r\}$ into a line are scalar multiples of each other. (Where of course for some choices of $p, q$ no such $r$ exist.) 

Our goal is to establish that (1) every two basis elements $e_i, e_j$ lie on such a line, i.e. have a product that is imaginary, and (2) that the third point on that line (which by the last lemma is unique up to scalar multiplication) can be taken to be the basis element indicated by Table \ref{multtabgeneral}. For four of the seven lines this has already been achieved by the above.

For example: our definition of $e_3$ as the product $e_0e_1$ establishes that the Fano-line through $e_0$ and $e_1$ passes through $e_3$, which by the lemma means concretely that $e_0e_1 \in Fe_3$, $e_3e_0 \in Fe_1$ and $e_1e_3 \in Fe_0$. Similarly we obtain the `lines' $\{e_1, e_2, e_4\}$, $\{e_2, e_3, e_5\}$ and $\{e_0, e_2, e_6\}$. What remains to be done to obtain the full Fano structure is verifying that the `line' through $e_3$ and $e_4$ contains $e_6$ (that is: verifying that $e_3e_4 \in F^\times e_6$) and that the line through $e_0$ and $e_4$ as well as the line through $e_1$ and $e_6$ both pass through $e_5$. This can be done by explicit computation:
\begin{eqnarray*}
e_4e_6 = (e_1e_2)(e_0e_2) = (e_2e_1^*)(e_0e_2) = \gamma(e_1^*e_0)^* = \gamma e_0^*e_1 = -\gamma e_0e_1 = -\gamma e_3 \\
e_0e_4 = e_0(e_1e_2) = e_0(e_2e_1^*) = e_2(e_0^*e_1^*) = e_2(e_0e_1) = e_2e_3 = e_5 \\  
e_6e_1 = (e_0e_2)e_1 = (e_0e_1^*)e_2 = e_2(e_0e_1^*)^* = -e_2(e_0e_1)^* = -e_2e_3^* = e_2e_3 = e_5
\end{eqnarray*}
Here the first equation uses (\ref{xi}) and (\ref{pq3}), both with $e_2$ in the role of $i$, followed by (\ref{star}) and (\ref{im2}). The second equation uses (\ref{xi}) and (\ref{pq1}), both with $e_2$ in the role of $i$, followed by (\ref{im2}). The third equation uses (\ref{pq2}) and (\ref{xi}), both with $e_2$ in the role of $i$, followed by (\ref{im2}). Of course all applications of (\ref{im2}) rely on Lemma \ref{imbasis} above. 

We have reached the point that for each two elements from $\{e_0, \ldots, e_6\}$ we know the unique Fano line $\ell$ to which they belong, and for each line $\ell$ we have computed the scalar $\lambda_{a, b} \in F$ in the equality $ab = \lambda_{ab}c$ for at least one of the six ways to asign the labels $a, b, c$ to the elements of $\ell$. This means that we can compute the remaining multiplications of elements on $\ell$ by the rules $ba = -ab = -\lambda_{ab}c$, $ac = \lambda_{a,b}^{-1}a^2b$, $ca = -ac$, $cb = \lambda_{ab}^{-1}ab^2$ and $bc = -cb$. This completes the verification of Table \ref{multtabgeneral}.

\begin{table}\caption{Multiplication table of the general Cayley-Algebra $D_\gamma(D_\beta(D_\alpha(F)))$ with respect to the basis $e_0, \ldots, e_6$ described in Section \ref{OtoFano} \label{multtabgeneral}}
\begin{tabular}{l|rrrrrrrr}
      &  $1$ &          $e_0$ &        $e_1$ &         $e_2$ &              $e_3$ &              $e_4$ &               $e_5$ & $e_6$  \\
      \hline
  $1$ &  $1$ &          $e_0$ &        $e_1$ &         $e_2$ &              $e_3$ &              $e_4$ &               $e_5$ & $e_6$ \\
$e_0$ & $e_0$ &      $\alpha$ &        $e_3$ &         $e_6$ &       $\alpha e_1$ &              $e_5$ &        $\alpha e_4$ & $\alpha e_2$ \\
$e_1$ & $e_1$ &        $-e_3$ &      $\beta$ &         $e_4$ &       $-\beta e_0$ &        $\beta e_2$ &        $-\beta e_6$ & $-e_5$ \\
$e_2$ & $e_2$ &        $-e_6$ &       $-e_4$ &      $\gamma$ &              $e_5$ &      $-\gamma e_1$ &        $\gamma e_3$ & $-\gamma e_0$ \\
$e_3$ & $e_3$ & $-\alpha e_1$ &  $\beta e_0$ &        $-e_5$ &     $-\alpha\beta$ &       $-\beta e_6$ &   $\alpha\beta e_2$ & $\alpha e_4$ \\
$e_4$ & $e_4$ &        $-e_5$ & $-\beta e_2$ &  $\gamma e_1$ &        $\beta e_6$ &     $-\beta\gamma$ &   $\beta\gamma e_0$ & $-\gamma e_3$ \\
$e_5$ & $e_5$ & $-\alpha e_4$ &  $\beta e_6$ & $-\gamma e_3$ & $-\alpha\beta e_2$ & $-\beta\gamma e_0$ & $\alpha\beta\gamma$ & $\alpha\gamma e_1$ \\
$e_6$ & $e_6$ & $-\alpha e_2$ &        $e_5$ &  $\gamma e_0$ &      $-\alpha e_4$ &       $\gamma e_3$ & $-\alpha\gamma e_1$ & $-\alpha\gamma$

\end{tabular}

\end{table}

\subsection{Which fields allow Cayley division algebras?}\label{Serre}
We are interested in the question in the title of this subsection because it are the Cayley division algebra that give rise, in Section \ref{qfano}, to $q$-Fano planes. It was pointed out to me by Matthias Wendt \cite{WendtMO} that an answer in terms of Galois cohomology can be found in Serre's paper \cite{Serre}: combining Theorem 9 of that paper with Theorem 5.7 of Merkurjev-Suslin \cite{MerkurSuslin} as indicated in Section 8.2 of \cite{Serre} we find:

\begin{theorem}\label{Serre2}
The field $F$ of characteristic $\neq 2$ allows Cayley division algebras if and only if the Galois cohomology group $H^3(F, \ZZ/2\ZZ)$ is non-trivial.
\end{theorem}

The definition and other applications of the groups $H^i(F, \ZZ/2\ZZ)$ are beyond the scope of the current paper. Instead we will discuss a class of examples where we can understand quite concretely what is going on.

Since the elements $e_i$ in Table \ref{multtabgeneral} are imaginary we have for a generic element $x = \eta + \xi_0e_0 + \ldots + \xi_6 e_6 \in D_\gamma(D_\beta(D_\alpha(F)))$ that $x^* = \eta 1 - \xi_0e_0 - \ldots - \xi_6 e_6 \in D_\gamma(D_\beta(D_\alpha(F)))$ (equation (\ref{im2})) and hence, by (\ref{taun}) and Table \ref{multtabgeneral} we find that 

\begin{equation}\label{generaln}
n(x) = \eta^2 - \xi_0^2 \alpha - \xi_1^2 \beta - \xi_2^2 \gamma + \xi_3^2 \alpha \beta + \xi_4^2 \beta\gamma + \xi_6^2 \gamma\alpha - \xi_5^2 \alpha \beta \gamma.
\end{equation}

We are interested in finding fields $F$ and elements $\alpha, \beta, \gamma \in F$ such that $D_\gamma(D_\beta(D_\alpha(F)))$ is a division algebra. From (\ref{generaln}) and Lemma \ref{splitn0} we see (with some relabeling) that this is the case exactly when the equation
\begin{equation}\label{stroth1}
\lambda_0^2 - \lambda_1^2\alpha - \lambda_2^2\beta - \lambda_3^2\gamma + \lambda_4^2\alpha\beta + \lambda_5^2\beta\gamma + \lambda_6^2\gamma\alpha - \lambda_7^2\alpha\beta\gamma = 0
\end{equation}
has only the trivial solution $\lambda_0 = \ldots = \lambda_7 = 0$.

From this observation we recognize one class of examples: 

\begin{corollary}\label{Fxyz}
Suppose that $F$ contains a subfield $K$ and elements $\alpha, \beta, \gamma$ that are algebraically independent over $K$ such that $F$ is isomorphic to the function field $F(\alpha, \beta, \gamma)$. Then $D_\gamma(D_\beta(D_\alpha(F)))$ is a division algebra.
\end{corollary}
Our proof is adapted from the proof of Bemerkung 5.3.1 in Kristin Stroth's PhD thesis \cite{stroth}. Chapter 5 of that thesis further provides an interesting geometric perspective on Cayley algebras of this type which will not be discussed here.
\begin{proof}
Let $\lambda_0, \ldots, \lambda_7 \in K(\alpha, \beta, \gamma)$ be such that $(\ref{stroth1})$ holds.

Multiplying both sides of the equation with the lowest common multiple of the denominators of the $\lambda_i^2$ we may assume that each $\lambda_i$ is an element of the polynomial ring $K[\alpha, \beta, \gamma]$.

Our goal is to show that (\ref{stroth1}) only has the zero solution over $K[\alpha, \beta, \gamma]$. We'll first show that the same holds for the simpler equation $\nu_0^2 - \nu_1^2 \alpha = 0$.

\begin{claim1}
Let $\nu_0, \nu_1 \in K[\alpha, \beta, \gamma]$ be such that $\nu_0^2 - \nu_1^2 \alpha = 0$. Then $\nu_0 = \nu_1 = 0$.
\end{claim1} 
\begin{proof}
Rewriting the equation as $\nu_0^2 = \nu_1^2 \alpha$ and interpreting both sides as polynomials in $\alpha$ with coefficients in $K[\beta, \gamma]$ we see that the left hand side has even degree, while the right hand side has odd degree (adopting the convention that the zero polynomial has every degree). Hence both sides must be zero.
\end{proof}
We want to use the same reasoning to show that the `intermediate level' equation $\mu_0 - \mu_1 \alpha - \mu_2 \beta + \mu_3 \alpha \beta = 0$ also only has the zero solution, but we need an extra step:
\begin{claim2}
Let $\mu_0, \mu_1 \in K[\alpha, \beta, \gamma]$. Then viewed as a polynomial in $\beta$, $\xi \coloneqq \mu_0^2 - \mu_1^2\alpha$ has even degree. 
\end{claim2}
\begin{proof}
Let $m_0, m_1$ be the degrees of $\mu_0, \mu_1$ respectively when viewed as polynomials in $\beta$ and let $m = \max(m_0, m_1)$. When $m_0 \neq m_1$ we have that the degree of $\xi = \mu_0^2 - \mu_1^2 \alpha$ equals $2m$ and we are done. When $m_0 = m_1 = m$ define $\nu_0, \nu_1$ as the leading (i.e. degree $m$) terms of $\mu_0, \mu_1$ respectively when viewed as a polynomial in $\beta$. The degree $2m$ term in $\xi$ equals $\nu_1^2 - \nu_2^2 \alpha$. Since $\xi$ clearly contains no terms of degree higher than $2m$ we see that the only scenario in which the degree of $\xi$ does not equal $2m$ is when its degree $2m$ term $\nu_0^2 - \nu_1^2\alpha$ equals zero. But by Claim 1 this only happens when $\nu_0 = \nu_1 = 0$, which, as the $\nu_i$ were defined as the leading terms of the $\mu_i$ implies that the $\mu_i$ are zero. Hence we see that either the degree of $\xi$ is $2m$ or $\xi = 0$, in which case we consider the degree of $\xi$ to be even as well.
\end{proof}
With this result under our belt we can lift the proof of Claim 1 one level higher:
\begin{claim3}
Let $\mu_0, \mu_1, \mu_2, \mu_3 \in K[\alpha,  \beta, \gamma]$ be such that $\mu_0 - \mu_1 \alpha - \mu_2 \beta + \mu_3 \alpha \beta = 0$. Then $\mu_0 = \mu_1 = \mu_2 = \mu_3 = 0$.
\end{claim3}
\begin{proof}
We can rewrite the equation as 
$$(\mu_0^2 - \mu_1^2\alpha) = (\mu_2^2 - \mu_3^2\alpha)\beta.$$
From Claim 2 we see that the left hand side has even degree when viewed as a polynomial in $\beta$ while the right hand side has odd degree. It follows that both sides of the equations equal zero, so that $\mu_0^2 - \mu_1^2 \alpha = 0$ and $\mu_2^2 - \mu_3^2 \alpha = 0$. But Claim 1 then implies that $\mu_0 = \mu_1 = 0$ and that $\mu_2 = \mu_3 = 0$.
\end{proof}
From this point on it is clear how we will proceed. Similar to how we derived Claim 2 from Claim 1 we derive from Claim 3 that expressions of the form
$\mu_0^2 - \mu_1^2 \alpha - \mu_2^2 \beta + \mu_3^2 \alpha \beta$ (with $\mu_0, \mu_1, \mu_2, \mu_3 \in K[\alpha, \beta, \gamma]$) have even degree when viewed as polynomials in $\gamma$.

We then use this to conclude that the left and right hand sides of the following equivalent reformulation of (\ref{stroth1}) have different degrees as polynomials in $\gamma$ and hence are both zero:

\begin{equation}\label{stroth2}
(\lambda_0^2 - \lambda_1^2\alpha - \lambda_2^2\beta + \lambda_4^2 \alpha\beta) = (\lambda_3^2 - \lambda_5^2\beta - \lambda_6^2\alpha + \lambda_7^2\alpha\beta)\gamma.
\end{equation}

In summary we find that

\begin{eqnarray}
\lambda_0^2 - \lambda_1^2\alpha - \lambda_2^2\beta + \lambda_4^2 \alpha\beta &= 0 \label{stroth3} \\
\lambda_3^2 - \lambda_5^2\beta - \lambda_6^2\alpha + \lambda_7^2\alpha\beta &= 0 \label{stroth4}
\end{eqnarray}
and Claim 3 then tells us that $\lamba_0 = \lambda_1 = \lambda_2 = \lambda_4 = 0$ and that $\lambda_3 = \lambda_5 = \lambda_6 = \lambda_7 = 0$, as we wanted to show.
\end{proof}

\begin{remark}
The proof that over the fields $K(\alpha, \beta, \gamma)$ equation (\ref{stroth1}) has no non-zero solutions relies on the \emph{multiplicative} properties of squares, notably that the product of two squares is a square and the product of a square and a non-square is a non-square. At the other extreme we find the subfields of $\RR$ where we can get to the same conclusion by looking at the \emph{additive} properties of squares over such field. Since in a subfield of $\mathbb{R}$ a sum of squares is always non-negative, we find that over such field equation (\ref{stroth1}) has no non-zero solutions whenever $\alpha < 0, \beta < 0, \gamma < 0$. 
\end{remark}

It is tempting to see if we can construct further examples that use some mixture of additive and multiplicative properties e.g. by somewhat loosening the condition that $\alpha, \beta, \gamma$ are algebraically independent over $K$ in Corollary \ref{Fxyz} while being careful about what it means to be square in the resulting field. We won't pursue that direction here, however.

\subsection{Some structure theory}\label{nilp}
So far we have focussed on the similarities rather than differences between split and divsion Cayley-Dickson algebras. However, since the method explored in Section \ref{qfano} that constructs $q$-Fano planes from Cayley division algebras fails to do so in the split case (for reasons explored in depth in Section \ref{split}) there must also be considerable differences between the two cases. Informally speaking the most important difference (to us) is that, unlike their division algebra counterparts, the split Cayley-Dickson algebras contain \emph{subalgebras} that are quite different from the Cayley-Dickson algebras of either type. The prototypical example of such a subalgebra is the algebra $U_3$ of upper triangular 2-by-2-matrices sitting inside the split quaternion algebra $M_4$ of all two-by-two matrices. Being three-dimensional this algebra cannot be a Cayley-Dickson algebra, even if it is closed under the star operator defined in Example \ref{Mat2F} and hence is a quadratic algebra with a strong involution. In this subsection we recall some notions from the structure theory of alternative algebras that enable us to articulate the differences between this algebra and the Cayley-Dickson algebras.

\begin{definition}
A two-sided \emph{ideal} in an  alternative algebra $A$ is a linear subspace $I$ such that $ai \in I$ and $ia \in I$ for every $a \in A, i \in I$. An algebra $A$ is called \emph{simple} if its only two sided ideals are $\{0\}$ and $A$ itself.
\end{definition}
The definition of ideal ensures that when $I$ is an ideal in $A$ then the multiplication in $A$ descends to a well-defined multiplication on the quotient vector space $A/I$. When $A$ is alternative or even associative then so is the quotient algebra $A/I$. It is easy to see that all division algebras are simple. Perhaps more surprising is that the split 4- and 8-dimensional Cayley-Dickson algebras are simple as well. (We'll come back to that at the end of this section). The algebra $U_3$ of upper triangular matrices is \emph{not} simple as its subspace consisting of strict upper triangular matrices is an ideal. The elements of this ideal also serve as an example to the following concept, that will be important in the sequel:

\begin{definition}
An element $x$ of an alternative algebra $A$ is called \emph{nilpotent} if there exists a natural number $n$ such that $x^n = 0$.
\end{definition}

\begin{lemma}\label{nilpotent}
Let $x \neq 0$ be an element of an alternative Cayley-Dickson algebra $O$. Then the following are equivalent:
\begin{enumerate}
\item $x$ is nilpotent
\item $x^2 = 0$
\item $x$ is a zero-divisor and $x \in \im(O)$
\end{enumerate}
\end{lemma}
\begin{proof}
$(1) \Rightarrow (2).$ Let $x$ be nilpotent and let $k$ be the smallest number such that $x^k = 0$. Since clearly $0$ is the only element for which $k = 1$, and clearly $0^2 = 0$ we will assume in the sequel that $k \geq 2$. From multiplicativity of the function $n$ (equation \ref{nmult}) we find that $n(x)^k = 0 \in F$ and hence $n(x) = 0$. Multiplying both sides of (\ref{taun}) by $x^{k-2}$ we find that $2\tau(x)x^{k-1} = 0$. By minimality of $k$ this implies that $\tau(x) = 0$. Plugging this back into the original equation (\ref{taun}) we find that $x^2 = 0$.

$(2) \Rightarrow (3).$ Assume $x^2 = 0$. It is clear that $x$ is a divisor of zero. The fact that $x \in \im(O)$ follows from (\ref{im}).

$(3) \Rightarrow (1)$ Let $x$ be an imaginary divisor of zero. By the latter property there exist a $y$ such that $xy = 0$ or $yx = 0$. In the latter case we have that $0^* = (yx)^* = x^*y^* = -xy^* = x(-y^*)$, so after replacing $y$ with $-y^*$ if needed we see that there exist non-zero $y$ such that $xy = 0$. Now by alternativity we can move the brackets in the following expression to obtain $0 = x(xy) = (x^2)y$. By $(\ref{im})$, $x^2 \in F1$ so the right hand side is a scalar multiple of the non-zero element $y$ and hence we find that $x^2 = 0$ proving (2) and hence (1).
\end{proof}

\begin{definition}\label{Jacobson}
An element $x$ in an alternative algebra $A$ is called \emph{strongly nilpotent} if $xy$ is nilpotent for all $y \in A$ or, equivalently, if $yx$ is nilpotent for all $y \in A$. The set of all strongly nilpotent elements is called the \emph{Jacobson radical} of $A$ and will be denoted $J(A)$. 
\end{definition}

When $A$ is associative, it is clear that $J(A)$ is a two-sided ideal in $A$. In fact this is already true, though highly non-obviously so, when $A$ is merely alternative. A proof of this fact (originally due to Zorn \cite{Zorn}) can be found in Chapter 3 of Schafer's book \cite{Schafer} (though unfortunately only in the 1966 edition and not in the 1961 edition that is freely available online), as can the proof of the following two theorems:

\begin{theorem}\label{semisimple}
Let $A$ be an alternative algebra. The following are equivalent:
\begin{enumerate}
\item $J(A) = \{0\}$
\item $A$ is isomorphic as an algebra to a direct sum of simple alternative algebras.
\end{enumerate}
\end{theorem}
The algebra direct sum $A_1 \oplus \ldots \oplus A_n$ of algebras $A_i$ is the vector space direct sum $A_1 \oplus \ldots \oplus A_n$ equiped with the \emph{pointwise} multiplication. In other words: multiplication of two elements in the same summand $A_i$ is just the ordinary multiplication in the algebra $A_i$ while the multiplication of elements from different summands equals zero. This multiplication defined on elements of the summands is then extended linearly to all of $A$. It is clear that this construction preserves alternativity and associativity.

\begin{definition}
An algebra satisfying the equivalent conditions of Theorem \ref{semisimple} is called \emph{semi-simple}.
\end{definition}

An example of a semi-simple algebra which is not simple is the two-dimensional split Cayley Dickson algebra which we met in Proposition \ref{CD2}. This algebra is isomorphic to the algebra $D_2$ of 2-by-2 diagonal matrices which sits as a subalgebra inside the algebra $U_3$ introduced at the beginning of this section. Since it is not hard to show that $J(U_3)$ equals the ideal of strict upper triangular matrices we find that $U_3$ decomposes (as a vector space, not as an algebra) as a sum of subalgebras $U_3 = D_2 \oplus J(U_3)$. It turns out that this is part of a more general pattern:

\begin{theorem}\label{SJ} %Deze blijkt niet in mijn masterscriptie te staan en ook niet in 1961-editie van Schafer, dus moeten even een 1966-schafer zien te bemachtigen om te zien of het daar in staat en anders weglaten...
Let $A$ be an alternative algebra with Jacobson radical $J(A)$. Then $A$ contains a semi-simple subalgebra $S$ such that as a vectorspace $A = S \oplus J(A)$ and such that the resulting vectorspace isomorphism $S \isom A/J(A)$ is in fact an isomorphism of algebras.
\end{theorem}
The fact $A/J(A)$ even \emph{is} an algebra follows from Zorn's result quoted above. The fact that the resulting quotient algebra is semi-simple is then a straightforward consequence of the definitions. The really surprising part is that $A$ already had this semi-simple quotient within it all along.
\begin{notation}
The subalgebra $S$ of Theorem \ref{SJ} will be called the \emph{semi-simple part} of $A$.
\end{notation}

In the cases we are interested in we can give an alternative characterisation of the Jacobson radical:

\begin{lemma}\label{stronglyimaginary}
Let $A$ be an alternative algebra with a strong involution (e. g. any subalgebra of a Cayley algebra) and let $x \in A$. Then $x$ is strongly nilpotent if and only if it is \emph{strongly imaginary}, that is: if and only if $xy \in \im(A)$ for every $y \in A$. (And, equivalently, $yx \in \im(A)$ for every $y \in A$.)
\end{lemma}
\begin{proof}
The `only if' direction is immediate from implication $1 \Rightarrow 3$ of Lemma \ref{nilpotent}. For the `if' direction let $x$ be strongly imaginary. Let $y \in A$. We first prove the parenthetical statement that $yx \in \im(A)$. Since $x1 \in \im(A)$, we see that strongly imaginary elements are imaginary themselves. Applying (\ref{im2}) to the imaginary elements $x$ and $xy^*$ we find that $(yx)^* = -yx$: $(yx)^* = x^*y^* = -xy^*$ $= (xy^*)^* = yx^* = -yx$. This implies that $yx$ is imaginary by (\ref{im2}).

Next we establish that $x$ itself is nilpotent. Since $x$ is strongly imaginary, $x^2 \in \im(A)$. But since $x$ is `ordinarily' imaginary, $x^2 \in F1$. It follows that $x^2 \in F1 \cap \im(A) = \{0\}$ and hence $x^2 = 0$.

Finally to see that $xy$ is nilpotent as well, we compute $(xy)^2 = -(xy)^*(xy) = -y^*x^*xy = y^*xxy = y^*0y = 0$.
\end{proof}

Using this lemma it is easy to show inductively:
\begin{theorem}\label{CDsemisimple}
Let $A$ be a Cayley-Dickson algebra of dimension 1, 2, 4 or 8, then $A$ is semi-simple.
\end{theorem}

We already saw that the same is not true for all \emph{subalgebras} of the split quaternion and Cayley algebras.

We conclude this section by proving of the following variant of Theorem \ref{CDsemisimple}:

\begin{theorem}\label{CDsimple}
Let $A$ be a Cayley-Dickson algebra of dimension 1, 4 or 8, then $A$ is simple.
\end{theorem}
\begin{definition}
The \emph{center} of an algebra $A$ is the vector space of all elements $c \in A$ that commute and associate with every element of $A$. Clearly $F1$ is contained in the center of any algebra $A$. When $F1$ equals the entire center of $A$ then $A$ is called \emph{central}.
\end{definition}
Since the center of a direct sum of algebras contains the direct sum of the centers of the summands, it follows that a semi-simple algebra can only be central when it consists of a single summand and thus is simple. We conclude that Theorem \ref{CDsimple} follows from Theorem \ref{CDsemisimple} together with the following proposition:
\begin{proposition}
All Cayley-Dickson algebras (not necessarily alternative) of dimension $\neq 2$ are central.
\end{proposition}
\begin{proof}
For the Cayley-Dickson algebra $F$ the statement is clearly true so we can restict our attention to algebras of the form $D_\gamma(A)$, the Dickson double of a Cayley-Dickson algebra $A$. Let $c$ be in the center of $D_\gamma(A)$. We can write $c = a + ib$ with $a, b \in A$ and $i$ as in $(\ref{pq1} - \ref{xistar})$. Expanding both sides of the equality $ci = ic$ into the standard form $r + is$ using $(\ref{pq1} - \ref{pq3})$ and comparing terms along the vectorspace decomposition $D_\gamma(A) = A \oplus iA$, we find that $a = a^*$ and $b = b^*$. It follows by Corollary \ref{strong} that $a \in F1$ and $b \in F1$. 

Since $c$ and $a$ are both in the center of $D_\gamma(A)$ so is $c - a = ib$. We distinguish two cases: $b = 0$ and $b \neq 0$. First suppose that $b \neq 0$. Since $b \in F1$ we conclude that $b$ is invertible and that, since $ib$ is in the center of $D_\gamma(A)$, so is $i$. This means in particular that for any $x \in A$ we have that $x = x^*$ by (\ref{xi}). By Corollary \ref{strong} this implies that $A = F1$ and hence that $D_\gamma(A)$ is two-dimensional. 

Conversely when $D_\gamma(A)$ is not two-dimensional we conclude that $b = 0$. But this implies that $c = a \in F1$. Since $c$ was an arbitrary element of the center of $D_\gamma(A)$, we conclude that $D_\gamma(A)$ is central.
\end{proof}

\begin{remark}
It is known that the dimension of any central simple \emph{associative} algebra is a square. Hence the smallest examples beyond $F$ are 4-dimensional and in the literature on associative algebras these are called \emph{quaternion algebras}. We just proved that a quaternion algebra in our sense, i.e. a 4-dimensional Cayley-Dickson algebra, is indeed central and simple. The converse also holds: any 4-dimensional central simple algebra $Q$ contains, as a subalgebra, a quadratic field extension $C$ of $F$ such that $Q$ is isomorphic to a Dickson double of $C$. (The proof of this latter fact can be found in any text on associative algebras, e.g. \cite{Pierce}, \cite{Vigneras}.) Thus the two notions of quaternion algebra are equivalent. 
\end{remark}

%\section{$q$-Fano planes from Cayley Division algebras}\label{qfano}
\section{$q$-Fano planes from Cayley Division algebras}\label{qfano}
In this section we will deduce the existence of a $2$-$(7, 3, 1)_F$-subspace design from the existence of a Cayley division algebra over $F$. As seen in Section \ref{CayleyDickson} this latter existence holds only for certain infinite non-algebraically closed fields. The first three results of this section (Lemmas \ref{H} and \ref{3of4} and Proposition \ref{main}) do hold over general fields of characteristic unequal to 2.

\begin{lemma}\label{H}
Let $H$ be a unital subalgebra of a Cayley algebra $O$ over $F$. Then $H$ is closed under $*$ and $\im(H) = H \cap \im(O)$.
\end{lemma}
\begin{proof}
$O$ is a quadratic algebra by Corollary \ref{strong} and Theorem \ref{quadratic}. It is immediate from the definition of quadratic algebras that subalgebras of quadratic algebras are quadratic again, so that $\im(H)$ is well defined. The statement that $\im(H) = H \cap \im(O)$ is immediate from characterization (\ref{im}) of the former space. Let $x \in H$. As in Section \ref{nice}, we write $\tau$ and $\im$ respectively for the projection operators along the decomposition $O = F1 \oplus \im(O)$. In particular for $x \in H$ there exist $\tau(x) \in F1, \im(x) \in \im(O)$ such that $x = \tau(x) + \im(x)$. Since $1 \in H$ by assumption and $x \in H$ by definition we have that $\im(x) = x - \tau(x) \in H$ and hence that $x^* = \tau(x) - \im(x) \in H$.
\end{proof}

\begin{lemma}\label{3of4}
Let $u, v$ be linearly independent imaginary elements of a Cayley algebra $O$ over $F$. Then the algebra $\langle u, v \rangle$ generated by $u$ and $v$ is equal as a vector space to the space  $\spam (\{1, u, v, uv\})$ and hence has dimension either 3 or 4.
\end{lemma}
\begin{proof}
It is clear that $\spam(1, u, v, uv)$ is contained in $\langle u, v \rangle$, which in turn is associative by alternativity of $O$. To get the converse inclusion it suffices to show that the linear space $\spam(1, u, v, uv)$ is closed under the multiplication. By associativity and the fact that $u^2, v^2 \in F1$ by (\ref{im}), we don't have to worry about the products $u^2, v^2, u^2v$ and $uv^2$ and the definition of quadratic algebra tells us that $(uv)^2 \in \spam(1, uv)$. Thus it only remains to show that the products $vu$, $uvu$ and $vuv$ are all in $\spam(1, u, v, uv)$.

By (\ref{star}) and (\ref{im2}) we find that $(uv)^* = v^*u^* = vu$ and hence $(uv + vu)^* = (uv + vu)$. The fact that $*$ is a strong involution by Corollary \ref{strong} then implies that $uv + vu \in F1$ and hence that $vu \in \spam(1, uv) \subseteq \spam(1, u, v, uv)$. It follows that $uvu \in \spam(u, u^2v) = \spam(u, v) \subset \spam(1, u, v, uv)$ and $vuv \in \spam(v, uv^2) = \spam(v, u) \subset \spam(1, u, v, uv)$ and we conclude that the latter space is closed under multiplication.

This proves that $\langle u, v \rangle = \spam(1, u, v, uv)$.
\end{proof}

\begin{proposition}\label{main}
Let $O$ be a Cayley algebra over $F$ and let $U \subseteq \im(O)$ be a 2-dimensional subspace not containing any nilpotent elements. Then $\langle U \rangle$ is a four-dimensional associative subalgebra of $O$.
\end{proposition}
\begin{proof}
Let $\{u, v\}$ be a basis of $U$. From Lemma \ref{generated} we know that $\langle U \rangle = \langle u, v \rangle$. The latter subalgebra is associative by alternativity of $O$ and from Lemma \ref{3of4} we know that it is either 3 or 4-dimensional. What remains to be shown is that this space is 4- rather than 3-dimensional. We will do this by showing that $\langle u, v \rangle$ is isomorphic to the Dickson-Double of the algebra $\langle u \rangle$, which itself is two-dimensional since $O$ is quadratic (Definition \ref{quadraticdef}). This is will be the proof announced in Remark \ref{special}.

Since $u$ is not nilpotent, we see that $n(u)$, which equals $-u^2$ by (\ref{taun}), is non-zero.

Let $i = v + \frac{\tau(uv)}{n(u)}u$. Since $\frac{\tau(uv)}{n(u)}$ is just a scalar in $F$ this is an element of the linear space $\spam(u, v) \subseteq \langle u, v \rangle \cap \im(O)$. In particular we have by (\ref{im2}) that

\begin{equation}\label{mini}
i^* = -i
\end{equation} 

By construction we also have that $\tau(ui) = 0$, where we exploited that, since $u \in \im(O)$, we have that $u^2 = -n(u)1$ by (\ref{taun}). It follows that $ui \in \im(O)$ and by bilinearity of the product we find that $xi \in \im(O)$ for every $x \in \langle u \rangle = \spam(1, u)$. We conclude from (\ref{im2}), (\ref{star}) and (\ref{mini}) respectively that 
\begin{equation}
xi = -(xi)^* = -(i^*)x^* = ix^*
\end{equation}
from which it follows that the pair $A = \langle u \rangle, i = v + \frac{\tau(uv)}{n(u)}u$ satisfies equation (\ref{xi}). Remark \ref{allfromxi} then implies that it satisfies (\ref{pq1}, \ref{pq2}, \ref{pq3}).

Define $\gamma \in F$ by $i^2 = \gamma1$. Since, by the assumptions of the proposition, $i$ is not nilpotent, $\gamma \neq 0$. Now the fact that the pair $A = \langle u \rangle, i = v + \frac{\tau(uv)}{n(u)}u$ satisfies (\ref{pq1}, \ref{pq2}, \ref{pq3}) implies by Lemma \ref{multiplication} that there is a non-zero homomorphism from the `abstract' four-dimensional Cayley-Dickson algebra $D_{\gamma}(\langle u \rangle)$ to $B \coloneqq \langle u, v \rangle$. 

Since the former of these algebras is simple by Theorem \ref{CDsimple} and the kernel of any homomorphism is a two-sided ideal, this homomorphism is injective and hence $\dim(B) \geq 4$. Since we already established that $\dim(\langle u, v \rangle) \leq 4$ in Lemma \ref{3of4} we find that in fact $\dim(\langle u, v \rangle) = 4$ and $\langle u, v \rangle \isom D_{\gamma}(\langle u \rangle)$.
\end{proof}

\begin{theorem}[Existence of $q$-Fano plane]\label{quantumfano}
Let $F$ be a field of characteristic unequal to 2 over which there exist at least one Cayley division algebra $O$. Then there exists a $2$-$(7, 3, 1)_F$-subspace design over $F$. More precisely: let $O$ be a Cayley division algebra over $F$, let $V = \im(O)$ and let 
$$\mathfrak{B} = \{\im(H) \colon H \textnormal{ is a 4-dimensional associative subalgebra of } O\}.$$ 
Then $\dim V = 7$, $\dim(B) = 3$ for every $B \in \mathfrak{B}$ and every two-dimensional subspace $U \subseteq V$ is contained in a unique element $B$ of $\mathfrak{B}$.
\end{theorem}
\begin{proof}
Let $O, V, \mathfrak{B}$ be as in the statement of the theorem. The given dimensions of $V$ and $B$ for $B \in \mathfrak{B}$ follow from Lemma \ref{codim1}. Let $U \subseteq V$ be 2-dimensional. Write $H = \langle U \rangle$. Since $O$ contains no divisors of zero, it certainly contains no nilpotent elements. By Proposition \ref{main} then, $H$ is a 4-dimensional associative subalgebra of $O$ containing $U$.

\emph{Existence of $B \in \mathfrak{B}$ containing $U$.} 
By construction, $H$ contains $U$. Also by definition $U \subseteq V$ so it follows that $U \subseteq H \cap V$. But by Proposition \ref{H} this latter space equals $\im(H) \in \mathfrak{B}$. 

\emph{Uniqueness of $B \in \mathfrak{B}$ containing $U$.} Suppose there are $B_1, B_2 \in \mathfrak{B}$ such that $U \subseteq B_1 \cap B_2$. By definition there exist four-dimensional associative algebras $H_1, H_2 \subseteq O$ such that $\im(H_1) = B_1$, $\im(H_2) = B_2$. Since both $H_1$ and $H_2$ contain $U$ and both are subalgebras, they both contain $\langle U \rangle$. But since $\dim H_1 = \dim \langle U \rangle$ by Proposition \ref{main} we have that $H_1 = \langle U \rangle$ and similarly $H_2 = \langle U \rangle$. It follows that $H_1 = H_2$ and hence $B_1 = \im(H_1) = \im(H_2) = B_2$. 
\end{proof}

%\begin{proposition}\label{main}
%Dummy placeholder theorem
%\end{proposition}
%
%\begin{theorem}\label{quantumfano}
%Dummy placeholder theorem
%\end{theorem}

%\section{The role of zero divisors}\label{split}

\section{The subalgebras generated by two-dimensional subspaces of $\im(O)$}\label{split}
Proposition \ref{main} showed that if the two-dimensional subspace $U \subset \im(O)$ contains no nilpotent elements, then $\langle U \rangle$ is a four-dimensional simple associative non-commutative algebra. In a sense, the demand that $U$ contains no nilpotents `feels' a bit too strong (though weak enough to cover all cases in case $O$ is a division algebra): the only way it is used in the proof is to guarantee that we can pick a non-nilpotent basis element $u$ of $U$ and that the element $i$ constructed from $u$ and the other, arbitrary, basis element $v$ is not a nilpotent. In the current section we will see that in the case that $O$ is split (and nilpotent elements hence do exist) the class of two-dimensional subspaces generating a four-dimenional simple associative non-commutative subalgebra is indeed strictly larger than the class of two-dimensional subspaces without nilpotents, and that moreover some other two-dimensional subspaces of $\im(O)$ (containing nilpotents) generate non-simple associative algebras that are nevertheless non-commutative and four-dimensional. For the purpose of obtaining a $q$-Fano plane by analogy to Theorem \ref{quantumfano} these non-simple algebras are clearly `good enough'. What we really want to know is whether every algebra $\langle U \rangle$ is necessarily four-dimensional. Unfortunately, when $O$ is split, the answer is no. Before giving an example of a space $U$ generating a three-dimensional subalgebra we discuss some other properties of these subspaces so that we know where to look for them.

\subsection{Subspaces generating three-dimensional subalgebras}

We begin with a trivial but crucial observation:

\begin{lemma}\label{uvinU}
Let $u, v \in \im(O)$. Then $\dim \langle u, v \rangle = 3$ if and only if $uv \in \spam(1, u, v)$. In terms of vector spaces rather than elements: the algebra $\langle U \rangle$ generated by a two-dimensional subspace $U \subset \im(O)$ is 3-dimensional if and only if $\im(\langle U \rangle) = U$ and $\langle U \rangle = F1 \oplus U$ as a vectorspace. In other words $\dim \langle U \rangle = 3$ if and only if the vectorspace $F1 \oplus U$ is closed under multiplication.
\end{lemma}

This gives us \emph{some} algortithm of probing $O$ for such subspaces: just pick random pairs of elements and see where their product lands. However we will try and be more efficient by understanding in more detail the structure of such subspaces and the algebras they generate.

Two-dimensional subspaces of $\im(O)$ that generate only a three-dimensional subalgebra come in two flavours which we will call \emph{type Z} and \emph{type U}:

\begin{proposition}\label{3dim}
Let $U \subset \im(O)$ be such that $\dim U = 2$, $\dim \langle U \rangle = 3$. Then either:
\begin{itemize}
\item $uv = 0$ for every $u, v \in U$ -- we will call such spaces \emph{type Z} spaces because the multiplication on the space $U$ acts as the zero map,
\end{itemize} 
or: 
\begin{itemize}
\item $U$ contains a single line $Fu$ of nilpotent elements and for every $v \in U$ not on that line we have that $v^2 \neq 0$ but $uv \in Fu$ (and hence $vu = -uv \in Fu)$. We will call spaces of the latter type to be of \emph{type U} for reasons explained later.
\end{itemize}

Conversely it is easy to see that when $U$ is of type either $U$ or $Z$ then $F1 \oplus U$ is closed under multiplication and hence equal to the algebra $\langle U \rangle$.
\end{proposition}
\begin{proof}
Let $U$ be such that $\langle U \rangle  = F1 \oplus U$ and let $\{u, v\}$ be a basis of $U$. By Proposition \ref{main} we may assume without loss of generality that $u$ is nilpotent and hence (by Lemma \ref{nilpotent}) that $u^2 = 0$. We know that $v^2 = \alpha 1$ and $uv = \beta 1 + \gamma u + \delta v$ for some $\alpha, \beta, \gamma, \delta \in F$. From $u(uv) = (uu)v$ we find that 
$\beta u + 0 + \delta\beta 1 + \delta \gamma u + \delta^2 v = 0$ so that  
$\beta = \delta  = 0$ and hence that $uv = \gamma u$. 
From $(uv)v = u(v^2)$ we then find that
$\gamma^2 u = \alpha u$ so that 
$\gamma^2 = \alpha$.

Now we distinguish two cases: $\gamma = 0$ and $\gamma \neq 0$. When $\gamma = 0$ we have that $v^2 = \alpha = \gamma^2 = 0$ and from $\beta = \gamma = \delta = 0$ we have that $uv = 0$. We recall that moreover we had $u^2 = 0$ from the start. It follows easily that the product of any two elements in $U = \spam(u, v)$ is zero and hence that $U$ is of type Z.

When $\gamma \neq 0$ we find from $\beta = \delta = 0$ that $uv = \gamma u \in Fu$ and hence in particular $uv \in \im(O)$ so that $vu = v^*u^* = (uv)^* = -uv = -\gamma u \in Fu$ as well. Moreover, from $\alpha = \gamma^2$ we see that $v^2 \neq 0$. In particular if $x = \epsilon u + \zeta v$ with $\zeta \neq 0$ is a generic element of $U$ not on $Fu$ we find that $ux = \zeta \gamma u \in Fu$  but $x^2 = \zeta^2\gamma^2 \neq 0 \in F$ so that $U$ is of type U.
\end{proof}

The algebra structure of $\langle U \rangle$ in case $U$ is of type $Z$ is very easy to understand. Every element $\langle U \rangle$ is of the form $\alpha 1 + u$ with $\alpha \in F, u \in U$ and multiplication is given by $(\alpha 1 + u)(\beta 1 + v) = \alpha \beta 1 + (\beta u + \alpha v)$. 

In case of spaces $U$ of type U we have the following result, explaining the name:

\begin{proposition}\label{typeU}
Let $U$ be a two-dimensional type U space. Then the three-dimensional algebra $\langle U \rangle$ is isomorphic to the algebra of upper triangular 2-by-2 matrices. More strongly: for every basis $u, v$ of $U$ with $u^2 = 0, v^2 = \gamma^2 1$ as in the proof of Proposition \ref{3dim} there is a unique isomorphism 
$\phi_{u, v} \colon \langle U \rangle \to 
\left\{ 
\begin{pmatrix}
\zeta & \eta \\ 0 & \theta 
\end{pmatrix} 
\colon \zeta, \eta, \theta \in F 
\right\}
\subset \Mat(2, F)$ 
given by 
$\phi(1) = \begin{pmatrix}
1 & 0 \\ 0 & 1 \end{pmatrix}$, 
$\phi(u) =  \begin{pmatrix}
0 & 1 \\ 0 & 0 \end{pmatrix}$, 
$\phi(v) =  \begin{pmatrix}
\gamma & 0 \\ 0 & -\gamma \end{pmatrix}$. 

Moreover if we define the involution $*$ on $\Mat(2, F)$ as in Example \ref{Mat2F} then each of the maps $\phi_{u, v}$ preserves the $*$-structure. In particular the space $U = \im(\langle U \rangle)$ is mapped to the space of traceless matrices, the Jacobson radical $J(\langle U \rangle) = Fu$ is mapped to the space of strictly upper triangular matrices and the subalgebra $\spam(1, v)$, which is isomorphic to the (necessarily split) Cayley-Dickson algebra $D_{\gamma^2}(F)$ is mapped to the subalgebra of diagonal matrices in $\Mat(2, F)$.
\end{proposition}
\begin{proof}
Clearly each of the $\phi_{u, v}$ is a linear isomorphism. It suffices to show that it preserves the multiplication and the involution, but this is clear from the properties of the basis $u, v$ derived in the proof of Proposition \ref{3dim}.
\end{proof}

\subsection{Multiplication table of the split Cayley-Algebra}

Now that we understand the structure of the two-dimensional subspaces of $\im(O)$ generating three-dimensional subalgebras and that of the algebras they generate, we return to the question of their existence. By Proposition \ref{main} such subspaces can only exist when $O$ contains zero-divisors, which by Theorem \ref{splitunique} means that it is unique up to isomorphism. This has the advantage that we can give a very explicit description of such Cayley algebras $O$. We will do so using Table \ref{multtab} taken from \cite{CRE}. The authors' remarkable choice of not using 1 as a basis element has the advantage that it highlights various unexpected symmetries among the generators.

\begin{lemma}\label{tab}
The split Cayley algebra over $F$ has a basis consisting of two idempotent elements $p_1$, $p_2$ satisfying $p_1 + p_2 = 1$ and six nilpotent elements $q_1, q_2, q_3, r_1, r_2, r_3$  with multiplication as given by Table \ref{multtab}. The involution acts by $p_1^* = p_2$, $p_2^* = p_1$ and $q_i^* = -q_i$, $r_i^* = -r_i$ for $i = 1, 2, 3$.
\end{lemma}
\begin{proof}
By Thm. \ref{splitunique} and Prop. \ref{Mat2F} we can realize the split Cayley algebra over $F$ as $D_{-1}(\Mat(2, F))$, which in particular means that as a vector space it decomposes as $\Mat(2, F) \oplus i\Mat(2, F)$ as in Definition \ref{DicksonDouble}. We pick $p_1, p_2, q_1, r_1$ in the first summand as follows:
$$p_1 = 
\begin{pmatrix}
1 & 0 \\
0 & 0
\end{pmatrix}; \quad
p_2 = 
\begin{pmatrix}
0 & 0 \\
0 & 1
\end{pmatrix}; \quad
q_1 = 
\begin{pmatrix}
0 & 1 \\
0 & 0
\end{pmatrix}; \quad
r_1 = 
\begin{pmatrix}
0 & 0 \\
1 & 0
\end{pmatrix}.
$$
and then define $q_2, r_2, q_3, r_3$ in the second summand by
$$q_2 = i p_2; \quad r_2 = i p_1; \quad q_3 = -ir_1; \quad r_3 = -iq_1.$$
Then the correctness of Table \ref{multtab} follows from (\ref{pq1}, \ref{pq2}, \ref{pq3}). 
\end{proof}
\begin{table}
\caption{Multiplication table of the split Cayley algebra in the basis described by Lemma \ref{tab}}\label{multtab}
\begin{tabular}{r|cc|ccc|ccc|}
\mbox{} & $p_1$ & $p_2$ & $q_1$  & $q_2$  & $q_3$  & $r_1$  & $r_2$  & $r_3$ \\ \hline
$p_1$   & $p_1$ & $0$   & $q_1$  & $q_2$  & $q_3$  & $0$    & $0$    & $0$ \\
$p_2$   & $0$   & $p_2$ & $0$    & $0$    & $0$    & $r_1$  & $r_2$  & $r_3$ \\ \hline
$q_1$   & $0$   & $q_1$ & $0$    & $r_3$  & $-r_2$ & $-p_1$ & $0$    & $0$ \\
$q_2$   & $0$   & $q_2$ & $-r_3$ & $0$    & $r_1$  & $0$    & $-p_1$ &  $0$ \\
$q_3$   & $0$   & $q_3$ & $r_2$  & $-r_1$ & $0$    & $0$    & $0$    & $-p_1$ \\ \hline
$r_1$   & $r_1$ & $0$   & $-p_2$ & $0$    & $0$    & $0$    & $q_3$  & $-q_2$ \\
$r_2$   & $r_2$ & $0$   & $0$    & $-p_2$ & $0$    & $-q_3$ & $0$    & $q_1$ \\
$r_3$   & $r_3$ & $0$   & $0$    & $0$    & $-p_2$ & $q_2$  & $-q_1$ & $0$ \\ \hline
\end{tabular}
\end{table}

\begin{example}\label{typeUex}
The space spanned by $u = q_1$, $v = (p_1 - p_2)$ is of type $U$ and the three-dimensional algebra it generates is spanned as a vector space by $p_1, p_2, q_1$ with $\phi_{u, v}(p_1) = \begin{pmatrix}
1 & 0 \\ 0 & 0
\end{pmatrix}$, $\phi_{u, v}(p_2) = \begin{pmatrix}
0 & 0 \\ 0 & 1
\end{pmatrix}$, $\phi_{u, v}(q_1) = \begin{pmatrix}
0 & 1 \\ 0 & 0
\end{pmatrix}$  
\end{example}

\begin{example}\label{typeZex}
For $i, j \in \{1, 2, 3\}$, $i \neq j$, the space spanned by $q_i$, $r_j$ is of type Z.  
\end{example}

\subsection{A classification of two-dimensional subspaces of $\im(O)$}

Now that we established the existence of two-dimensional subspaces of $\im(O)$ that generate subalgebras that are only 3-dimensional in case that $O$ is split, it is clear that the proof of Theorem \ref{quantumfano} does not immediately extend to fields over which every Cayley algebra is split, such as $\mathbb{C}$ and $\mathbb{F}_q$ (cf Section \ref{CayleyDickson}). On the other hand it is also not clear that some small variation of the proof would not still work. For example while the algebra $\langle p_1 - p_2, q_1 \rangle$ of Ex. \ref{typeUex} is not a simple four-dimensional subalgebra itself, it is certainly contained in one: the algebra with vector space basis $p_1, p_2, q_1, r_1$. We even have that the map $\phi_{u, v}$ of Ex. \ref{typeUex} extend to an isomorphism of this algebra to $\Mat(2, F)$ by setting $\phi_{u, v}(r_1) = \begin{pmatrix}
0 & 0 \\ 1 & 0
\end{pmatrix}$. All in all it is not yet clear that, even when $O$ is split, the set $\mathcal{B}$ of Theorem \ref{quantumfano} will fail to yield a $q$-Fano plane. In order to trample any optimism stemming from this observation we classify the remaining two-dimensional subspaces of $\im(O)$ and the (four-dimensional) algebras they generate. Inspired by Propositions \ref{main} and \ref{3dim}, we base our classification of two-dimensional subspaces of $\im(O)$ on the number of nilpotent lines they contain and in accordance with the famous cliche about the way mathematicians count, we restict the possibilities for this number to `zero', `one', `two' and `all of them'.

\begin{definition}\label{types}
Let $U$ be a two-dimensional subspace of $\im(O)$. Then:
\begin{itemize}
\item If $U$ contains no nilpotent lines, we say that $U$ is of type Q.
\item If $U$ contains exactly one nilpotent line $Fu$ then:
	\begin{itemize}
	\item If $uv \in Fu$ for every $u \in Fu, v \in U$ then we say that $U$ is of type U
	\item If not, then we say that $U$ is of type D.
	\end{itemize}
\item If $U$ contains exactly 2 nilpotent lines, we say that $U$ is of type M.
\item If every element in $U$ is nilpotent then:
	\begin{itemize}
	\item If the product of any two elements of $U$ is zero we say that $U$ is of type Z
	\item If not we say that $U$ is of type J.
	\end{itemize}
\end{itemize}
\end{definition}
In case $O$ is a division algeba, it is clear that only type Q spaces appear and hence that our classification covers all possible cases. This latter fact is not yet clear in case $O$ is split, but will be established in Lemma \ref{complete} below. The names of the types of subspaces are based on the type of subalgebras of $O$ they generate. The main focus of this section is to show:

\begin{theorem}\label{classification}
Let $U$ be a two-dimensional subspace of $\im(O)$. Then:
\begin{itemize}
\item If $U$ is of type $Q$ then $\langle U \rangle$ is a four-dimensional quaternion subalgebra of $O$. (Prop. \ref{main} above)
\item If $U$ is of type U then $\langle U \rangle$ is a three-dimensional subalgebra of $O$ isomorphic to the upper triangular 2-by-2-matrices (Prop. \ref{typeU} above)
\item If $U$ is of type D then $\langle U \rangle$ is a four-dimensional subalgebra of $O$ with vectorspace decomposition $\langle U \rangle  = D \oplus Z$ where $D$ is a two-dimensional Cayley-Dickson subalgebra of $O$, $Z$ is the two-dimensional Jacobson-radical of $\langle U \rangle$ (so every $z \in Z$ is nilpotent and $uz \in Z$ for every $z \in Z, u \in \langle U \rangle$) and $Z$ happens to be a type Z subspace of $\im(O)$. (Prop. \ref{typeD} below.)
\item If $U$ is of type $M$ then $\langle U \rangle$ is a four-dimensional split quaternion subalgebra of $O$, so in particular isomorphic to $\Mat(2, F)$. (Prop. \ref{typeM} below.)
\item If $U$ is of type $Z$ then $\langle U \rangle$ is three-dimensional, with two-dimensional Jacobson radical $J(\langle U \rangle) = \im(\langle U \rangle) = U$ and multiplication as in the text below Proposition \ref{3dim}. (Prop. \ref{3dim} above.)
\item If $U$ is of type $J$ then $\langle U \rangle$ is a four-dimensional algebra $F1 \oplus J$ with three-dimensional Jacobson radical $J = J(\langle U \rangle) = \im(\langle U \rangle) \supsetneq U$. (Prop \ref{typeJ} below.)
\end{itemize}
\end{theorem}

We note that in order to show that our classification covers all possible cases, it suffices to show that any two-dimensional subspace $U$ containing both a non-zero nilpotent $u$ and a non-nilpotent element $v$, is of one of the three types U, D or M. We will do so by introducing a slightly different trichotomy for spaces of this type which clearly covers all cases and then showing that the two trichotomies coincide.

\begin{lemma}[Completeness of the classification of Def. \ref{classification}]\label{complete}
Let $U = \spam(u, v) \subset \im(O)$ with $u^2 = 0, v^2 = \alpha 1$ for some $\alpha  \neq 0 \in F$. We have:
\begin{enumerate}
\item If $uv \in U$ then $U$ is of type U.
\item If $uv \nin U$ but still $uv \in \im(O)$ then $U$ is of type D.
\item If $uv \nin \im(O)$ then $U$ is of type M.
\end{enumerate}
\end{lemma}
\begin{proof}
Since $u, v \in \im(O)$ we have that $vu = v^*u^* = (uv)^*$ so that $uv + vu = 2\tau(uv)$ by (\ref{taustar}). This means that $uv \in \im(O)$ if and only if $uv + vu = 0$.

When indeed $uv \in \im(O)$ it follows that $(\lambda u + \mu v)^2 = \mu^2 \alpha$ so that the generic element $(\lamba u + \mu v)$ of $U$ is nilpotent if and only if $\mu = 0$. If follows that the only nilpotent elements of $U$ lie on the line $Fu$. However, since $vu = -uv$ in that case, we find that $(uv)^2 = u(vu)v = -u^2v^2 = 0$ so that the product $uv$ is nilpotent. Statements 1 and 2 of the lemma then follow from the definitions of types U and D.

When $uv \nin \im(O)$, we have that $uv + vu = \beta 1$ for some $\beta \neq 0 \in F$. Let $u' = \frac{-\alpha}{\beta}u$ Then $(u' + v)^2 = 0 + \frac{-\alpha}{\beta}\beta + \alpha  = 0$. It follows that $U$ contains at least two distinct nilpotent lines that together span $U$: $Fu = Fu'$ and $F(u' + v)$. What remains to be shown is that these are the only nilpotent lines; that is that $(\lambda u' + \mu(u' + v))^2 = 0$ only if $\lambda = 0$ or $\mu = 0$. To verify this we compute that $(\lambda u' + \mu(u' + v))^2 = \lambda \mu (u'v + vu') = \lambda\mu (\frac{-\alpha}{\beta}(uv + vu)) = -\lambda \mu \alpha$. Since $\alpha \neq 0$, the claim follows.
\end{proof}

In a curious turn of events the lemma proves its own converse:

\begin{corollary}\label{curious} \mbox{}

\begin{enumerate}
\item Let $u, v \in \im(O)$ be elements of a type U space $U$ with $u^2 = 0$, $v^2 \neq 0$. Then $uv \in U$.
\item Let $u, v \in \im(O)$ be elements of a type D space $U$ with $u^2 = 0$, $v^2 \neq 0$. Then $uv \nin U$, but still $uv \in \im(O)$.
\item Let $u, v \in \im(O)$ be elements of a type M space $U$ with $u^2 = 0$, $v^2 \neq 0$. Then $uv \nin \im(O)$.
\end{enumerate}
\end{corollary}
\begin{proof} \mbox{}

\begin{enumerate}
\item If $uv$ landed anywhere else, the space $U$ would be of type $D$ or $M$ by Lemma \ref{complete}.
\item If $uv$ landed anywhere else, the space $U$ would be of type $U$ or $M$ by Lemma \ref{complete}.
\item If $uv$ landed anywhere else (that is: in $\im(O)$), the space $U$ would be of type $U$ or $D$ by Lemma \ref{complete}.
\end{enumerate}
\end{proof}
This converse is useful as it gives us a description of spaces $U$ of type D and M that already anticipates the determination of the algebras $\langle U \rangle$. After all: if $uv \nin F1 \oplus U$ (as is clearly the case for type D and will later be verified for type M) then Lemma \ref{3of4} tells us that $\{1, u, v, uv\}$ is a basis the space $\langle U \rangle$. 

We'll put this principle into action right away.

\subsection{The algebras generated by subspaces of type D}

\begin{proposition}\label{typeD}
Let $U \subseteq \im(O)$ be two-dimensional of type D, so with basis $u, v$ satisfying $u^2 = 0$, $v^2 = \alpha \in F \backslash \{0\}$ and $w \coloneqq uv \in \im(O) \backslash U$. Then $\langle U \rangle$ is four-dimensional, with two-dimensional Jacobson radical $Z = \spam(u, w)$ and semi-simple quotient $\langle U \rangle / Z \isom D_\alpha(F)$. The space $Z$ is of type $Z$, all other two-dimensional subspaces $U'$ of $\langle U \rangle$ are either of type $D$ or type $U$ with the unique nilpotent line in $U'$ being the intersection $U' \cap Z$ with the space $Z$.
\end{proposition} 
\begin{proof}
Since $u, v, w$ are linearly independent and imaginary it is clear that $\im(\langle U \rangle)$ is at least three-dimensional, and by Lemma \ref{3of4} it is exactly three-dimensional with basis $\{u, v, w\}$. Secondly, we note that $\langle v \rangle = F1 \oplus Fv \isom D_\alpha(F)$ so that we obtain a vectorspace decomposition $\langle U \rangle = D_\alpha(F) \oplus Z$ so that, as in text below Theorem \ref{SJ}, the claim that $\langle U \rangle / Z \isom D_\alpha(F)$ follows as soon as we establish that $Z$ is indeed the Jacobson radical. This is easiest from the characterisation of the Jacobson radical in Lemma \ref{stronglyimaginary} as the space of all $x \in \langle U \rangle$ satisfying $xy \in \im(\langle U \rangle)$ for all $y \in \langle U \rangle$. Indeed: let $\beta1 + \gamma u + \delta v + \epsilon w$ be an element of the Jacobson radical, then, setting $y = 1$ we find that $\beta = 0$ and setting $y = v$ we find that $\delta = 0$. This shows that $y \in Z$ so that $J(\langle U \rangle) \subset Z$. Conversely, since $u1 = u$, $u^2 = 0$, $uv = w$ and $uw = u^2v = 0$ are all imaginary we have that $u \in J(\langle U \rangle)$ and similarly imaginarity of $w1 = w, wu = -(vu)u = -vu^2 = 0$, $wv = uv^2 = \alpha u$ and $w^2 = -(vu)(uv) = vu^2v = 0$ implies that $w \in J(\langle U \rangle)$. Together these inclusions show that $Z \subset J(\langle U \rangle)$. The fact that $Z$ is of type Z follows from the equations $u^2 = uw = w^2 = 0$ established above.

Let $x = x_v + x_z$ be the decomposition of $x \in \langle U \rangle$ along the vectorspace decomposition $\langle U \rangle = \langle v \rangle \oplus Z$. Since $\langle v \rangle$ is a subalgebra and $Z$ is an ideal we find that $x^2 = 0$ can only hold if $x_v^2 = 0$. However since $\langle v \rangle \isom D_\alpha(F)$ which is either a division algebra or isomorphic to the algebra direct sum $F \oplus F$ (Proposition \ref{CD2}) contains no nilpotent elements we find that $x^2 = 0$ implies that $x_v = 0$ and hence $x \in Z$. By Lemma \ref{nilpotent} this means that $Z$ is not only the space of all strongly nilpotent elements but also the space of \emph{all} nilpotent elements in $\langle U \rangle$.  

With this observation in mind, the claim that all other two-dimensional subspaces of $\im(\langle U \rangle)$ are of either type U or type D is trivial from dimension considerations: since $\im(\langle U \rangle)$ is three-dimensional each two-dimensional subspace $U'$ other than $Z$ has one-dimensional intersection with $Z$ and this intersection consists of all nilpotent elements in the subspace. Hence $U'$ contains exactly one nilpotent line and thus is of type U or D by Definition \ref{types}. We will see in Chapter \ref{counting} (Corollary \ref{TXF2J2}) that type U occurs if and only if $D_\alpha(F)$ is split.
\end{proof}

\begin{example}\label{typeDex}
Let $\alpha \in F \backslash \{0\}$ and $i \neq j \in \{1, 2, 3\}$. With notation as in Lemma \ref{tab} and Table \ref{multtab}, the space spanned by $u = q_i$ and $v = q_j - \alpha r_j$ is of type D by Lemma \ref{complete}. In particular, since $v^2 = \alpha 1$, we see from the proof of Proposition \ref{typeD} that every isomorphism class of two-dimensional Cayley-Dickson algebras $D_\alpha(F)$ appears as the semi-simple part of some subalgebra $\langle U \rangle \subset O$ with $U$ of type D.
\end{example} 

\subsection{The algebras generated by subspaces of type M}

For the type M case we work with a slightly different basis (which we can think of as the pair $(u', u' + v)$ of the proof of Lemma \ref{complete}):

\begin{lemma}\label{typeMbasis}
Let $U \subset \im(O)$ be a two-dimensional subspace of type M. Then $U$ has a basis $q, r$ satisfying $q^2 = r^2 = 0$, $2\tau(qr) = qr + rq = 1$.
\end{lemma}
\begin{proof}
Pick non-zero elements $q, r'$ on the two nilpotent lines. $q$ and $r'$ span $U$ and $q + r'$ can not be a nilpotent which in particular means that $(q + r')^2 \neq 0$. On the other hand we know (from $(q + r') \in \im(O)$) that $(q + r')^2 \in F1$. Define $\alpha \in F$ by $(q + r')^2 = \alpha 1$ and define $r = r'/\alpha$. Then expanding $(q + \alpha r)^2 = \alpha 1$ yields $qr + rq = 1$. The fact that $qr + rq = 2\tau(qr)$ whenever $q, r \in \im(O)$ has been derived many times above.
\end{proof}

\begin{proposition}\label{typeM}
Let $U \subseteq \im(O)$ be a two-dimensional subspace with basis elements $q, r$ satisfying $q^2 = r^2 = 0$, $2\tau(qr) = 1$. Then $\langle U \rangle$ is four-dimensional and isomorphic to $\Mat(2, F)$ with isomorphism given by $$
\begin{pmatrix}
1 & 0 \\
0 & 0
\end{pmatrix} \mapsto qr; \quad
\begin{pmatrix}
0 & 1 \\
0 & 0 
\end{pmatrix} 
\mapsto q; \quad
\begin{pmatrix}
0 & 0 \\
1 & 0 
\end{pmatrix} 
\mapsto r; \quad
\begin{pmatrix}
0 & 0 \\
0 & 1 
\end{pmatrix} 
\mapsto rq.
$$ 
\end{proposition}
The suggestive notation is chosen partially in order to make it easier to write down examples of type $M$ subspaces in terms of the basis of Table \ref{multtab}.
\begin{proof}
Set $e_1 = qr$, $e_2 = rq$. We first verify that $e_1$ and $e_2$ are commuting idempotents summing to 1. By (\ref{n}) we have that $n(qr) = (qr)^*qr = (rq)(qr) = r(qq)r = 0$. Hence by $(qr)^2 = 2\tau(qr)qr - n(qr) = qr$ (\ref{taun}) we find idempotence of $qr$. The calculation for $rq$ is identical. We also notice that $e_1e_2 = qrrq = 0$ and $e_2e_1 = rqqr = 0$. The identity $e_1 + e_2 = 1$ is just the identity $x + x^* = 2\tau(x)$ applied to $x = qr$. 

We see that $\spam(e_1, e_2)$ is a two-dimensional algebra and hence $\spam(e_1, e_2) \cap \im(O) = \im(\spam(e_1, e_2))$ is one-dimensional by Lemma \ref{H}. From $e_1^* = e_2$ it then follows that $\spam(e_1, e_2) \cap \im(O) = F(e_1 - e_2)$. In order to establish linear independence of $e_1, e_2, q, r$ it hence suffices to show that $(e_1 - e_2) \nin U$. Suppose, by contradiction, that $qr - rq = \lambda q + \mu r$ then, from multiplying on the left by $q$, we would have that $qrq = -\mu qr$. However, from $(qrq)^* = q^*r^*q^* = (-1)^3qrq$ we see that $qrq \in \im(O)$ so that $\tau(qrq) = 0$. On the other hand $\tau(\mu qr) = \mu/2$ so that $qrq = -\mu qr$ implies that $\mu = 0$. Similarly we arrive at $\lambda = 0$ by multiplying our assumption $qr - rq = \lambda q + \mu r$ with $r$. But this implies that $qr = rq$, which by $qr + rq = 1$ would mean that $qr = 1/2$ and hence $n(qr) = 1/4$, contradicting the already established relation $n(qr) = 0$.

We proceed to verify the algebraic relations between the basis elements $e_1, e_2, q, r$ implied by the supposed isomophism to $\Mat(2, F)$. Relations $e_1r = e_2q = qe_1 = e_2r = 0$ are obvious from associativity once writing out the definition of $e_1$ and $e_2$. Since $e_1q = qe_2 = qrq$ and $e_1r = re_2 = rqr$ it remains to show that $qrq = q$ and $rqr = r$. From $qr + rq = 1$ we see that $qrq + rq^2 = q$ but since $rq^2 = 0$ this reduces to $qrq = q$. Similarly $qr + rq = 1$ gives $rqr + r^2q = r$.  
\end{proof}

\subsection{The algebras generated by subspaces of type J}

As a byproduct of Proposition \ref{typeM} we find that the converse of Lemma \ref{typeMbasis} also holds:
\begin{lemma}
Let $q, r$ be linearly independent nilpotent elements in $O$ such that $2\tau(qr) = 1$ then $U \coloneqq \spam\{q, r\}$ is of type $M$.
\end{lemma}
\begin{proof} 
By Proposition \ref{typeM}, $U$ generates a subalgebra of $O$ isomorphic to $\Mat(2, F)$ in which the space $U$ is mapped isomorphically to the subspace $\left\{ \begin{pmatrix}
0 & \beta \\ \alpha & 0 \end{pmatrix} \colon \alpha, \beta \in F \right\}$. Using our knowledge of linear algebra to count the number of nilpotent lines in the latter space, we find that this number equals 2.
\end{proof}
\begin{corollary}\label{typeJbasis}
Let $U \subset \im(O)$ be a two-dimensional subspace of type J. Then $2\tau(uv) = 0$ for every $u, v \in U$.
\end{corollary}
\begin{proof}
We have that every element of $U$ is nilpotent by definition of type $J$ so if $u, v$ are linearly depedent we find that $uv = 0$ by Lemma \ref{nilpotent}. It follows that any hypothetical pair $\{u, v\}$ with $2\tau(uv) \neq 0$ would span $U$, as would the pair $\{u, v/(2\tau(uv)\}$. But the latter pair satisfies the conditions of the last lemma, showing that in a type J space no such pair can exist.
\end{proof}
This corollary will help us finish the proof of Theorem \ref{classification} by identifying the algebras generated by subspaces of type $J$. We will first however give an example of such a space.
\begin{example}\label{typeJex}
In the notation of Table \ref{multtab}, every two-dimensional subspace $U$ of $\spam\{q_1, q_2, q_3\}$ is of type $J$. Note that for any two linearly independent elements $u, v$ in such a space $U$ the product $uv$ does not only lie outside $U$ (as is necessary for $\langle U \rangle$ to be four-dimensional) but even outside all of $\spam\{q_1, q_2, q_3\}$. On the other hand we notice that $uv$ still lies in $\im(O)$ as predicted by Corollary \ref{typeJbasis}. By symmetry also all two-dimensional subspaces of $\spam\{r_1, r_2, r_3\}$ are of type $J$. 
\end{example}
\begin{proof}
It suffices to show that every element of $\spam\{q_1, q_2, q_3\}$ is nilpotent. For this it is suffices to see that $q_1^2 = q_2^2 = q_3^2 = 0$ and $q_iq_j + q_jq_i = 0$ for $i \neq j$. The first of these is immediate from Table \ref{multtab}, the second can be re-expressed as $2\tau(q_iq_j) = 0$ which is equivalent to $q_iq_j \in \im(O)$ for all $i, j$. This latter fact is again immediate from Table \ref{multtab}.
\end{proof}
Note that we do not claim that Example \ref{typeJex} covers all possible type J spaces, just as Example \ref{typeUex} did not cover all type U spaces. We will return to the question of finding (or at least counting) all spaces of a given type in Section \ref{counting}.

The following proposition completes the proof of Theorem \ref{classification}.
\begin{proposition}\label{typeJ}
Let $U = \spam(u, v) \subseteq \im(O)$ be of type J. (So $u^2 = 0$, $v^2 = 0$, $uv \neq 0$ and, by Corollary \ref{typeJbasis}, $\tau(uv) = 0$.)
Then $\langle U \rangle$ is four-dimensional with basis $1, u, v, uv$ and multiplication given by $u^2 = v^2 = (uv)^2 = 0$; $uv = -vu$; $(uv)x = x(uv) = 0$ for all $x \in \im(\langle U \rangle)$. In particular every element of $\im(\langle U \rangle)$ is nilpotent and $\im(\langle U \rangle)$ is the three-dimensional Jacobson radical of $\langle U \rangle$.
\end{proposition}
\begin{proof}
Since $\tau(uv) = 0$ we have that $uv \in \im(O)$. By Lemma \ref{H} this means that $uv \in \im(\langle U \rangle)$ and hence that in order to (counterfactually) have $\dim(\langle U \rangle) = 3$, one must have that $uv \in U$. But this would imply the existence of $\lambda, \mu \in F$ such that $uv = \lambda u + \mu v$. Multiplying both sides with $v$ from the right yields $\lambda uv = 0$ and hence $\lambda = 0$. Similarly, multiplying from the left by $u$ yields $\mu = 0$, a contradiction. It follows that indeed $\langle U \rangle$ is four-dimensional. 

Since $n(uv) = n(u)n(v) = 0$ we see from (\ref{taun}) that $(uv)^2 = 0$. From $u, v, uv \in \im(O)$ we have that $-uv = (uv)^* = v^*u^* = (-v)(-u) = vu$ and from this it follows that $vuv = -uv^2 = 0, uvu = -u^2v = 0$ verifying the claims about multiplication.
\end{proof}

\begin{corollary}\label{F1J3subspaces}
Let $U, u, v$ be as in Proposition \ref{typeJ}. An element $w \in \langle U \rangle$ satisfies $wx = xw = 0$ for all $x \in \im(\langle u, v \rangle)$ if and only if $w$ lies on the line $F(uv)$ and consequently every two-dimensional subspace of $\im(\langle u, v \rangle)$ containing that line is of type Z. Conversely, every two-dimensional subspace $U' \subseteq \im(\langle u, v \rangle)$ \emph{not} containing the line $F(uv)$ is of type $J$ and satisfies $\langle U' \rangle = \langle u, v \rangle$. 
\end{corollary}
\begin{proof}
The statements in the first sentence follow immediately from the description of the multiplication in Proposition \ref{typeJ}. For the second sentence, let $x = \alpha u + \gamma v + \epsilon uv, y = \beta u + \delta v + \zeta uv$ be a basis of a subspace $U' \subseteq \im(\langle U \rangle$). We assume that $xy = 0$ and derive that $U'$ contains $Fuv$.

From Proposition \ref{typeJ} we compute that $xy = (\alpha\delta - \beta\gamma)uv$. Here we recognize in the coefficient of $uv$ the determinant of the two-by-two matrix $\begin{pmatrix}
\alpha & \beta \\
\gamma & \delta
\end{pmatrix}$.
We know from linear algebra that this number equals $0$ only if the vectors $\begin{pmatrix} \alpha \\ \gamma \end{pmatrix}, \begin{pmatrix} \beta \\ \delta \end{pmatrix}$
are linearly dependent. In other words: when $xy = 0$ we have that there exist $\lamba, \mu \in F$, not both zero, such that $\lambda \begin{pmatrix} \alpha \\ \gamma \end{pmatrix} + \mu \begin{pmatrix} \beta \\ \delta \end{pmatrix} = 0$. This implies that $\lambda x + \mu y \in Fuv$. But since $x, y$ are linearly independent we also have that $\lambda x + \mu y \neq 0$. So $U'$ contains a non-zero element of $Fuv$ and hence the entire line $Fuv$.

Conversely this means that if $U'$ does not contain $Fuv$ we have that $xy \neq 0$, implying that $U'$ is of type J as, by Proposition \ref{typeJ}, every element in $\im(\langle u, v \rangle)$ and so in particular every element in $U'$ is nilpotent. It remains to verify that in this case $\langle x, y \rangle = \langle u, v \rangle$. Even when $xy \neq 0$, we still have that $xy \in Fuv$, so that the algebra $\langle x, y \rangle$ \emph{does} contain the line $Fuv$. It follows that $\langle x, y \rangle$ contains the elements $x - \epsilon uv = \alpha u + \gamma v$ and $y - \zeta uv = \beta u + \delta v$ and since we established that $\det \begin{pmatrix} \alpha & \beta \\ \gamma  & \delta \end{pmatrix} \neq 0$ when $U'$ does not contain $Fuv$, these elements together span $U = \spam(u, v)$. Hence $\langle x, y \rangle \supseteq \langle u, v \rangle$. The converse inclusion is true by assumption.
\end{proof}

\subsection{No $q$-Fano planes from split Cayley algebras}

Now armed with a better understanding of the algebras generated by the various two-dimensional subspaces of $\im(O)$ we return to the question raised at the beginning of the section of how far the collection of 3-dimensional spaces $\im(H)$ where $H$ ranges over the four-dimensional associative subalgebras of the split octonion algebra $O$, is from providing a $q$-Fano plane structure on $\im(O)$. From Theorem \ref{classification} and the completeness (Lemma \ref{complete}) of the classification given in Definition \ref{types} the following `strengthening' of Theorem \ref{quantumfano} is immediate:

\begin{corollary}[Existence of `almost q-Fano planes']\label{almost}
Let $O$ be a Cayley algebra (not-necessarily non-split) and let $V$, $\mathfrak{B}$ be as in Theorem \ref{quantumfano}. Then for every two-dimensional subspace $U$ of $V$ not of type U or Z there exists a unique $B \in \mathfrak{B}$ such that $U \subseteq B$.
\end{corollary}

Of course this `strengthening' is of little additional value since subspaces of type U and Z \emph{do} appear in $O$ whenever $O$ is split. The interesting question is what goes wrong if we extend the claim in the naive way to all subspaces $U$ including those for which $\dim \langle U \rangle = 3$. Is it existence, is it uniqueness, is it both? We already saw in Example \ref{typeUex} that for at least one type U space existence is \emph{not} a problem and the same is true for the type $Z$ spaces mentioned in Corollary \ref{F1J3subspaces}. It seems not too much of a stretch to imagine that for \emph{every} $U$ of type U or Z there exists at least one $B$ in $\mathcal{B}$ such that $U \subset B$. It was pointed out to me by Relinde Jurrius that this would mean that the pair $(V, \mathcal{B})$ provides an example of a $q$-covering design and we will show in the next section that this is indeed the case. 

However when it comes to uniqueness, things are looking less sunny.

\begin{notation}
Let $u, w \in \im(O)$ such that $u^2 = w^2 = 0$, $2\tau(uw) = 1$ (so $\spam(u, w)$ is of type M by Lemma \ref{typeMbasis}). Then we write $H_{u, w} \coloneqq \langle u, w \rangle$ and write $\phi_{u, w}\colon \Mat(2, F) \to H_{u, w}$ for the isomorphism given in Proposition \ref{typeM}. (So $\phi_{u, w}$ depends not just on the type M space, but really on the vectors $u$ and $w$.) Moreover let $T \subset \Mat(2, F)$ denote the three-dimensional subalgebra of upper triangular matrices.
\end{notation}

\begin{example}\label{problem}
Let $O$ be a split Cayley algebra. Then there exists a 2-dimensional linear space $U \subset \im(O)$ of type U and a two-dimensional \emph{affine} subspace $W \subseteq \im(O)$ such that for every $w \in W$ the space $U' \coloneqq \spam(u, w)$ is of type M, the four-dimensional algebra $H_{u, w}$ contains the three-dimensional algebra $\langle U \rangle$ (so in particular $U \subseteq \im(H_{u, w})$) and $\langle U \rangle = \phi_{u, w}(T) \subset \phi_{u, v}(\Mat(2, F)) = H_{u, w}$. Moreover, for $w_1 \neq w_2 \in W$ we have that $H_{u, w_1} \neq H_{u, w_2}$ and hence $H_{u, w_1} \cap H_{u, w_2} = \langle U \rangle$.
\end{example}
\begin{proof}
In the basis of $O$ given by Lemma \ref{tab}, let $U$ be the space spanned by $u = q_1$ and $v = p_1 - p_2$. It follows that $\langle U \rangle = \spam(p_1, p_2, q_1)$. Let $W = -r_1 + Fr_2 + Fr_3$. Then for every $w \in W$ we have that $w^2 = 0$ and $uw = p_1$ so that $2\tau(uw) = p_1 + p_1^* = p_1 + p_2 = 1$. The statements on the structure of $H_{u, w}$ then follow from (the proof of) Proposition \ref{typeM}. 

To see that $H_{u, w_1}$ and $H_{u, w_2}$ are different we notice that as vecor spaces they are equal to $\langle U \rangle \oplus F w_1$, $\langle U \rangle \oplus F w_2$ respectively and so the result follows from the observations that 
$\{w_1 - w_2 \colon w_1, w_2 \in W\} = \spam(r_2, r_3)$ and $\spam(r_2, r_3) \cap \langle U \rangle = \{0\}$.
\end{proof}

Of course it is totally conceivable that one can prove a statement of the form `for every Type U space $U$ there is a \emph{unique} four-dimensional associative subalgebra of $O$ containing $\langle U \rangle$ and satisfying additional property X', but this won't salvage the $q$-Fano plane. The problem here lies in the fact that for each type M space $U'$ appearing in Example \ref{problem}, the algebra $H_{u, w}$ is the \emph{only} four-dimensional associative algebra containing $U'$. Hence we see:

\begin{corollary}\label{hopeless}
Let $\mathfrak{H}$ be any collection of four-dimensional associative subalgebras of a \emph{split} Cayley algebra $O$ and let $\mathfrak{B} = \{\im(H) \colon H \in \mathfrak{H}\}$. Then either there is a type M subspace $U'$ not contained in any $B \in \mathfrak{B}$ or there is a type U subspace $U$ contained in more than one of the $B \in \mathfrak{B}$.
\end{corollary}

In other words: in order to extend Theorem \ref{quantumfano} to fields were every Cayley algebra is split, a new idea is needed.

%\section{$q$-Covering designs from split Cayley algebras}\label{qcover}
%deze naam is misschien niet de beste maar willen nu even Tex-dingen goed krijgen.
%\begin{theorem}\label{main2}
%Dummy placeholder theorem
%\end{theorem}

\section{$q$-Covering designs from split Cayley Algebras}\label{qcover}
It was pointed out to me by Relinde Jurrius that Theorem \ref{almost} `almost' shows that the set $\mathcal{B}$ of Theorem \ref{quantumfano} is a $(7, 3, 2)$-$q$-covering design over $F$ (even when it fails to be a proper $2$-$(7, 3, 1)$-subspace design over $F$, that is in the cases where the underlying Cayley algebra $O$ is split.) We recall the definition:
\begin{definition}
A $(v, b, t)$-$q$-covering design over a field $F$ is a collection $\mathcal{B}$ of $b$-dimensional subspaces of a $v$-dimensional space $V$ such that every $t$-dimensional subspace $T$ of $V$ is contained in at least one element of $\mathcal{B}$.
\end{definition}

Indeed, in the light of Corollary \ref{almost} the only thing missing from a proof that $\mathcal{B}$ yields a $q$-covering design is an argument showing that the type Z spaces mentioned in Proposition \ref{F1J3subspaces} and the type U space mentioned in Example \ref{typeUex} (both of which are obviously contained in one or more elements of $\mathcal{B}$) are, in the right sense, `representative' of all type Z and type U spaces respectively. Such an argument is handed to us on a golden platter by the following result, taken from Section 1.7 of the book by Springer and Veldkamp \cite{SprVeld}:

\begin{theorem}\label{transitive}
Let $D_1, D_2$ be isomorphic subalgebras of a Cayley algebra $O$. Then every linear isomorphism $\phi \colon D_1 \to D_2$ extends to an automorphism of all of $O$.
\end{theorem}

By definition an automorphism is a linear map from an algebra to itself that preserves the multiplication. We verify the (unsurprising) fact that in the special case of Cayley algebras the automorphims also preserve most of the other structure we care about:

\begin{lemma}\label{autonice}
Let $\phi$ be an automorphism of $O$ and $x \in O$. Then  $\tau(\phi(x)) = \tau(x)$, $\phi(x^*) = \phi(x)^*$ and $\phi(x) \in \im(O)$ if and only if $x \in \im(O)$.
\end{lemma}
\begin{proof}
All statements are obvious if $x = 0$, so from now on we assume that $x \neq 0$. We recall that any automorphism satisfies $\phi(1) = 1$ and hence by linearity maps every element of $F1$ to itself.

We first prove the third statement. Let $x \in \im(O)$. Then $x^2 \in F1$ and hence $\phi(x^2) = x^2$. But $\phi(x^2) = \phi(x)^2$ so that $\phi(x)^2 \in F1$. It follows that either $\phi(x) \in \im(O)$ or $\phi(x) \in F1$. But the latter case we would have $x = \phi^{-1}(\phi(x)) \in F1$ which is absurd. It follows that $\phi(x) \in \im(O)$ as desired. Conversely, suppose that $\phi(x) \in \im(O)$, then $x = \phi^{-1}(\phi(x)) \in \im(O)$ by the statement just proved.

Now let $x$ be general. We can write $x = \tau(x) + \im(x)$ as in Section \ref{quadratic} with $\tau(x) \in F1$ and $\im(x) \in \im(O)$. It then follows that $\phi(x) = \phi(\tau(x)) + \phi(\im(x)) = \tau(x) + \phi(\im(x))$ and since we just established that $\phi(\im(x)) \in \im(O)$ we find that $\tau(\phi(x)) = \tau(x)$ and $\im(\phi(x)) = \phi(\im(x))$. Finally, from $x^* = \tau(x) - \im(x)$ we then obtain that $\phi(x)^* = \phi(x^*)$.
\end{proof}

We explicitely write out the relevant consequences of this theorem for type Z and type U spaces:

\begin{corollary}\label{ZUtrans}
Let $O$ be a Cayley algebra. The automorphism group $\Aut(O)$ of $O$ acts transitively on the set of type Z subspaces of $O$ and on the set of type U subspaces of O.
\end{corollary}
\begin{proof}
Let $Z_1, Z_2$ be type Z subspaces of $O$ and let $\phi \colon Z_1 \to Z_2$ be \emph{any} linear isomorphism between them. Then $\phi$ extends to an algebra morphism between the algebras $F1 \oplus Z_1$ and $F1 \oplus Z_2$ by setting $\phi(1) = 1$ (cf the text just below Propostion \ref{3dim}) and hence to an automorphism of all of $O$ by Theorem \ref{transitive}.

Similarly let $U_1, U_2$ be two U subspaces. By Propostion \ref{typeU} both $F1 \oplus U_1$ and $F_1 \oplus U_2$ are subalgebras of $O$, isomorphic to the algebra of upper triangular 2-by-2-matrices over $F$ and hence in particular isomorphic to eachother. Let $\phi: F1 \oplus U_1 \to F1 \oplus U_2$ be any isomorphism. Then $\phi$ extend to an isomorphism of all of $O$ which, by Lemma \ref{autonice} maps $U_1 = \im(F1 \oplus U_1)$ to $\im(F1 \oplus U2) = U_2$.
\end{proof}

In particular: let $U_2$ be any type U or Z subspace of $\im(O)$ and and $U_1$ a subspace of the same type we know to be contained in $\im(H)$ for some four-dimensional subalgebra $H$ of $O$. Then by Cor. \ref{ZUtrans} there exists an automorphism $\phi$ of $O$ such that $\phi(U_1) = U_2$. It follows that $U_2$ is contained in $\phi(H)$, which is a four-dimensional subalgebra of $\im(O)$ by the automorphism property of $\phi$. Lemma \ref{autonice} then reassures us that in fact it is even contained in $\im(H)$. Combining this with the results of the previous two sections we obtain:

\begin{theorem}\label{main1}
Let $O$ be a Cayley algebra over $F$, $V = \im(O)$ and 
$$\mathcal{B} = \{\im(H) \colon H \textrm{ is a 4-dimensional subalgebra of } O\}.$$ 
Then $(V, \mathcal{B})$ is a $(7, 3, 2)$-$q$-covering design which is a $2$-$(7, 3, 1)$-subspace design if and only if $O$ is a division algebra.
\end{theorem} 

Now the existence of a $q$-covering designs is more interesting the less blocks it contains. After all the set of \emph{all} 3-dimensional subspaces of $V$ constitutes a $(7, 3, 2)$-$q$-covering design as well. Hence we might hope that we can get a more impressive, if somewhat more ugly, result by specifiying in Theorem \ref{main1} that the collection $\mathcal{B}$ should be restricted to the imaginary parts of four-dimensional algebras generated by two-dimensional subspaces of type Q, D, M or J. However such a specification is unnecessary since, as it turns out, \emph{every} four-dimensional subalgebra of a Cayley algebra $O$ is generated by a two-dimensional subspace of one of the four mentioned types. By the results of Section \ref{split} it suffices to show that every four-dimensional subalgebra of $O$ \emph{contains} at least one subpace of one of the four types and that is what we will do now.

\begin{proposition}\label{QMDJ}
Let $O$ be a Cayley algebra. Every four-dimensional subalgebra of $O$ contains (and hence is generated by) at least one subspace of type either Q, M, D or J.
\end{proposition}
\begin{proof}
The proof consists of imagining what a four-dimensional algebra $H$ such that all two-dimensional subspaces of $\im(H)$ are of type U or Z would look like and then showing that such an algebra cannot appear as a subalgebra of $O$. This suffices since by Lemma \ref{complete} the six types Q, M, D, J, U, Z are the only possibilities. Throughout the proof we assume that $H = F1 \oplus \im(H)$ with every two-dimensional subspace of $\im(H)$ being of type either U or Z.

The simplest example of such an algebra $H$ is an algebra in which the product of \emph{any} two elements of $\im(H)$ is zero. This is a perfectly well defined associative four-dimensional quadratic algebra with a strong involution in which every two-dimensional subspace of $\im(H)$ is of type Z. However it cannot appear as a subalgebra of a Cayley division algebra (obviously) and neither as a subalgebra of the split one as we will verify with help of Table \ref{multtab}. For future reference we make this latter claim into a separate lemma.

\begin{lemma}\label{Z3}
The split Cayley algebra $O$ contains no four-dimensional subalgebras $H$ such that the product of any two imaginary elements in $H$ equals zero.
\end{lemma}
\begin{proof} 
Aiming for a contradition, assume that $H \subset O$ is such an algebra. Let $z_1, z_2, z_3$ a basis of $\im(H)$. Since $\spam(z_1, z_2)$ is of type Z, (the proof of) Corollary \ref{ZUtrans} guarantees the existence of an automorphism $\phi$ of $O$ such that $\phi(z_1) = q_1$ and $\phi(z_2) = r_2$. Since $\phi(z_3) \in \im(O)$, we can write $\phi(z_3) = \alpha(p_1 - p_2) + \sum_{i =1}^3 \beta_i q_i + \gamma_i r_i$ for scalars $\alpha, \beta_i, \gamma_i \in F$. Solving $q_1 \phi(z_3) = r_2\phi(z_3) =  0$ we find that $\alpha = \beta_2 = \beta_3 = \gamma_1 = \gamma_3 = 0$ and hence that $\phi(z_3) \in \spam(\phi(z_1), \phi(z_2))$ contradicting the fact that $\phi$, being an automorphism, is in particular a linear isomorphism.
\end{proof}
We continue with the proof of Proposition \ref{QMDJ}. Recall that we assumed $H$ to be a four-dimensional subalgebra of $O$ such that every two-dimensional subspace of $\im(H)$ is of type either U or Z. Since Lemma \ref{Z3} rules out the case that \emph{all} two-dimensional subspaces are of type Z we can move on to algebras $H$ that contain at least one non-nilpotent imaginary element. 

Let $\{h_1, h_2, h_3\}$ be a basis of $\im(H)$ with $h_1$ not nilpotent. Since the spaces $\spam(h_1, h_2)$ and $\spam(h_1, h_3)$ both contain $h_1$ they are not of type Z and hence of type U. Let $l_1, l_2$ be their respective unique nilpotent lines. Then $l_1 \neq l_2$ and hence they span a space $Z$ which, containing at least two nilpotent lines cannot be of type U and hence must be of type Z. We conclude $\im(H)$ must contain at least one type $Z$ space.

On the other hand, such $H$ also contains at most one type Z space. For let $Z_1, Z_2$ be two distinct type Z subspaces of $\im(H)$ and let $M$ be a two-dimensional subspace containing at least one non-nilpotent element. Then since $M$ contains at least two nilpotent lines ($M \cap Z_1$ and $M \cap Z_2$) as well as a non-nilpotent, Lemma \ref{complete} states that it has to be of type M, contradicting the assumption that all subspaces of $\im(H)$ are of types Z or U.

Knowing that $H$ contains exactly one type Z space $Z$ it is easy to deduce that every element of $\im(H)$ not in $Z$ is non-nilpotent: any plane $U \neq Z$ must be of type U meaning that the nilpotent line $U \cap Z$ is the \emph{only} nilpotent line in $U$.

Now again, algebras of this shape are possible, but they cannot occur as subalgebras of $O$. Again we prove this using Table \ref{multtab}. Let $U$ be a type U subspace of $H$ and $z \in Z$ in $H \backslash U$ so that $\im(H) = Fz \oplus U$. By Lemma \ref{ZUtrans} there is an automorphism $\phi$ of $O$ such that $\phi(U) = \spam(p_1 - p_2, q_1)$ where $p_1, p_2, q_1$ are as in Table \ref{multtab}. We compute the possible locations of $\phi(z)$.

Since $p_1 - p_2$ is not nilpotent we have that $\spam(p_1 - p_2, \phi(z))$ is of type U with $Fz$ being its unique nilpotent line. It follows that $\phi(z)$ is an eigenvector for the operator $L_{p_1 - p_2}$ (left multiplication by $p_1 - p_2$) acting on $O$. Since $(p_1 - p_2)^2 = 1$ the only eigenvalues of $L_{p_1 - p_2}$ are $1$ and $-1$ and we compute from Table \ref{multtab} that $L_{p_1 - p_2}$ has four-dimensional $+1$-eigenspace $\spam(p_1, q_1, q_2, q_3)$ and four-dimensional $-1$-eigenspace $\spam(p_2, r_1, r_2, r_3)$. Since $\phi(z)$ must belong to one of these but also must belong to $\im(O)$ we find that $\phi(z) \in \spam(q_1, q_2, q_3) \cup \spam(r_1, r_2, r_3)$. On the other hand, since $q_1$ is nilpotent and hence $q_1 \in \phi(Z)$ we find that $\phi(z) \in \ker L_{q_1} \cap \im(O)= \spam(q_1, r_2, r_3)$ where the latter equality can be read off from Table \ref{multtab}. The intersection of this space with the $+1$-eigenspace of $L_{p_1 + p_2}$ is the one dimensional line $Fq_1$ and since $q_1$ and $\phi(z)$ are linearly independent, $\phi(z)$ can not lie in this space. It follows that $\phi(z)$ is an eigenvector of $L_{p_1 - p_2}$ of eigenvalue $-1$.

Now we finally arrive at our contradiction. Since the elements $q_1$ and $\phi(z)$ are both elements of $\phi(Z)$ we have that their sum is an element of $\phi(Z)$ as well. This vector is the sum of a nonzero eigenvector of $L_{p_1 - p_2}$ at eigenvalue $1$ and a non-zero eigenvector of $L_{p_1 - p_2}$ at eigenvalue $-1$ and hence cannot be an eigenvector of $L_{p_1 - p_2}$ itself. It follows that the product of $p_1 - p_2$ and $q_1 + \phi(z)$ lies outside the subspace $U' \coloneqq \spam(p_1 - p_2, q_1 + \phi(z))$ and hence that $U' \subset \im(\phi(H))$ is of type $D$. But since $\phi$ is an automorphism so is $\phi^{-1}$ and we find that the subspace $\phi^{-1}(U') \subset \im(H)$ is of type $D$ as well, contradicting the presumed nature of $H$.

\end{proof}

Proposition \ref{QMDJ} shows that the set $\mathcal{B}$ is mininimal with respect to the partial ordering on the set of all $(7, 3, 2)$-$q$-covering designs given by inclusion: every $B \in \mathcal{B}$ contains a two-dimensional subspace of type either $Q, D, M$ or $J$ and since subspaces of those types cannot be contained in any other 4-dimensial subalgebra than the one they generate we see that no proper subcollection of $\mathcal{B}$ can be a $(7, 3, 2)$-$q$-covering design as well. We thus obtain the main result of the paper, the following, slightly improved, version of Theorem \ref{main1}.

\begin{theorem}\label{main2}
Let $O$ be a Cayley algebra over $F$, $V = \im(O)$ and 
$$\mathcal{B} = \{\im(H) \colon H \textrm{ is a 4-dimensional subalgebra of } O\}.$$ 
Then $(V, \mathcal{B})$ is a $(7, 3, 2)$-$q$-covering design which is always a `local minimum' in the sense that no proper subcollection of $\mathcal{B}$ covers all 2-dimensional subspaces of $V$ and which is a `global minimum', in the sense of being a $2$-$(7, 3, 1)$-subspace design if and only if $O$ is a division algebra.
\end{theorem}

%\section{The relation to associativity}\label{associative}
\section{Relation to associativity}\label{associative}
Another consequence of Proposition \ref{QMDJ} and Theorem \ref{classification} which is not directly related to the $q$-covering design is the following.
\begin{corollary}\label{4ass}
Every 4-dimensional subalgebra of a Cayley algebra is associative.
\end{corollary}
In particular this means that the blocks in the $q$-cover design of Theorem \ref{main2} are examples of the following structure.
\begin{definition}
Let $A$ be an algebra. A subspace $V$ of $A$ will be called \emph{associative} when $(xy)z = x(yz)$ for all $x, y, z \in V$, regardless of whether the products $xy, xyz$ and $yz$ lie inside or outside $V$. Equivalently, we say that $V$ is associative if the associatior $(. , ., .)$ (Defined in Thm. \ref{associator}) vanishes identically on $V$. 
\end{definition}
By Artin's Theorem \ref{associator}, checking associativity of three-dimensional subspaces of an alternative algebra is not that hard: by linearity and alternativity of the associator we have that:
\begin{corollary}\label{altcheck}
A three-dimensional subspace $D$ of an alternative algebra $A$ is associative if and only if it has a basis $a, b, c$ such that $(a, b, c) = 0$.
\end{corollary}
The purpose of this section is to establish the following converse to Corollary \ref{4ass}.
\begin{theorem}\label{ass4}
Let $A$ be a three-dimensional associative subspace of a Cayley algebra $O$ over $F$. Then $F1 \oplus A$ is closed under the multiplication on $O$ and hence a four-dimensional associative subalgebra of $O$.
\end{theorem}
Theorem \ref{ass4} shows that not only are the blocks $B \in \mathcal{B}$ from Theorem \ref{main2} three-dimensional associative subspaces of $O$ but also together they are \emph{all} such spaces. Hence we get the following,`cleaner' reformulation of Theorem \ref{main2}, already mentioned in the introduction.

\begin{theorem}\label{main3}
Let $O$ be a Cayley algebra over $F$, $V = \im(O)$ and let $\mathcal{B}$ the collection of three-dimensional associative subspaces of $V$. Then $(V, \mathcal{B})$ is a $2$-$(7, 3, 1)$ $q$-cover design which is `locally minimal' in the sense that no proper subcollection of $\mathcal{B}$ covers all 2-dimensional subspaces of $V$ and which is  a $2$-$(7, 3, 2)$ subspace design if and only if $O$ is a division algebra.
\end{theorem} 

The current section will also provide a proof of the following result which is not needed in the sequel, but is included because it is a beautiful result and settles an issue that bothered me for a long time.

\begin{theorem}\label{ass1234}
A subalgebra of a  Cayley algebra is associative if and only if it has dimension less than or equal to $4$.
\end{theorem}

I would not at all be surprised if Theorem \ref{ass1234} has been found before, and if indeed it is known in the literature we can derive Theorem \ref{ass4} from it with relative ease. This argument will be given at the end of the section in Remark \ref{easyroute}. However, since I never saw Theorem \ref{ass1234} in print we will first take the opposite route: giving a `direct' proof of Theorem \ref{ass4} and then deriving Theorem \ref{ass1234} from it. Our proof of Theorem \ref{ass4} requires the following two results, the second of which has already been announced in Remark \ref{special}.

\begin{lemma}\label{abc}
Let $a, b, c \in \im(O)$ be three elements satisfying $(a, b, c) = 0$. Then $(x, y, z) = 0$ for every triple $x, y, z$ where exactly one of $x, y, z$ is a product of two of the elements $a, b, c$ and the other two are taken from the triple $a, b, c$.
\end{lemma}
(The somewhat clumsy formulation seemed more effecient than listing all 243 equations.)
\begin{proof}
Since the associator alternates by Theorem \ref{associator} we may freely move around the entries in the associatior. In particular we may assume that $y$ is the product of two elements from $a, b, c$ and $x$ and $z$ are elements of $\{a, b, c\}$. If $x = z$ we have that the entire equation takes place within the algebra $\langle x, y \rangle$ which is associative by alternativity. Similarly we have that $(x, y, z) = 0$ if $y = xz$ or $y = zx$ or when $y \in \{a^2, b^2, c^2\} \subset F1$, because in all these cases the equation takes place in the associative algbera $\langle x, z \rangle$. In the remaining case exactly one factor of $y$ equals either $x$ or $z$ and by permuting the terms of the associator some more we may assume it is $x$. It follow that we may assume without loss of generality that the associator we are interested in is either $(a, ab, c)$ or $(a, ba, c)$. Finally since $ba = (ab + (-b)(-a)) - ab = (ab + b^*a^*) - ab = (ab + (ab)^*) - ab = \tau(ab) - ab$ we have that $(a, ba, c) = (a, \tau(ab), b) - (a, ab, c) = -(a, ab, c)$ and hence the only thing left to show is that $(a, ab, c) = 0$. Writing out the definition of the associatior we are comparing $l \coloneqq (a(ab))c$ to $r \coloneqq (a((ab)c))$.

Write $a^2 = \alpha 1$ with $\alpha \in F$. Now the `inner' multiplication $a(ab)$ in $l$ takes place inside the associative algebra $\langle a, b \rangle$ and hence yields $\alpha b$. Linearity of the multiplication implies that we don't have to worry about the remaining brackets and we find that $l = \alpha b c$. On the other hand we know by assumption that $(ab)c = a(bc)$ and hence $r$ can be rewritten $r = a(a(bc))$. Now this multiplication takes place inside the associative algebra $\langle a, bc \rangle$ and hence we find that $r = (aa)(bc) = \alpha bc = l$. 
\end{proof}

\begin{theorem}\label{doubling2}
Let $A$ be an alternative Cayley-Dickson algebra of dimension $2^n, n = 1, 2, 3$ and $B$ a Cayley-Dickson subalgebra of dimension $2^{n-1}$ such that the involution of $B$ is just the restriction to $B$ of the involution of $A$. Then there exists a $\gamma \in F^\times$ and $i \in \im(A)$, not contained in $B$ such that $A = B \oplus iB$ as a vector space and $(\ref{pq1}, \ref{pq2}, \ref{pq3})$ hold in $A$ for all $p, q \in B$. Consequently, $A$ is isomorphic to the Dickson double $D_\gamma(B)$. Moreover if $B$ is split we have that the same statements hold for \emph{every} $\gamma \in F^\times$. 
\end{theorem}
\begin{proof}
We have already proven this in the case that $A$ (and hence $B$) is a division algebra (Prop. \ref{doubling}), so we are left with the case that $A$ is split. We will take $\gamma$ arbitrary when $B$ is split and otherwise take $\gamma = 1$. The 8-dimensional case is the only one we need for the proof of Theorem \ref{ass4}, but we include the other two cases for the sake of completeness.

When $A$ is split, Theorem \ref{splitunique} (cf. also its Corrolary \ref{sqrt1})  implies that $A$ is isomorphic to the `abstract' Dickson double $D_\gamma(B)$ and in particular that there exists a $2^{n-1}$-dimensional subalgebra $B' \subset A$ isomorphic to $B$ and an element $i'$ not in $B'$ satisfying $(i')^2 = \gamma$ such that $A = B' + i'B'$ and such that multiplication is governed by relatations (\ref{pq1}, \ref{pq2}, \ref{pq3}). Now if $\dim A = 2$ we have that $\dim B = 1$ and hence $B' = B = F1$ by uniqueness of the one-dimensional Cayley-Dickson algebra. At the other extreme, if $\dim A = 8$, we have by Theorem \ref{transitive} that there is an automorphism $\phi$ of $A$ mapping $B'$ to $B$ and setting $i \coloneqq \phi(i')$ we have that $A = B + iB$ with multiplication again governed by relatations (\ref{pq1}, \ref{pq2}, \ref{pq3}). 

We are left with the intermediate case that $\dim A = 4$. We know that we can identify $A$ with $\Mat(2, F)$ (Prop. \ref{Mat2F}). We'll immitate the proof of the $(\dim A = 8)$-case, replacing Theorem \ref{transitive} with more familiar results from linear algebra. Let $\beta$ be such that $B \isom D_\beta(F)$ and let $j \in B$ be the imaginary squareroot of $\beta$ used in the construction of $B$ from $F$. Let $i'$, $B'$ be as before and let $j' \in B$ be the imaginary squareroot of $\beta$ used in the doubling step to obtain $B'$ from $F1$. Since $B'$ and $B$ are two-dimensional and $j^2 = (j')^2 = \beta1$ we have that any linear map $\phi \colon A \to A$ sending $1$ to $1$ and $j'$ to $j$ is an algebra isomorphism from $B'$ to $B$. What we need to finish the proof is a map $\phi$ that does this while at the same time being an algebra automorphism of $A$.

From Proposition \ref{Mat2F} we see that the characteristic polynomial of $j$ equals $\det(j - \lambda I) = n(j - \lambda 1)$ which by (\ref{n}) and (\ref{im2}) can be written $(j - \lambda1)(-j - \lambda1) = \lambda^2 - \beta$. For the same reason the characteristic polynomial of $j'$ equals $\lambda^2 - \beta$ as well. Now the roots of this polynomial (in a sufficently large field extension $F'$ of $F$ containing them) are distinct. Hence we don't have to worry about Jordan blocks and can conclude that there exists a matrix $g \in \Mat(2, F')$ such that $g^{-1}j'g = j$. By the standard argument we can choose $g$ in such a way that all its coefficients are from $F$ and hence $g \in \Mat(2, F)$. Now the map $x \mapsto g^{-1}xg$ is clearly an algebra automorphism of $\Mat(2, F)$ and hence we can take this map as our desired automorphism $\phi$. 
\end{proof}

\begin{proof}[{P}roof of Theorem \ref{ass4}]
Let $A \subset \im(O)$ be an associative subspace. We fall back on the case distinction of Theorem \ref{classification}. If $ab \in F1 \oplus A$ for every $a, b \in A$ we have that $F1 \oplus A$ is a subalgebra and there is nothing left to prove. Otherwise we can pick $a, b$ such that $ab \nin F1 \oplus A$ which in particular means that $ab \nin \spam(a, b)$. It follows that $\spam(a, b)$ is not of type U or Z and hence, by Lemma \ref{complete}, of type either Q, D, M or J in the terminology of Section \ref{split}. (In the case that $O$ is a division algebra, we can skip the previous two sentences and just recall that every two-dimensional subspace of $\im(O)$ is of type Q.) The cases where $A$ contains a two-dimensional subspace $T$ of type either Q or M will be treated first. Then we will consider the case where $A$ contains a space of type $J$ and finally the case where $A$ contains neither, but does contain a type D space.

In the case where $A$ contains a subspace $T$ of type either Q or M the subalgebra $H$ generated by $T$ is a four-dimensional Cayley-Dickson algebra. (See Thm \ref{classification}.) 

From Theorem \ref{doubling2} we have the decomposition 
\begin{equation}\label{HiH}
O = H + iH.
\end{equation} 
If $T$ is of type Q, we see from Proposition \ref{main} that we can pick a basis $a, b$ of $T$ such that $\{a, b, ab\}$ is a basis of $\im(H)$ and if $T$ is of type M we obtain the same conclusion from Proposition \ref{typeM}.

Let $c \in A$. By (\ref{HiH}) we can write $c = p + iq$ with $p, q \in H$. Since $A$ is an associative subspace we have that $(a, b, c) = 0$. It then follows from Lemma \ref{abc} that $(x, y, c) = 0$ for every $x, y \in \{a, b, ab\}$ and from there that $(x, y, c) = 0$ for all $x, y \in H$. (The case that both $x$ and $y$ equal $ab$ is not covered by Lemma \ref{abc} but follows directly from alternativity of $O$.) For reasons that will become clear soon we note in particular that $(r^*, s^*, c) = 0$ for all $r, s \in H$. At the same time, because $H$ is an associative algebra, we have that $(r^*, s^*, p) = 0$ for all $r, s \in H$. It follows that $(r^*, s^*, iq) = (r^*, s^*, c) - (r^*, s^*, p) = 0$ for all $r, s \in H$. Writing out the definition of the associator $(r^*, s^*, iq)$ and applying (\ref{pq1}) three times we obtain that $((sr - rs)q)i = 0$, and hence 
\begin{equation}\label{qnul}
(sr - rs)q = 0 \textnormal{ for all } r, s \in H.
\end{equation} 
This is an equation that takes place entirely inside the four-dimensional Cayley-Dickson algebra $H$. When $H$ is a division algebra, the last equation immediately implies $q = 0$. Before discussing the consequences of that fact for the proof of Propostion \ref{ass4}, we will first show that $(\ref{qnul})$ \emph{also} implies that $q = 0$ in the case that $H$ is split, hence isomorphic to the algebra of two-by-two matrices over $F$ (Proposition \ref{Mat2F}). 

In that scenario, equation (\ref{qnul}) states that $R_q(x) = 0$ for every $x \in H \isom \Mat(2, F)$ expressible in the form $x = sr - rs$. Viewing $q$ as a non-zero two-by-two matrix, it must have determinant 0 owing to the fact that there exist two-by-to matrices with non-zero determinant expressible in the form $sr - sr$. An example is the matrix $h$ given below. 

When we think of the matrix $q$ as acting on row vectors by right multiplication we know from $\det(q) = 0$ that either $q = 0$ or $\dim \ker(q) = 1$. In the latter case it follows that $\dim((\ker R_q) \cap H) = 2$: the only matrices mapped to zero by $R_q$ are those both whose rows are vectors in the one-dimensional space $\ker(q)$. On the other hand, the subpace of $\Mat(2, F)$ of matrices expressible in the form $(sr - rs)$ contains a basis of the set of all traceless 2-by-2-matrices: the famous $\mathfrak{sl}_2(F)$-triple $h = \begin{pmatrix}
1 & 0 \\ 0 & -1
\end{pmatrix}$, $x = \begin{pmatrix}
0 & 1 \\ 0 & 0
\end{pmatrix}$, $y = \begin{pmatrix}
0 & 0 \\ 1 & 0
\end{pmatrix}$, satisfying $h = xy - yx , x = \frac{1}{2}hx - \frac{1}{2}xh, y = \frac{1}{2}yh - \frac{1}{2}hy$. This shows that $\ker(R_q) \cap H$ is at least three-dimensional, which rules out the case that $q \neq 0$. We conclude that $q = 0$ both in the case that $H$ is split and in the case that $H$ is a division algebra.

Now recalling the definition of $q$ we see that $q = 0$ implies that $c \in H$. But since $c$ was a generic element of $A$ this means that $A \subset H$ and hence, since $H$ is an algebra, that $\langle A \rangle \subset H$. By dimension considerations it then follows that $\langle A \rangle = F1 \oplus A$ as we wanted to show.

We move on to the case that $A$ contains a subspace $T$ of type $J$. This case can be handled in a more hands-on way, owing to the fact that the isomorphism class of the algebra $H \coloneqq \langle T \rangle$ is completely determined by Proposition \ref{typeJ}. In particular we can choose elements $u, v \in T, w = uv \in H$ that behave as in that proposition. With notation as in Table \ref{multtab}, let $\phi \colon H \to \spam(1, q_1, q_2, r_3) \subset H$ be the linear map sending $1 \mapsto 1, u \mapsto q_1, v \mapsto q_2, w \mapsto r_3$. Then by Theorem \ref{transitive}, $\phi$ extends to an automorphism of all of $O$. Let $a \in A$. By definition of associative space we have that $(u, v, a) = 0$ and hence we find that $0 = (\phi(u), \phi(v), \phi(a)) = (q_1, q_2, \phi(a)) = (L_{q_1q_2} - L_{q_1}L_{q_2})(\phi(a))$. In other words: $\phi(a)$ lies in the kernel of the linear map $(L_{r_3} - L_{q_1}L_{q_2})$. We can write down this map rather explicitely: the matrix representing this map with respect to the basis of Table \ref{multtab} equals:
\begin{equation}\label{assq1q2}
\begin{pmatrix}
0 &  0 & 0 & 0 &  1 &  0 &  0 & 0 \\
0 &  0 & 0 & 0 & -1 &  0 &  0 & 0 \\
0 &  0 & 0 & 0 &  0 &  0 &  0 & 0 \\
0 &  0 & 0 & 0 &  0 &  0 &  0 & 0 \\
0 &  0 & 0 & 0 &  0 &  0 &  0 & 0 \\
0 &  0 & 0 & 0 &  0 &  0 & -1 & 0 \\
0 &  0 & 0 & 0 &  0 &  1 &  0 & 0 \\
1 & -1 & 0 & 0 &  0 &  0 &  0 & 0 
\end{pmatrix}.
\end{equation}
It is not hard to see that this matrix has four-dimensional kernel, which then, by the fact that $\spam(1, q_1, q_2, r_3)$ is an associative subalgebra of $O$ must mean that $\ker(L_{q_1q_2} - L_{q_1}L_{q_2}) = \spam(1, q_1, q_2, r_3)$. Of course, this latter fact can also be seen directly by staring at the above matrix. Either way, we conclude that $\phi(a) \in \spam(1, q_1, q_2, r_3)$ and hence $a \in \phi^{-1}(\spam(1, q_1, q_2, r_3)) = H$. But since $a$ was a generic element of $A$ we have that $A \subset H$ and since $H$ is a subalgebra we find that $\langle A \rangle \subset H$. It follows that $\langle A \rangle$ is 4-dimensional and hence equal to $F1 \oplus A$ as we wanted to show.

The only case left from the case distinction at the beginning of the proof is the case where $A$ does not have any two-dimensional subspaces of types Q, M or J but does have a subspace of type D. Absense of type Q subspaces means that the three-dimensional space $A$ must contain at least 2 nilpotent lines. (This is an understatement of course). Lemma \ref{complete} states that every two-dimensional subspace containg more than two nilpotent lines consist entirely of nilpotent elements. Absense of type M spaces then implies that $A$ must contain at least one two-dimensional subspace $Z$ consisting entirely of nilpotent elements, and furthermore that the elements of $A \backslash Z$ are either all nilpotent or all non-nilpotent. Existence of a type D subpace $T$ then implies that we are in the latter of the last to scenarios. Absence of type J subspaces in A implies moreover that $Z$ is of type Z.

Again we let $H = \langle T \rangle$ be the four-dimensional algebra generated by $T$, and this time we know from Proposition \ref{typeD} that it has basis $\{1, u, v, uv\}$ where $\{u, v\}$ is a basis of $T$, with $u$ nilpotent so that $Fu = T \cap Z$, while $v^2 \neq 0$ and $(uv)^2 = 0$. 

We note that the vector space $A$ decomposes as $A = Fv \oplus Z$. Hence, if $uv \in Z$ we have that $\{u, v, uv\}$ is a basis of $A$ so that $A \subset H$ and hence $\langle A \rangle = H$ as desired.  We will rule out the alternative: $uv \not\in Z$. 

Under the latter assumption, we look at the space $Z_3 = Z + Fuv$ which now is three-dimensional. Let $z \in Z \backslash T$ so that $\{z, u\}$ is a basis of $Z$ and $\{z, u, uv\}$ is a basis of $Z_3$. We know that $zu = 0$ since $Z$ is of type Z. From associativity between elements of $A$ we have that $z(uv) = (zu)v = 0v = 0$. Since $zu = 0$, $z(uv) = 0$, and $u(uv) = 0$, we have that $Z_3$ is a three-dimensional space all whose elements multiply to zero. We already saw in the proof of Proposition \ref{QMDJ} (concretely: in Lemma \ref{Z3}) that such spaces do not appear as subspaces of $O$.
\end{proof}

\begin{proof}[Proof of Thm \ref{ass1234}]
We have already done all the work on the `if' direction of the theorem: when a subalgebra $A$ is of dimension $\leq 3$ it is generated by the at most two-dimensional subspace $\im(A)$ and hence associative by alternativity. The case that $\dim A = 4$ is covered by Corollary \ref{4ass}. What remains to be done is proving `only if' direction of the theorem, i.e. the claim that any associative subalgebra of a Cayley algebra has dimension at most 4.

Aiming for a contradiction, let $\mathcal{A}$ be an associative subalgebra of a Cayley algebra $O$ such that $\dim \mathcal{A} \geq 5$. Let $A_1, A_2$ be two distinct three-dimensional subspaces of $\im(\mathcal{A})$ with two-dimensional intersection. $A_1, A_2$ are associative subspaces since they are contained in the associative algebra $\mathcal{A}$ and hence, by Theorem \ref{ass4} the spaces $F1 \oplus A_1$ and $F1 \oplus A_2$ are closed under the multiplication in $O$. It follows that the same is true for their intersection $F1 \oplus (A_1 \cap A_2)$ and hence $A_1 \cap A_2$ is of type either U or Z by Theorem \ref{classification} and Lemma \ref{complete}. When $O$ is a division algebra we know that this is impossible and the result follows. In case $O$ is split a litle bit more work is needed.

Since every two-dimensional subspace of the space $\im(\mathcal{A})$ can be realized as the intersection of two three-dimensional subspaces of that space, we find that every two-dimensional subspace of the at-least-four-dimensional space $\im(\mathcal{A})$ is of type either U or Z. We discussed \emph{three}-dimensional spaces with this property before: by Lemma \ref{Z3} they must contain exactly one type Z subspace $Z$ and for each type U subspace $U$ of such a space the unique nilpotent line in $U$ equals $U \cap Z$. 

Now let $A_1, A_2$ be two three-dimensional subspaces of $\im(\mathcal{A})$ whose intersection contains a non-nilpotent element and hence is of type U. We denote by $Z_1, Z_2$ the unique type Z subspaces of $A_1$, $A_2$ respectively. Since $A_1 \cap A_2$ is of type U we have $Z_1 \cap Z_2$ is one-dimensional: it is the unique nilpotent line in $A_1 \cap A_2$. Let $z_0 \in Z_1 \cap Z_2$ and $z_1 \in Z_1, z_2 \in Z_2$ so that $z_1 \nin Z_2$ and $z_2 \nin Z_1$. Now the space $\spam(z_1, z_2) \subset \im(\mathcal{A})$ contains at least two nilpotent lines and hence must be of type $Z$. It follows that $z_1z_2 = 0$. But since we already knew that $z_0z_1 = 0$ (as this multiplication takes place within $Z_1$) and that $z_0z_2 = 0$ (as this multiplication takes place in $Z_2$), we find that $\spam(1, z_0, z_1, z_2)$ is a subalgebra of $O$ of the type forbidden by Lemma \ref{Z3}.
\end{proof}

\begin{remark}
We saw in Section \ref{split} that a split Cayley algebra $O$ contains exactly two isomorphism classes of three-dimensional subalgebras while a Cayley division algebra contains none: any such subalgebra $A$ is of the form $F1 \oplus \im(A)$ with the two-dimensional subspace $\im(A)$ being of either type U or type Z, which then fully determines the algebra structure. Now if $\im(A)$ is of type U we know (cf the beginning of Section \ref{qcover}) that $A$ is contained in a subalgebra $B$ of $O$ isomorphic to the split Cayley-Dickson algebra of dimension 4. Applying Theorem \ref{doubling2} to $B$ we then conclude that $O$ contains a six-dimensional subalgebra isomorphic to the Dickson double $D_\gamma(A)$ for every $\gamma \in F^\times$. A natural question is whether the same is true when $\im(A)$ is of type Z.

Theorem \ref{ass1234} tells us that the answer is no: since in that case $A$ is commutative, the six-dimensional algebra $D_\gamma(A)$ is associative by Albert's theorem \ref{Albert} and hence not contained in $O$.
\end{remark} 
\begin{remark}\label{easyroute}
As remarked at the beginning of the current section, Theorem \ref{ass4} can be derived from Theorem \ref{ass1234}. In fact it follows directly from the combination of Theorem \ref{ass1234} and the following result which can easily be proven by induction using Lemma \ref{abc}: 
\end{remark}
\begin{corollary}[to Lemma \ref{abc}]
Every subalgebra of an alternative algebra generated by an associative subspace is associative.
\end{corollary}

%\section{Geometric reformulation of the main results}\label{geometric}
\section{Geometric reformulation of the main results}\label{geometric}
We are now in the position to derive the `octonion free' formulation of Theorems \ref{main2} and \ref{main3} given (for the special cases that $F \subset \mathbb{R}$ or $F$ is finite) as Theorem \ref{mainintro} in the introduction. We recall
\begin{definition}
Let $n, k \in \mathbb{N}$ and let $V$ be an $n$-dimensional vector space over $F$. The \emph{Grasmannian} $\Gr_{k}(V)$ is defined as the set of all $k$-dimensional subspaces of $V$. Let $W \coloneqq \bigwedge^k V$ and let $\mathbb{P}(W)$ its projective space, i.e. the set of lines in $W$. The \emph{Pl\"{u}cker embedding} $\psi \colon \Gr_{k}(V) \to \mathbb{P}(W)$ is the map that takes a $k$-subspace $A$ and maps it to the line $F (a_1 \wedge \ldots \wedge a_k)$ where $\{a_1, \ldots, a_k\}$ is any basis of $A$. It is well-defined (if $\{b_1, \ldots, b_k\}$ as any other basis of $A$ then $b_1 \wedge \ldots \wedge b_k = \Delta (a_1 \wedge \ldots \wedge a_k)$ where $\Delta$ is the determinant of the unique linear transformation mapping the one basis to the other) and moreover injective, giving $\Gr_{k}(V)$ the structure of projective algebraic variety.
\end{definition}
The proof of injectivity can be found in many places, including Wikipedia.

Our concern will be with the Grasmanian $\Gr_{3}(V)$ of 3-dimensional subspaces of the seven-dimensional space $V = \im(O)$ for $O$ a Cayley algebra. The set of blocks $\mathcal{B}$ from Theorems \ref{main2} and \ref{main3} turns out be a projective subvariety of $\Gr_{3}(V)$ that can be described very explicitly in the above realization of $\Gr_{3}(V)$ as a subvariety of $\mathbb{P}(W)$ where $W = \bigwedge^3(V)$. The crucial observation is that $V = \im(O)$ already has an interesting and natural relationship with the space $W = \bigwedge^3 V$ thanks to alternativity of $O$: since the associator $(. , . , . )$ is tri-linear and alternating (see Theorem \ref{associator}) the universal property of the $\bigwedge^3$-functor implies that it factors over a linear map $\Ass \colon W \to O$, which, by Lemma \ref{tauabc} below takes values in $\im(O)$ and hence can be viewed as a linear map $\Ass \colon W \to V$. In more down-to-earth terms:

\begin{lemma}\label{tauabc}
Let $O$ be a Cayley algebra, $V = \im(O)$ and $W = \bigwedge^3 V$. Then the map $\Ass \colon W \to V$ given by $a \wedge b \wedge c \mapsto (ab)c - a(bc)$ is well-defined.
\end{lemma}
\begin{proof}
In light of Artin's Theorem \ref{associator}, it suffices to show that $(a, b, c) \in \im(O)$ for all $a, b, c \in \im(O)$.

Using the properties of the $*$ operator (equation (\ref{star}) and linearity) we obtain

\begin{equation}\label{abcstar}
(a, b, c)^* = ((ab)c)^* - (a(bc))^* = c^*(ab)^* - (bc)^*a^* = c^*(b^*a^*) - (c^*b^*)a^*
\end{equation}
Now since $a, b, c \in \im(O)$ we have $a^* = -a, b^* = -b, c^* = -c$ by (\ref{im2}) so that (\ref{abcstar}) reduces to
\begin{equation}\label{abccba}
(a, b, c)^* = -c(ba) + (cb)a = (c, b, a)
\end{equation}
But $(c, b, a) = -(a, b, c)$ by Artin's Theorem \ref{associator} so that (\ref{abccba}) reads 
\begin{equation}
(a, b, c)^* = -(a, b, c).
\end{equation}
This means, by (\ref{im2}), that $(a, b, c) \in \im(O)$ as we wanted to show.
\end{proof}

The following observation is trivial, but no less useful:

\begin{lemma}
Let $w \in W$ then $\Ass(w) = 0$ if and only if $\Ass(w') = 0$ for every $w'$ on the line $Fw \subset W$.
\end{lemma}
Finally we note: 
\begin{lemma}\label{K}
Let $O, V, W$ be as in Lemma \ref{tauabc} and let $\psi \colon \Gr_{3}(V) \to \mathbb{P}(W)$ be the Pl\"{u}cker embedding. The associative subspaces among the three-dimensional subspaces of $V$ are precisely the spaces $x \in \Gr_{3}(V)$ for which the line $\psi(x)$ lies entirely inside the space $K \coloneqq \ker(\Ass)$.
\end{lemma}
Combining this with Theorem \ref{main3} and the fact that over every field $F$ of characteristic not 2 at least one Cayley algebra exists (Section \ref{CayleyDickson}) we obtain the following `algebra-free' version of one of our main results:
\begin{proposition}\label{geometicqcover}
Let $V$ be a seven-dimensional space over a field $F$ of characteristic $\neq 2$, let $W \coloneqq \bigwedge^3 V$, $\pi \colon W \backslash \{0\} \to \mathbb{P}(W)$ be the projection map and $\psi \colon \Gr_{3}(V) \to \mathbb{P}(W)$ be the Pl\"{u}cker embedding. Then there exists an explicitly computable 28-dimensional linear subspace $K \subset W$ such that the set $\mathcal{B} \coloneqq \psi(\Gr_{3}(V)) \cap \pi(K \backslash \{0\})$ is a $q$-covering design which is minimal with respect to the inclusion ordering on the set of $q$-covering designs in $\Gr_{3}(V)$.
\end{proposition}
And combining Lemma \ref{K} with Theorem \ref{main3} and Serre's Theorem \ref{Serre2} we obtain the following, even nicer, algebra-free result:
\begin{proposition}\label{geometicqcover2}
Let $V$ be a seven-dimensional space over a (necessarily infinite) field $F$ of characteristic $\neq 2$, satisfying $H^3(F, \mathbb{Z}/2\mathbb{Z}) \neq 0$. Let $W \coloneqq \bigwedge^3 V$, $\pi \colon W \backslash \{0\} \to \mathbb{P}(W)$ be the projection map and $\psi \colon \Gr_{3}(V) \to \mathbb{P}(W)$ be the Pl\"{u}cker embedding. Then there exists an explicitly computable 28-dimensional linear subspace $K \subset W$ such that the set $\mathcal{B} \coloneqq \psi(\Gr_{3}(V)) \cap \pi(K \backslash \{0\})$ is a $q$-Fano plane.
\end{proposition}

\section{Quantizing the Fano plane: combinatorial reformulation of the main results}\label{combinatorics}

We will now live up to the claim that the space $K$ is explicitly computable by explicitly computing it. More precisely we'll derive a more general version of Theorem \ref{mainintro} from Section \ref{introresults} (Theorem \ref{maincombinatorially} below), constructing the $q$-covering designs and $q$-Fano planes of Theorem \ref{main1} from an ordinary Fano plane, a field $F$, a choice of three elements $\alpha, \beta, \gamma$ in $F$ and a choice of an order 7 autormorphism $\phi$ and a point $v_0$ of the Fano-plane. The first two claims of Theorem \ref{mainintro} (that the construction given there yields a $q$-covering design which is a $q$-Fano plane when $F \subset \mathbb{R}$) follow from Theorem \ref{maincombinatorially} by setting $\alpha = \beta = \gamma = -1$. The last claim of Theorem \ref{mainintro} (giving the number of blocks in the design when $F$ is finite) will be proven in Section \ref{counting}.

\begin{theorem}\label{maincombinatorially} Let $F$ be field of characteristic $\neq 2$, let $\alpha, \beta, \gamma \in F^\times$ and let $(\mathcal{V}, \mathcal{L})$ be a Fano-plane with vertex set $\mathcal{V}$ and line set $\mathcal{L} \subset \mathcal{V}^3$. We fix an automorphism $\phi$ of $(\mathcal{V}, \mathcal{L})$ so that the cyclic group generated by $\phi$ acts transitively on $\mathcal{V}$ (and hence $\mathcal{L}$) and choose a cyclic ordering on the three elements of each $l \in \mathcal{L}$ in such a way that application of $\phi$ will preserve the orderings on the lines. Finally we pick a `special' point $v_0 \in \mathcal{V}$ and label the points and lines in $\mathcal{L}$ with elements of $\ZZ/7\ZZ$ as follows (cf Figure \ref{fanoplot} (a)):
if $\{v_0, \phi(v_0), \phi^3(v_0)\} \in \mathcal{L}$ then we define $v_n \coloneqq \phi^n(v_0)$ for $n \in \ZZ/7\ZZ$, if $\{v_0, \phi(v_0), \phi^3(v_0)\} \nin \mathcal{L}$ then we define $v_n \coloneqq \phi^{-n}(v_0)$ for $n \in \ZZ/7\ZZ$. Note that in both cases we have for each $n \in \ZZ/7\ZZ$ that $l_n \coloneqq \{v_n, v_{n+1}, v_{n+3}\} \in \mathcal{L}$ and that every $l \in \mathcal{L}$ is of this form. We note moreover that $v_0 = l_{-1} \cap l_0$, $v_1 = l_0 \cap l_1$, $v_2 = l_1 \cap l_{-1}$. 

Now as in Theorem \ref{mainintro} we define $V$ to be the $F$-vector space with basis $\mathcal{V}$ and define $W = \bigwedge^3 V$. Again we set $\Delta = v_0 \wedge v_1 \wedge v_2 \ldots \wedge v_6 \in \bigwedge^7 V \isom F$. Unlike in Theorem \ref{mainintro} we moreover define a function $h \colon \mathcal{L} \to F$ by $h(l_1) = -\alpha$, $h(l_{-1}) = -\beta, h(l_0) = -\gamma$ and $h(l) = 1$ for $l \in \mathcal{L} \backslash \{l_{-1}, l_0, l_1\}$. Now, for each $l \in \mathcal{L}$ let $w_l = h(l) u_1 \wedge u_2 \wedge u_3 \in W$ where $u_1, u_2, u_3$ are the elements of $l$ in the given cyclic order. (Note that this is well defined since 3 is an odd number and that the contribution of the function $h$ is invisible in the special case that $\alpha = \beta = \gamma = -1$.) 

These data in turn define, for every $v \in \mathcal{V}$, a linear functional $\eta_v \colon W \to F$ by $v \wedge (\sum_{l \in \mathcal{L}} w_l) \wedge w = \eta_v(w)\Delta$.

Let $\Gr_3(V)$ be the set of three-dimensional subspaces of $V$ and $\Gr_1(W)$ the set of one-dimensional subspaces of $W$. Let $\psi \colon \Gr_3(V) \to \Gr_1(W)$ be the Pl\"{u}cker embedding. 

Then the set $\mathfrak{B} = \{B \in \Gr_3(V) \colon \psi(B) \subset \ker \eta_v \textnormal{ for all } v \in \mathcal{V} \}$ is an inclusion minimal $q$-covering design with parameters $(7, 3, 2)$ on $V$. %That is: every two-dimensional subspace of $V$ is contained in at least one element of $\mathfrak{B}$ and no proper subset of $\mathfrak{B}$ has this property. 
Moreover, if $F, \alpha, \beta, \gamma$ are such that the eight-dimensional Cayley-Dickson algebra $D_\gamma(D_\beta(D_\alpha(F)))$ is a division algebra, then $\mathfrak{B}$ is a $q$-Fano plane.%, that is: in that case we have that every two-dimensional subspace of $V$ is contained in \emph{exactly one} element of $\mathfrak{B}$. 
\end{theorem}
\begin{remark}
If we want an entirely algebra-free statement we can replace the condition `if $F, \alpha, \beta, \gamma$ are such that the eight-dimensional Cayley-Dickson algebra $D_\gamma(D_\beta(D_\alpha(F)))$ is a division algebra' in the last sentence by the equivalent condition `if $F, \alpha, \beta, \gamma$ are such that equation (\ref{stroth1}) only has the zero-solution $\lambda_0 = \lambda_1 = \ldots =  \lambda_7 = 0$', for reasons explained just above that equation.
\end{remark}
\begin{figure}
\caption{The Fano plane with (from left to right): (a) lines $l_{-1}, l_0, l_1$ and points $v_0, \ldots v_6$ from Thm. \ref{maincombinatorially}; (b) the values of the function $s$ of (\ref{defS1}, \ref{defS2}) at those points; (c)  an example of a labeling of the points as in (\ref{efg}, \ref{abc1}, \ref{abc2}) for the special case $v = v_0$, $l_{v, i} = l_1$. \label{fanoplot}}
\includegraphics[width = \textwidth]{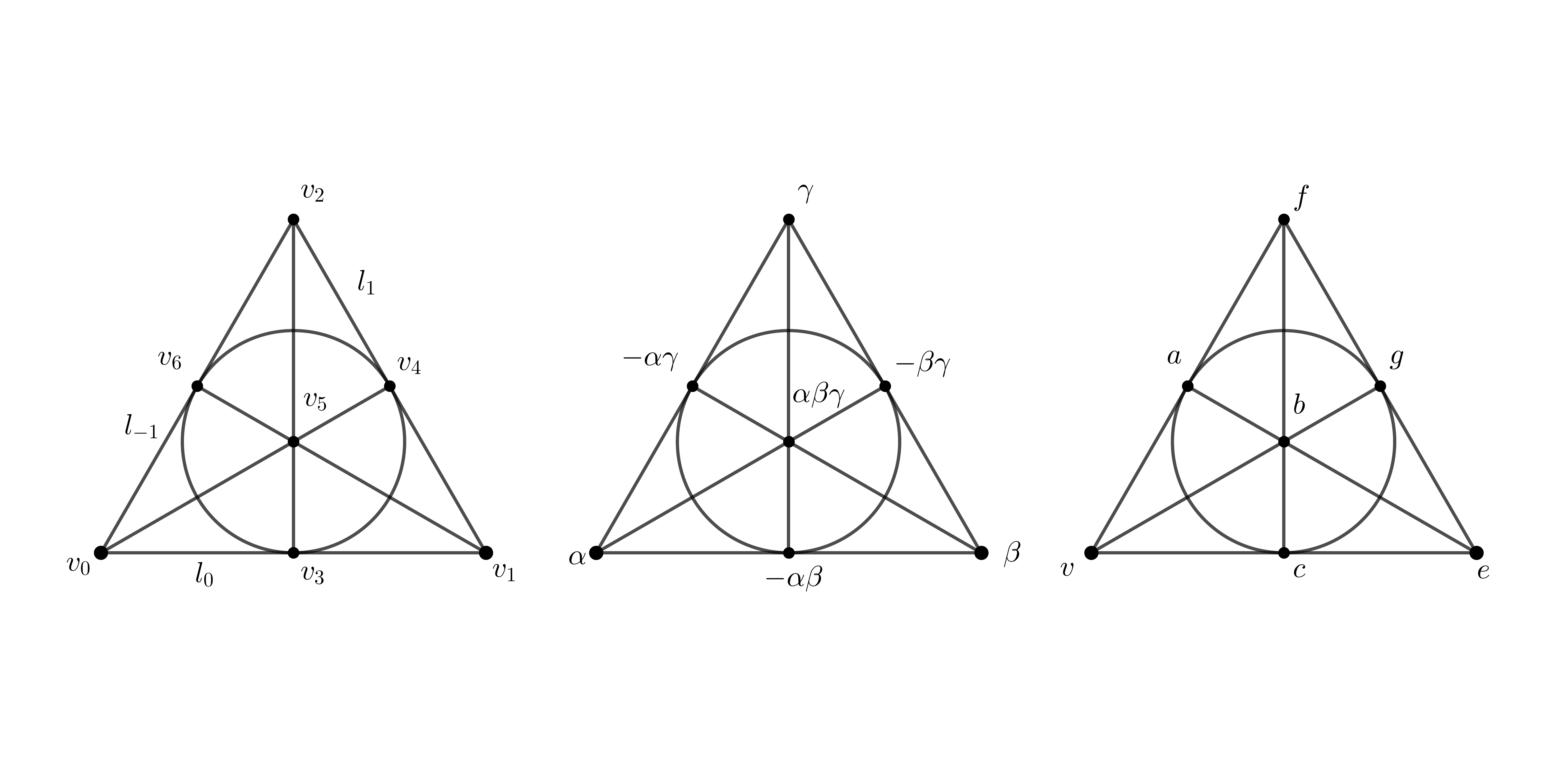}
\end{figure}

\begin{proof}
We will reduce the theorem to Theorem \ref{main3} above. The first step is then of course to construct the appropriate Cayley algebra from the data given. 

In addition to the function $h \colon \mathcal{L} \to F$ defined above as
 
\begin{align}\label{defH1}
h(l_1) &= -\alpha, \qquad h(l_{-1}) = -\beta,  \qquad h(l_0) = -\gamma \\
h(l) &= 1 \qquad \textnormal{ for } l \in \mathcal{L} \backslash \{l_1, l_{-1}, l_0\}, \label{defH2}
\end{align}
we define functions $s \colon \mathcal{V} \to F$, $r \colon \mathcal{V} \times \mathcal{V} \to F$, that together describe the multiplication in the Cayley algebra. 

Concretely, 
the function $s$ is defined by 
\begin{equation}\label{defS1}
s(v_0) = \alpha, \qquad s(v_1) = \beta, \qquad s(v_2) = \gamma 
\end{equation}
and demanding that for every $l \in \mathcal{L}$ we have that 
\begin{equation}\label{defS2}
\Prod_{v \in l} s(v) = -\frac{(\alpha\beta\gamma)^2}{h(l)^2}. 
\end{equation}
(See also figure \ref{fanoplot} (b).)

The function $r \colon \mathcal{V} \times \mathcal{V} \to F$ is defined by 
\begin{equation}\label{defR1}
r(a, b) = \frac{\alpha\beta\gamma}{h(l)s(c)}
\end{equation}
where $l \in \mathcal{L}$ is the unique line through $a$ and $b$ and $c$ is the third point on $l$. 

From here we see that
\begin{equation}\label{rsymmetric}
r(a,b) = r(b, a).
\end{equation}

Combining (\ref{defR1}) and (\ref{defS2}) we moreover find the relation
\begin{equation}\label{defR2}
r(a,b)^2 = -\frac{s(a)s(b)}{s(c)}
\end{equation}
where $c$, again, is the third point on the unique line through $a$ and $b$. (Cf Figure \ref{fanoplot} (b).) 

\begin{remark}
In the special case that $\alpha = \beta = \gamma = -1$ we have that $h$ is identically 1, $s$ is identically $-1$ and $r$ is identically $1$ again.  
\end{remark}

Before moving on to the construction of the Cayley algebra from these functions, we note for future reference the following relation between the functions $h$ and $s$ which is in a sense dual to relation (\ref{defS2}) and can be verified from Figure \ref{fanoplot}(b):

\begin{equation}\label{dooreenpunt}
\prod_{l \ni p} h(l) = \frac{\alpha\beta\gamma}{s(p)} \qquad \textnormal{for all } p \in \mathcal{V}.
\end{equation}

\begin{proposition}\label{FanotoO}
Let $\mathcal{O}$ be the eight-dimensional vector space $F1 \oplus V$. We define a multiplication on $\mathcal{O}$ making it into an algebra as defined in Definition \ref{algebra} by letting the element $1$ act as the identity and defining the product of basis elements of $V$ as follows:

\begin{itemize}
\item For each $v \in \mathcal{V}$ we set $v^2 = s(v)1$
\item For each cyclically ordered line $l = (a, b, c)$ we set $ab = r(a, b) c$ while $ba = -r(a, b) c$.
\end{itemize}

Then the algebra $\mathcal{O}$ is isomorphic to the Cayley algebra $D_\gamma(D_\beta(D_\alpha(F)))$ defined in Definition \ref{DicksonDouble}. 
\end{proposition}
\begin{proof}
Let the $v_n, l_n$ and $\phi$ be as in Theorem \ref{maincombinatorially}
Now if the orientation of the line $l_0$ reads $(v_0, v_1, v_3)$ we have that for every line $l_n$ the orientation reads $(v_n, v_{n+1}, v_{n+3})$ and we define $e_n = v_n$ for each $n = 0, \ldots, 6$. If on the other hand the orientation reads $(v_0, v_3, v_1)$ then we set $e_n = -v_n$ for each $n = 0, \ldots, 6$. In both cases we can check, unpacking the definitions of $s$ and $r$, that the given multiplication on $\mathcal{O}$ is described exactly by Table \ref{multtabgeneral} from Section \ref{preliminaries}.
\end{proof}

\begin{remark}
In view of the description of the relationship between the classical Fano-plane and the Cayley algebra $D_\gamma(D_\beta(D_\alpha(F)))$ in Section \ref{OtoFano} it was to be expected that functions $s$ and $r$ as in the definition of the multiplication in Proposition \ref{FanotoO} would appear: existence of \emph{some} scalar $s(v)$ such that $v^2 = s(v)1$ is essentially Lemma \ref{imbasis} while existence of some $r(a,b)$ such that $ab = r(a,b)c$ for each Fano line $\{a, b, c\}$ is Lemma \ref{linescalars} (where $r(a, b)$ was called $\lambda_{a, b}$). Also the origin of the relation (\ref{defR2})
is easy to understand from the perspective of the multiplication: we find on one hand that $(ab)^2 = (r(a,b)c)^2 = r(a,b)^2s(c)1$ and on the other hand, because $a, b, ab \in \im(\mathcal{O})$, that $(ab)^2 = a(ba)b = -a(ab)b = -a^2b^2 = -s(a)s(b)1$.

What is not immediately obvious is which of the two `solutions' to (\ref{defR2}), i.e. (\ref{defR1}) or (\ref{defR1}) with the right hand side multiplied with $-1$, yields the correct definition of $r$. We see by explicit computation (as in the proof of Proposition \ref{FanotoO}) that (\ref{defR1}) is the correct expression but I don't have a more conceptual explanation for that.
\end{remark}

By Proposition \ref{FanotoO}, Theorem \ref{associator} and Lemma \ref{tauabc} the associator $(.,.,.) \colon V \times V \times V \to \mathcal{O}$ defined by $(a, b, c) = (ab)c - a(bc) \in \mathcal{O}$ takes values in $V$ and is trilinear and alternating; hence the associated linear map $V \tensor V \tensor V \to V$ factors over a linear map $\Ass \colon W \to V$. In the light of Lemma \ref{K} and Theorem \ref{main3}, Theorem \ref{maincombinatorially} reduces to the equality 
\begin{equation}\label{etavsass}
\bigcap_{v \in \mathcal{V}} \ker \eta_v = \ker \Ass. 
\end{equation}

We will derive this equality from a number of intermediate results, starting with a closer look at the functionals $\eta_v$ defined in theorem \ref{maincombinatorially}.

\begin{lemma}\label{eta}
Let $v \in \mathcal{V}$ and let $w = a \wedge b \wedge c \in W$ where $a, b, c \in \mathcal{V}$. Then:
\begin{itemize}
\item if $v \in \{a, b, c\}$ then $\eta_v(w) = 0$
\item if $\{a, b, c\} \in \mathcal{L}$ or $\{a, b, v\} \in \mathcal{L}$ or $\{a, v, c\} \in \mathcal{L}$ or $\{v, b, c\} \in \mathcal{L}$ then $\eta_v(w) = 0$.
\item if $\{a, b, c, v\}$ is a four element subset of $\mathcal{V}$ containing no lines then $l' \coloneqq \mathcal{V} \backslash \{a, b, c, v\}$ is a line (i.e. an element of $\mathcal{L}$) and either $\eta_v(w) = h(l')$ or $\eta_v(w) = -h(l')$. 
\end{itemize} 
\end{lemma}
\begin{proof}
The first of these cases is straightforward. For $w$ of this form we have $v \wedge (\sum w_l) \wedge w = 0$. Since $\Delta \neq 0$ we must have that $\eta_v(w) = 0$.

Let $l' \in \mathcal{L}$ be the line mentioned in the second bullet point and let $d \in V$ be the unique element of $\{a, b, c, v\} \backslash l'$. Rearranging we find that $\eta_v(w)\Delta = \pm \frac{1}{h(l')} d \wedge (\sum w_l) \wedge w_{l'} = \pm \frac{1}{h(l')} \sum_{l \in \mathcal{L}} d \wedge w_l \wedge w_{l'}$. Each term in this sum is zero, owing to how in the Fano plane every two lines have non-empty intersection.  It follows that the sum itself is zero and hence that $\eta_v(w) = 0$.

Finaly let $l'$ be as in the third case. Then every $l \neq l'$ has non-empty intersection with $\{a, b, c, v\}$ and hence we find $v \wedge (\sum w_l) \wedge w = v \wedge w_{l'} \wedge w = \pm h(l')\Delta$ and the result follows. 
\end{proof} 
 
\begin{corollary}\label{wvi}
For each $v \in \mathcal{V}$ there are exactly four elements of $W$, to be called $w_{v, 1}, w_{v, 2}, w_{v, 3}, w_{v, 4}$ simultaneously satisfying:
\begin{itemize}
\item $w_{v, i}$ is of the form $\zeta a \wedge b \wedge c$ for some $a, b, c \in \mathcal{V}$ and $\zeta \in F^\times$.
\item $\eta_v(w_{v, i}) = 1$
\end{itemize} 
Moreover, for $u \neq v \in \mathcal{V}$ we have 
$$\spam(w_{u, 1}, w_{u, 2}, w_{u, 3}, w_{u, 4}) \cap \spam(w_{v, 1}, w_{v, 2}, w_{v, 3}, w_{v, 4}) = \{0\} \subset W.$$
\end{corollary}
Concretely, we see that if $\zeta a \wedge b \wedge c$ is one of the $w_{v, i}$ for given $v \in \mathcal{V}$ then $\{a, b, c, v\}$ is a four element set containing no lines and $\zeta = \pm 1/h(l)$ where $l = \mathcal{V} \backslash \{a, b, c, v\} \in \mathcal{L}$ and the sign depends on the order of $a, b$ and $c$.

\begin{corollary}\label{Wbasis}
The 7 elements $w_l$, (one for each $l \in \mathcal{L}$) defined in Theorem \ref{maincombinatorially} together with the 28 elements $w_{v, i}$, (on for each $v \in \mathcal{V}, i = 1, 2, 3, 4$) defined in the last corollary form a basis of the 35-dimensional space $W$.
\end{corollary}

\begin{corollary}[to Lemma \ref{eta}]
By the last corollary every element $w \in W$ can be written as $w = \sum_{l \in \mathcal{L}} \xi_l w_l + \sum_{v \in \mathcal{V}} \sum_{i=1}^4 \xi_{v, i} w_{v, i}$ for unique scalars $\xi_l, \xi_{v, i} \in F$. Let $u \in \mathcal{V}$. Then 
\begin{equation}\label{kereta}
\ker \eta_u = \{\sum_{l \in \mathcal{L}} \xi_l w_l + \sum_{v \in \mathcal{V}} \sum_{i=1}^4 \xi_{v, i} w_{v, i} \in W \colon \xi_{u, 1} + \xi_{u, 2} + \xi_{u, 3} + \xi_{u, 4} = 0 \}.
\end{equation}
\end{corollary}
\begin{corollary}\label{kerofalletas}
$\bigcap_{v \in \mathcal{V}} \ker \eta_v =  \{\sum_{l \in \mathcal{L}} \xi_l w_l + \sum_{v \in \mathcal{V}} \sum_{i=1}^4 \xi_{v, i} w_{v, i} \in W \colon \sum_{i = 1}^4 \xi_{v, i} = 0 \textnormal{ for all } v \in \mathcal{V} \}$
\end{corollary}
In order to show that this space equals $\ker \Ass$ we'll give a description (Corollary \ref{kerass} below) of this latter space in terms of the basis of Corollary \ref{Wbasis}.

\begin{lemma}
Let $a, b, c \in \mathcal{V}$ and $(. , . ,. )$ be the associator of the multiplication on $F1 \oplus V$ definined in Proposition \ref{FanotoO}. Then:
\begin{itemize}
\item If $\{a, b, c \} \in \mathcal{L}$ then $(a, b, c) = 0$
\item If $\{a, b, c \} \nin \mathcal{L}$ then $(a, b, c) = \lambda d$ for some $\lambda \in F^\times$ where $d \in \mathcal{V}$ is the unique element so that $(\mathcal{V} \backslash \{a, b, c, d\}) \in \mathcal{L}$.
\end{itemize}
\end{lemma}
%Bewijs: houden lambda hier vaag; deze wordt pas concreet in volgende lemma.
\begin{proof}
This can be checked by explicit computation. Since $\mathcal{V}$ forms a basis of $V$, the automorphism $\phi$ of $(\mathcal{V}, \mathcal{L})$ described in Theorem \ref{maincombinatorially} gives rise to a bijective linear endomorphism $\phi \colon V \to V$ which we can extend to a bijective linear endomorphism $\phi \colon \mathcal{O} \to \mathcal{O}$ by letting it act as the identity on $F1$. This linear endomorphism is not an algebra automorphism of $\mathcal{O}$ (unless $\alpha = \beta = \gamma = -1$) but it is close: for every $x, y \in O$ we have that $\phi(x)\phi(y) = \lambda \phi(xy)$ for some $\lambda \in F^\times$. This `projective automorphism' property of the map $\phi$ allows us to verify the lemma on just two triples $\{a, b, c \}$ (one with $\{a, b, c\} \in \mathcal{L}$ and one with $\{a, b, c\} \nin \mathcal{L}$) and conclude the lemma holds for all triples by transitivity of $\phi$.
\end{proof}
\begin{corollary}\label{Asslv}
$\Ass(w_l) = 0$ for all $l \in \mathcal{L}$ and for each $v \in \mathcal{V}$ and $i = 1, 2, 3, 4$ there exists a $\lambda_{v, i} \in F^\times$ such that $\Ass(w_{v, i}) = \lamba_{v, i} v$.
\end{corollary}
Now the reason behind defining the $\eta_v$ in the way we did lies in the following proposition.
\begin{proposition}\label{workhorse}
Let $\lambda_{v, i} \in F^\times$ be as in the last corollary. Then for each $v \in \mathcal{V}$ we have $\lambda_{v, 1} = \lambda_{v, 2} = \lambda_{v, 3} = \lambda_{v, 4}$.
\end{proposition}

Before giving the proof we will see how this proposition helps us in our quest to prove Theorem \ref{maincombinatorially}. Together, Corollary \ref{Asslv} and Proposition \ref{workhorse} can be summarized as: 
\begin{corollary}
$\Ass(w_l) = 0$ for all $l \in \mathcal{L}$ and for each $v \in \mathcal{V}$ there is a $\lambda_v \in F^\times$ such that for each $i \in \{1, 2, 3, 4\}$ we have that $\Ass(w_{v, i}) = \lambda_v v$. Equivalently, there exist $\lambda_v \in F^\times$, one for each $v \in \mathcal{V}$ such that:
\begin{equation}\label{Assaction}
\Ass(\sum_{l \in \mathcal{L}} \xi_l w_l + \sum_{v \in \mathcal{V}} \sum_{i=1}^4 \xi_{v, i} w_{v, i}) = \sum_{v \in \mathcal{V}} \lambda_v (\xi_{v,1} + \xi_{v,2} + \xi_{v,3} + \xi_{v,4})v
\end{equation} 
\end{corollary}
By Corollary \ref{Wbasis}, every element of $w$ can be written in the form of the argument of the $\Ass$-operator in (\ref{Assaction}), so (\ref{Assaction}) fully describes the action of $\Ass$ on $W$. It follows that:
\begin{corollary}\label{kerass}
$$\ker \Ass = \{\sum_{l \in \mathcal{L}} \xi_l w_l + \sum_{v \in \mathcal{V}} \sum_{i=1}^4 \xi_{v, i} w_{v, i} \in W \colon \sum_{i = 1}^4 \xi_{v, i} = 0 \textnormal{ for all } v \in \mathcal{V} \}$$
\end{corollary}

Comparing Corollary \ref{kerass} to Corolarry \ref{kerofalletas} we find that
$$ \bigcap_{v \in \mathcal{V}} \ker \eta_v = \ker \Ass \subset W.$$
This is equation (\ref{etavsass}) which in conjunction with Lemma \ref{K} and Theorem \ref{main3} implies Theorem \ref{maincombinatorially}, as explained above that equation.
\end{proof}

\begin{proof}[Proof of Proposition \ref{workhorse}]
We will assume without loss of generality that the orientation on the line $\{v_n, v_{n+1}, v_{n+3}\} \in \mathcal{L}$ reads $(v_n, v_{n+1}, v_{n+3})$ and show that for each $i \in \{1, 2, 3, 4\}$ and each $v \in \mathcal{V}$ we have that $\Ass(w_{v, i}) = \frac{-2\alpha\beta\gamma}{s(v)}v$. Since the right hand side does not depend on $i$, this proves the claim. When instead the orientation on $\{v_n, v_{n+1}, v_{n+3}\}$ reads $(v_{n+1}, v_{n}, v_{n+3})$, one can show with a different but highly similar proof that in that case $\Ass(w_{v, i}) = \frac{+2\alpha\beta\gamma}{s(v)}v$ for each $i \in \{1, 2, 3, 4\}$ and each $v \in \mathcal{V}$, which again does not depend on $i$ and hence suffices to prove Proposition \ref{workhorse}.

By its definition in Corollary \ref{wvi}, the element $w_{v, i}$ is a scalar multiple of a wedge product $a \wedge b \wedge c$ for some set $\{a, b, c\} \subset \mathcal{V}$ with the property that $\mathcal{V} \backslash \{a, b, c, v\}$ is a line in $\mathcal{L}$. We'll denote this line by $l_{v, i}$ and assign the labels $e, f, g$ to the elements of the line $l_{v, i}$ in such a way that
\begin{equation}\label{efg}
(e, f, g) \textnormal{ is the line } l_{v, i} \textnormal{ in the cyclic ordering given in Thm. \ref{maincombinatorially}}
\end{equation}
Note that there are three possible choices satisfying \ref{efg}. Once we have settled on one, we'll asign the labels $a, b, c$ to the elements of the set $\{a, b, c\}$ defined above by
\begin{equation}\label{abc1}
\{v, a, f\} \in \mathcal{L}; \quad  \{v, b, g\} \in \mathcal{L}; \quad \{v, c, e\} \in \mathcal{L}.
\end{equation}
Since $\{e, f, g\}$ is a line in $\mathcal{L}$ not containig $v$ we see (as in Figure \ref{fanoplot} (c)) that $a, b, c$ are well defined by (\ref{abc1}) and not equal to any of $e, f, g, v$, hence $\mathcal{V} = \{a, b, c, v, e, f, g\}$. From here we conclude that
\begin{equation}\label{abc2}
\{a, b, e\} \in \mathcal{L}; \quad  \{b, c, f\} \in \mathcal{L}; \quad \{c, a, g\} \in \mathcal{L}.
\end{equation}
Now these choices have two not immediately obvious implications.
\begin{claim1}
$a \wedge b \wedge c \wedge v \wedge e \wedge f \wedge g = \Delta$
\end{claim1}
and
\begin{claim2}
$$(a, b, e) \sim (v, c, e); \quad (b, c, f) \sim (v, a, f); \quad (c, a, g) \sim (v, b, g)$$
where the equivalence relation $\sim$ between ordered elements of $\mathcal{L}$ is defined by $l \sim r$ if either \emph{both} $l$ and $r$ are ordered according to the given cyclic ordering on the elements of $\mathcal{L}$ or \emph{neither} of them are. 
\end{claim2}
We note that all statements made up-to-and-including Claim 1 do also hold when we had assumed the cyclic ordering on the lines to be $(v_{n+1}, v_n, v_{n+3})$ rather than $(v_{n}, v_{n+1}, v_{n+3})$. However Claim 2 should be replaced by a different Claim $2'$ under that assumption.

In order to verify (under the current assumption) the two claims for a given pair $(v, i)$, it suffices to do so for only one of the three assignments of $e, f, g$ to $l_{v, i}$ satisfying (\ref{efg}): switching to any of the other two amounts to cyclically permuting the labels $e, f, g$ which, by (\ref{abc1}, \ref{abc2}) results in cyclically permuting $a, b, c$ in the same direction. It is clear that both claims are invariant under simultaneous permuations of this type.

Similarly we see that once we verified the two claims for all four pairs $(v, 1)$, $(v, 2)$, $(v, 3)$, $(v, 4)$ for a fixed element $v \in \mathcal{V}$ it follows that the claims hold for all 28 combinations $(v, i)$ as the claims are clearly invariant under application of $\phi$. 

That said, I unfortunately don't see a more elegant way of proving the claims than doing just that: verifying them for the four cases coming from a fixed $v$ (e.g. the element $v_0$ of Thm \ref{maincombinatorially}) by drawing four pictures of the type of Figure \ref{fanoplot} (c).

We now look at some consequences of the two claims. In particular we assume that $v$ and $i$ are fixed and elements $a, b, c, e, f, g$ have been chosen accordingly as above. From Claim 1 we find that $v \wedge e \wedge f \wedge g \wedge a \wedge b \wedge c = \Delta$ as well and hence that 
\begin{equation}\label{hDelta}
v \wedge w_{l_{v, i}} \wedge a \wedge b \wedge c = h(l_{v, i})\Delta. 
\end{equation}
This enables us to compute $w_{v, i}$. 

From the definition of $w_{v, i}$ in Corollary \ref{wvi} we see that 
\begin{enumerate}
\item $w_{v, i} = \zeta a \wedge b \wedge c$ with $a, b, c$ as determined by (\ref{abc1}) for some $\zeta \in F^\times$,
\item $v \wedge (\sum_{l \in \mathcal{L}} w_l) \wedge w_{v, i} = \Delta$.
\end{enumerate}
Combined these equalities read
\begin{equation}\label{zetaDelta}
\zeta(v \wedge (\sum_{l \in \mathcal{L}} w_l) \wedge a \wedge b \wedge c) = \Delta.
\end{equation}
However, since every line in $\mathcal{L}$ except for $l_{v, i}$ intersects the four element set $\{v, a, b, c\}$, we find that the left hand side of (\ref{zetaDelta}) equals $\zeta(v \wedge w_{l_{v, i}} \wedge a \wedge b \wedge c)$. Thus (\ref{zetaDelta}) and (\ref{hDelta}) imply that $\zeta = \frac{1}{h(l_{v, i})}$ and hence
\begin{equation}\label{wviabc}	
w_{v, i} = \frac{1}{h(l_{v, i})} a \wedge b \wedge c.
\end{equation}

It follows that 
\begin{equation}\label{Asswvi}
\Ass(w_{v, i}) = \frac{1}{h(l_{v, i})}(a, b, c) = \frac{1}{h(l_{v, i})}((ab)c - a(bc)).
\end{equation}
Using (\ref{abc1}, \ref{abc2}), Claim 2 and the definition of the multiplication given in Proposition \ref{FanotoO} we can compute the terms on the right hand side of (\ref{Asswvi}) as follows:
\begin{eqnarray}
(ab)c = -r(a, b)r(e, c)v  \\
a(bc) = +r(a, f)r(b, c)v
\end{eqnarray}
Substituting this into (\ref{Asswvi}) we get:
\begin{equation}\label{Asswvi2}
-\Ass(w_{v, i}) = \left(\frac{r(a, b)r(e, c)}{h(l_{v, i})}\right) v + \left(\frac{r(a, f)r(b, c)}{h(l_{v, i})}\right) v.
\end{equation}
Recall that $l_{v, i} = \{e, f, g\}$. Expanding out the coefficient of the first term on the right hand side of (\ref{Asswvi2}) using the definition (\ref{defR1}) of the function $r$ and the relation (\ref{dooreenpunt}) with $p = e$ we find:
\begin{equation}\label{Asswvi3}
\frac{r(a, b)r(e, c)}{h(l_{v, i})} = \frac{(\alpha\beta\gamma)(\alpha\beta\gamma)}{h(\{e, f, g\})h(\{a, b, e\})h(\{e, c, v\})s(e)s(v)} = \frac{\alpha\beta\gamma}{s(v)}.
\end{equation}
Expanding out the coefficient of the second term on the right hand side of (\ref{Asswvi2}) using the definition (\ref{defR1}) of the function $r$ and the relation (\ref{dooreenpunt}) with $p = f$ we find:
\begin{equation}\label{Asswvi4}
\frac{r(a, f)r(b, c)}{h(l_{v, i})} = \frac{(\alpha\beta\gamma)(\alpha\beta\gamma)}{h(\{e, f, g\})h(\{a, f, v\})h(\{b, c, f\})s(f)s(v)} = \frac{\alpha\beta\gamma}{s(v)}.
\end{equation}
Combining(\ref{Asswvi2}, \ref{Asswvi3}, \ref{Asswvi4}) we obtain:
\begin{equation}\label{Asswvi5}
\Ass(w_{v, i}) = -2\frac{\alpha\beta\gamma}{s(v)} v
\end{equation}
as announced in the first line of the proof.
\end{proof}

%\section{Counting the number of blocks in the finite field case}\label{counting}
\section{Quantifying the results of section \ref{split} in the finite field case}\label{counting}

\subsection{Notational conventions}
We specialize to the case $F = \mathbb{F}_q$ for $q$ an odd prime power. $O$ is still a Cayley algebra over $F$, which now necessarily is split by Corollary \ref{splitfinite}. Everything being finite we are able to count the number of two-dimensional subspaces of each of the six types as well as the number of 4-dimensional associative subalgebras of $O$. The latter number will equal the number of blocks in the $q$-covering design of Theorem \ref{main2} and thus provides an upper bound to the $q$-covering number $\mathcal{C}_q(7, 3, 2)$ defined in \cite{Lambert} as follows:

\begin{definition}
The $q$-covering number $\mathcal{C}_q(n, k, t)$ is the minimum size of a $q$-covering design over $\mathbb{F}_q$.
\end{definition}
 
Lambert \cite{Lambert} gives a construction for $q$-Covering designs over finite fields $\mathbb{F}_q$ which, for parameters $(7, 3, 2)$, contains $q^8 + 2q^6 + 3q^4 + q^3 + 2q^2 + q + 1$ blocks. This provides the following bounds on $\mathcal{C}_q(7, 3, 2)$: 
\begin{equation}\label{Cq}
\qbinom{7}{2}/\qbinom{3}{2} \leq C_q(7, 3, 2) \leq q^8 + 2q^6 + 3q^4 + q^3 + 2q^2 + q + 1
\end{equation}
where the lower bound, which equals $q^8 + q^6 + q^5 + q^4 + q^3 + q^2  + 1$ is the number of block in a hypothetical $q$-fano plane over $\mathbb{F}_q$. As we will see in this section, the number blocks in the $q$-covering design of Theorem \ref{main2} in case of finite $F$ will be strictly larger than the upperbound of (\ref{Cq}), showing that, apparently, vastly different $q$-covering designs that are minimal with respect to the inclusion ordering can exist over the same field. Of course, from Theorem \ref{main2} we already knew that in the case of fields, such as $\RR$, over which both split and division Cayley algebras exist.  

The lower bound in (\ref{Cq}) employed the following standard notation:

\begin{notation}
Viewing $q$ as an abstract variable that (if needed) can take values in the complex numbers, we have for $n, k \in \mathbb{N}$ the following polynomials in $q$:
\begin{itemize}
\item $[n]_q \coloneqq \frac{q^{n} - 1}{q-1} = 1 + q + \ldots + q^{n-1}$, with the convention that $[0]_q = 0$. It is easy to see that $\lim_{q \to 1} [n]_q = n$.
\item $[n]_q! \coloneqq [1]_q[2]_q \cdots [n]_q$ with the convention $[0]_q! = 1$. Again $\lim_{q \to 1} [n]_q! = n!$.
\item $\qbinom{n}{k} \coloneqq \frac{[n]_q!}{[n-k]_q![k]_q!}$. We see that $\lim_{q \to 1} \qbinom{n}{k} = \binom{n}{k}$. The $\qbinom{n}{k}$ are called $q$-binomial coefficients or Gaussian binomial coefficients.
\end{itemize}
\end{notation}

We recall:

\begin{proposition}[See e.g. \cite{Cohn}]
\mbox{}
\begin{itemize} 
\item The $q$-binomial coefficients satisfy the recurrence relation 
$$\qbinom{n}{k} = \qbinom{n-1}{k} + q^{n-k}\qbinom{n-1}{k-1}$$ 
and (hence) are te coefficients of $X^kY^{n-k}$ in the expansion of $(X + Y)^n$ in the ring $\mathbb{Q}[X, Y]_q$ of (non-commutative) polynomials in two variables $X, Y$ subject (only) to the relaton $YX = qXY$.
\item Let $q$ be an actual number and let $F$ be a field with $q$ elements (or, more generally, a topological field with Euler characteristic $q$). Then $\qbinom{n}{k}$ equals the number of $k$-dimensional subspaces in an $n$-dimensional vector space over $F$. (Or, more generally, the Euler-characteristic of the Grasmannian $\Gr_{n, k}(F)$).
\end{itemize}
\end{proposition}

Of course the second bullet point is the main reason that this notation is useful to us in the current section.

In Section \ref{split} we defined 6 types of two-dimensional subspaces: Q, U, D, M, J and Z. In Proposition \ref{QMDJ} we saw that every four-dimensional associative subalgebra is generated by a subspace of type Q, D, M or J and hence has structure described by Proposition \ref{main}, \ref{typeD}, \ref{typeM} or \ref{typeJ} respectively. (See also Theorem \ref{classification}) In particular they are either a four-dimensional Cayley-Dickson algebra, a (vectorspace) direct sum of a two-dimensional Cayley-Dickson algebra and a two-dimensional Jacobson radical, or a direct sum of the 1-dimensional Cayley-Dickson algebra and a three-dimensional Jacobson radical. Over the finite field $F = \mathbb{F}_q$ there are up-to-isomorphism exactly one four-dimensional Cayley-Dickson algebra (the algebra $\Mat(2, F)$) by Cor. \ref{splitfinite} and exactly two two-dimensonal Cayley-Dickson algebras: the split one and one isomorphism class of division Cayley-Dickson algebra. The reason for the uniqueness of the latter is that 2-dimensional Cayley-Dickson algebras are commutative by Thm \ref{Albert} and hence the divison algebras among them are fields, more specifically quadratic field extensions of the ground field. The uniqueness of these in the case the ground field is finite has been known for a long time. Having thus a complete classification up-to-isomorphism of the four-dimensional subalgebras of $O$ it remains to give each class of subalgebra a name in order to simplify notation throughout the rest of the section.

\begin{notation}\label{subalgtypes}
Let $H$ be a four-dimensional subalgebra of $O$ then we say that $H$ is of:
\begin{itemize}
\item[Type M4] if it is a four-dimensional simple algebra, which is necessarily isomorphic to $\Mat(2, F)$ by the above.
\item[Type F2J2] If it has 2-dimensional Jacobson radical $J(H)$ and decomposes as a vector space direct sum $H = D \oplus J(H)$ with $\dim(D) = \dim J(H) = 2$, where $J(H)$ is a type Z space and $D$ is a two-dimensional Cayley-Dickson division algebra (hence isomorphic to $\mathbb{F}_{q^2}$). All algebras in this class are isomorphic since, by Proposition \ref{QMDJ}, they contain a type D subspace and hence the relation between the subalgebras $D$ and $J(H)$ is determined completely by Propostion \ref{typeD}.
\item[Type S2J2] If it has 2-dimensional Jacobson radical $J(H)$ and decomposes as a vector space direct sum $H = D \oplus J(H)$ with $\dim(D) = \dim J(H) = 2$, where $J(H)$ is a type Z space and $D$ is the two-dimensional split Cayley-Dickson algebra (hence isomorphic to the \emph{algebra} direct sum $F \oplus F$). All algebras in this class are isomorphic since, by Proposition \ref{QMDJ}, they contain a type D subspace and hence the relation between the subalgebras $D$ and $J(H)$ is determined completely by Propostion \ref{typeD}.
\item[Type F1J3] If it has 3-dimensional Jacobson radical $J(H)$ and hence $H$ decomposes as the vector space direct sum $F1 \oplus J(H)$. All algebras in this class are isomorphic again, with structure described by Proposition \ref{typeJ}. (So unlike the M4 and F2J2 cases the uniqueness up to isomorphism here does not use that $F$ is finite.)
\end{itemize} 
\end{notation}

We recall from Theorem \ref{classification} that 2-dimensional subspaces of types Q and M generate subalgebras of type M4, 2-dimensional subspaces of type $J$ generate subalgebras of type F1J3 and two-dimensional subspaces of type D generate subalgebras of type either F2J2 or S2J2. We refine this latter point:

\begin{lemma}\label{Dns}
Let $D$ be a type D space with basis $u, v$ as in Proposition \ref{typeD}, so $u^2 = 0$ and $v^2 = \beta \neq 0 \in F$. Let $v' \in D$ be any non-nilpotent element. Then $(v')^2$ is a square in $F$ if and only if $\beta$ is a square in $F$ if and only if $\langle D \rangle$ is of type S2D2. Conversely, $(v')^2$ is a non-square in $F$ if and only if $\beta$ is a non-square in $F$ if and only if $\langle D \rangle$ is of type F2D2. 
\end{lemma} 
\begin{proof}
Directly from Proposition \ref{typeD} and Lemma \ref{2dim}
\end{proof}
\begin{notation}
We say that a type D subspace of $\im(O)$ all whose non-nilpotent elements square to squares in $F$ is of type Ds and that a type D subspace of $\im(O)$ all whose non-nilpotent elements square to non-squares in $F$ is of type Dn.
\end{notation}

\subsection{Summary of the results}
The goal of this section is to count, in case that $F = \mathbb{F}_q$ the following quantities:
\begin{itemize} 
\item For each type $X$ of two-dimensional subspace of $\im(O)$ the total number of spaces of type $X$. This number will be denoted $\Theta_X$.
\item For each type $Y$ of four-dimensional associative subalgebra of $O$ the total number of subalgebras of that type. This number will be denoted $N_Y$.
\item For each type $X$ of two-dimensional subspace of $\im(O)$, each fixed subspace $U$ of type $X$ and each type $Y$ of four-dimensional associative subalgebra of $O$ the number of subalgebras of type $Y$ containing $U$. This number will be denoted $H_{X,Y}$.
\item For each type $X$ of two-dimensional subspace of $\im(O)$, each type $Y$ of four-dimensional associative subalgebra of $O$ and each fixed subalgebra $H$ of type $Y$, the number of spaces of type $X$ contained in $H$. This number will be denoted $T_{X,Y}$.
\end{itemize}

It follows from Theorem \ref{transitive} that the last two numbers only depend on the types $X$, $Y$ and not on the specific choice of $U$ and $H$ so that the absense of $U$ and $H$ from the notation makes sense.

%Some numbers \emph{will} hovever depend on whether or not $-1$ is a square in $F$, which is the case if only if $q \equiv 3 \mod 4$. We write $\epsilon$ for the unique elmenent of $\{+1, -1\}$ that satisfies $q \equiv \epsilon \mod 4$, a notational convention taken form \cite{die notes op internet. Zie een van die emails met links}.

The results of this section are collected in tables \ref{Thetavalues}, \ref{Nvalues}, \ref{Qvalues}, \ref{Tvalues}. We note that the values of $H_{\textrm{Q}, Y}$, $H_{\textrm{M}, Y}$, $H_{\textrm{Ds}, Y}$, $H_{\textrm{Dn}, Y}$ $H_{\textrm{J}, Y}$ given in Table \ref{Qvalues} have already been established in section \ref{split} as have the zeros among the $T_{\textrm{Q}, Y}$, $T_{\textrm{M}, Y}$, $T_{\textrm{Ds}, Y}$, $T_{\textrm{Dn}, Y}$, $T_{\textrm{J}, Y}$ in Table \ref{Tvalues}. The remaining entries in tables \ref{Thetavalues}, \ref{Nvalues}, \ref{Qvalues}, \ref{Tvalues} will be calculated here. 

We will frequently use the `double counting' identity 

$$\Theta_X H_{X, Y} = T_{X, Y} N_Y.$$

\begin{table}
\caption{\label{Thetavalues} Number $\Theta_X$ of two-dimensional subspaces of type $X$ of $\im(O)$ where $X$ runs over the seven types described in section \ref{split}.}
\begin{tabular}{r|r}
Q  & $\frac{1}{2}([11]_{-q} - [5]_{-q})$   \\ 
U  & $[8]_q - [2]_q$   \\
Dn & $\frac{q^9 + q^8 - q^3 - q^2}{2}$   \\
Ds & $\frac{q^9 + q^8 - q^3 - q^2}{2}$   \\
M  & $\frac{1}{2}([11]_q - [5]_q)$ \\
J  & $[8]_q - [2]_q$   \\
Z  & $[6]_q$    
\end{tabular}
\end{table}

\begin{table}
\caption{\label{Nvalues} Number $N_Y$ of four-dimensional associative subalgebras of $O$ of type $Y$ where $Y$ runs over the three types defined above.}
\begin{tabular}{rrrrr}%heading misschien wel centreren of links uitlijnen? Ben vergeten hoe het moet. Misschien ook onzin.
M4                  & F2J2                 & S2J2                                                   & F1J3                                                %& Total 
\\ \hline
$q^8 + q^6 + q^4$   & $\frac{q^7 - q}{2}$  & $\frac{q^7 + 2q^6 + 2q^5 + 2q^4 + 2q^3 + 2q^2 + q}{2}$ & $q^5 + q^4 + q^3 + q^2 + q + 1$  %& $q^8 + q^7 + 2q^6 + 2q^5 + 3q^4 + 2q^3 + 2q^2 + q + 1$ 
%dus totaal F2J2 + S2J2 = q^7 + q^6 + q^5 + q^4 + q^3 + q^2 = q^2[6]_q          
\end{tabular}
\end{table}

\begin{table}
\caption{\label{Qvalues} Number $H_{X, Y}$ of four-dimensional associative subalgebras of $O$ of type $Y$ (indexing the columns) containing a fixed two-dimensional subspace $U$ of $\im(O)$ of type $X$ (indexing the rows).}
\begin{tabular}{l|rrrrr}%heading misschien wel centreren of links uitlijnen? Ben vergeten hoe het moet. Misschien ook onzin.
     &    M4 &              F2J2 &               S2J2 &    F1J3    %& Sum 
     \\ \hline 
Q    &   $1$ &               $0$ &                $0$ &     $0$   \\% & $1$        \\    
U    & $q^2$ &               $0$ &              $q+1$ &     $0$    \\%& $[3]_q$    \\
Dn   &   $0$ &               $1$ &                $0$ &     $0$    \\%& $1$        \\
Ds   &   $0$ &               $0$ &                $1$ &     $0$   \\% & $1$        \\
M    &   $1$ &               $0$ &                $0$ &     $0$    \\% & $1$        \\
J    &   $0$ &               $0$ &                $0$ &     $1$    \\% & $1$        \\
Z    &   $0$ & $\frac{q^2-q}{2}$ &  $\frac{q^2+q}{2}$ & $q + 1$    \\% & $[3]_q$   \\
\end{tabular}
\end{table}

\begin{table}
\caption{\label{Tvalues} Number $T_{X, Y}$ of two-dimensional subspaces of type $X$ (indexing the rows) of $\im(H)$ where $H$ is a fixed four-dimensional associative subalgebra of $O$ of type $Y$ (indexing the columns).}
\begin{tabular}{l|rrrr}%heading misschien wel centreren of links uitlijnen? Ben vergeten hoe het moet. Misschien ook onzin.
    &                     M4 &      F2J2 &      S2J2 &    F1J3   \\ \hline
Q   & $\frac{1}{2}(q^2 - q)$ &       $0$ &       $0$ &     $0$        \\    
U   &                $q + 1$ &       $0$ &      $2q$ &     $0$        \\
Dn  &                    $0$ & $q^2 + q$ &       $0$ &     $0$        \\
Ds  &                    $0$ &       $0$ & $q^2 - q$ &     $0$        \\
M   & $\frac{1}{2}(q^2 + q)$ &       $0$ &       $0$ &     $0$        \\
J   &                    $0$ &       $0$ &       $0$ &   $q^2$        \\
Z   &                    $0$ &       $1$ &       $1$ & $q + 1$        \\
\end{tabular}
\end{table}

By Proposition \ref{QMDJ} and Theorem \ref{classification}, the number of blocks in the $q$-covering design of Theorem \ref{main2} is the sum of the entries in Table \ref{Qvalues} and hence equals

$$q^8 + q^7 + 2q^6 + 2q^5 + 3q^4 + 2q^3 + 2q^2 + q + 1.$$

A very useful tool in this section will be Theorem \ref{transitive} which states that isomorphisms between subalgebras extend to automorphisms of $O$. %We finish this subsection with establishing a further useful property of these automorphisms: %[STAAT AL IN H5]

%\begin{lemma}\label{autonice}
%Let $\phi$ be an automorphism of $O$ and $x \in O$. Then  $\tau(\phi(x)) = \tau(x)$, $\phi(x^*) = \phi(x)^*$ and $\phi(x) \in \im(O)$ if and only if $x \in \im(O)$.
%\end{lemma}
%\begin{proof}
%All statements are obvious if $x = 0$, so from now on we assume that $x \neq 0$. We recall that any automorphism satisfies $\phi(1) = 1$ and hence by linearity maps every element of $F1$ to itself.
%
%We first prove the third statement. Let $x \in \im(O)$. Then $x^2 \in F1$ and hence $\phi(x^2) = x^2$. But $\phi(x^2) = \phi(x)^2$ so that $\phi(x)^2 \in F1$. It follows that either $\phi(x) \in \im(O)$ or $\phi(x) \in F1$. But the latter case we would have $x = \phi^{-1}(\phi(x)) \in F1$ which is absurd. It follows that $\phi(x) \in \im(O)$ as desired. Conversely, suppose that $\phi(x) \in \im(O)$, then $x = \phi^{-1}(\phi(x)) \in \im(O)$ by the statement just proved.
%
%Now let $x$ be general. We can write $x = \tau(x) + \im(x)$ as in Section \ref{quadratic} with $\tau(x) \in F$ and $\im(x) \in \im(O)$. It then follows that $\phi(x) = \phi(\tau(x)) + \phi(\im(x)) = \tau(x) + \phi(\im(x))$ and since we just established that $\phi(\im(x)) \in \im(O)$ we find that $\tau(\phi(x)) = \tau(x)$ and $\im(\phi(x)) = \phi(\im(x))$. Finally, from $x^* = \tau(x) - \im(x)$ we then obtain that $\phi(x)^* = \phi(x^*)$.
%\end{proof}

\subsection{The subalgebras of type M4}

We start our journey by introducing a notion Section 1.7 of \cite{SprVeld} (the same section that gave us Theorem \ref{transitive}):

\begin{definition}
A \emph{special $(-1, 1)$-pair} is a pair $(e_+, e_-)$ of elements of $O$ such that $e_+ \in \im(O)$, $e_- \in \im(O)$, $e_+^2 = 1$, $e_-^2 = -1$, $e_+e_- \in \im(O)$. 
\end{definition}

We find the following relation to the classification in Section \ref{split}:

\begin{lemma}\label{epluseminM}
Let $(e_+, e_-)$ be a special $(-1, 1)$-pair. Then $u \coloneqq \frac{e_+ + e_1}{2}$ and $v \coloneqq \frac{e_+ - e_1}{2}$ satisfy $u^2 = v^2 = 0$, $2\tau(uv) = 1$. Hence $\spam(e_+, e_-)$ is of type M by Lemma \ref{typeMbasis} and subsequently $\langle e_+, e_- \rangle$ is of type M4. Conversely, any type M space $U$ with $u, v$ as in Corollary \ref{typeMbasis} contains at least one special $(-1, 1)$-pair given by $e_+ = u + v$, $e_- = u- v$.
\end{lemma} 
\begin{proof}
The statement about $\langle e_+, e_- \rangle$ follows from Prop. \ref{typeM}. For the remaining statements it suffices to remember (from Section \ref{quadratic}) that for $x, y \in \im(O)$ we have that $2\tau(xy) = xy + yx$.
\end{proof}

\begin{corollary}\label{plusminpaar}
Let $(e_+, e_-)$ be a special $(-1, 1)$-pair. Then $\langle e_+, e_- \rangle \isom D_{-1}(D_1(F)) \isom \Mat(2, F)$ where we may view $e_-$ as the element $i$ used in the second doubling step and $e_+$ as the element $i$ used in the first doubling step. 
\end{corollary}
\begin{proof}
Directly from Lemma \ref{epluseminM} and Proposition \ref{typeM}.
\end{proof}

We use this information to estabish the numbers $N_{\textrm{M4}}$, $\Theta_\textrm{M}$ and $T_{\textrm{M}, \textrm{M4}}$. Of course we already know from Proposition \ref{typeM} that $H_{M, M4} = 1$.

\begin{lemma}
There are $q^{11} - q^5$ special $(-1, 1)$-pairs.
\end{lemma}
\begin{proof}
Using the basis of Table \ref{multtab} (and recalling that $p_1 + p_2 = 1$), we see that the elements $p_1 - p_2$, $q_1, q_2, q_3$, $r_1, r_2, r_3$ form a basis of $\im(O)$. Let $e_+ = \zeta(p_1 - p_2) + \sum_{i = 1}^3 \eta_i q_i + \theta_i r_i$ with greek letters denoting scalars in $F$. Then the condition $e_+^2 = 1$ ammounts to $\zeta^2 - \eta_1 \theta_1 - \eta_2 \theta_2 - \eta_3 \theta_3 = 1$. It is clear that there are $q^5(q-1)$ solutions to this equation satisfying $\eta_3 \neq 0$: given any choice of $\zeta, \eta_1, \zeta_1, \eta_2, \zeta_2 \in F$, $\eta_3 \in F^\times$ we can compute $\zeta_3 = \frac{-1 + \zeta^2 - \eta_1 \theta_1 - \eta_2 \theta_2}{\eta_3}$. Similarly there are $q^3(q-1)q$ solutions satisfying $\eta_3 = 0$, $\eta_2 \neq 0$ (where the factor $q^3$ denotes the choices of $\zeta, \eta_1, \theta_1$, the factor $q - 1$ the choices for $\eta_2$ and the last $q$ the choices for $\theta_3$) and $q(q-1)q^2$ solutions satisfying $\eta_3 = \eta_2 = 0$, $\eta_1 \neq 0$. Finally there are $2q^3$ solutions satisfying $\eta_1 = \eta_2 = \eta_3 = 0$ where the $2$ corresonds to the possible choices $\zeta = 1$ and $\zeta = -1$ for $\zeta$ and the $q^3$ to the choices of the $\theta_i$. Adding it all up we find that there are $q^6 + q^3$ choices for $e_+$.

One such possible choice is $e_+ = p_1 - p_2$. Let $(e_+, e_-)$ be any special $(-1, 1)$ pair. Both $\langle e_+  \rangle$ and $\langle (p_1 - p_2) \rangle $ are isomorphic to the two-dimensional split Cayley-Dickson algebra  $D_1(F)$ and hence, by Thm. \ref{transitive} there is an automorphism $\phi$ of $O$ such that $\phi(e_+) = p_1 - p_2$. Now we see with help from Lemma \ref{autonice} that $(\phi(e_+), \phi(e_-))$ is a special $(-1, 1)$-pair as well. In other words: for a given $e_+ \in \im(O)$ satisfying $e_+^2 = 1$, the number of elements $e_- \in \im(O)$ satisfying both $e_-^2 = -1$ and $e_+e_- \in \im(O)$ does not depend on $e_+$ and equals the number of elements $e_- \in \im(O)$ satisfying both $e_-^2 = -1$ and $(p_1 - p_2)e_- \in \im(O)$.

In order to count the elements satisfying these last two equations we write $e_- = \zeta'(p_1 - p_2) + \sum_{i=1}^3 \eta_i' q_i + \theta_i' r_i$. The condition $\tau((p_1 - p_2)e_-) = 0$ implies that $\zeta' = 0$ and the condition that $e_-^2 = -1$ then gives that $\sum \eta_i' \theta_i' = 1$. With reasoning similar to what we did for $e_+$ we find that there are $q^4(q-1) + q^2(q-1)q + (q-1)q^2 = q^5 - q^2$ possible choices for the element $e_-$.

Hence the total number of special $(-1, 1)$-pairs is $(q^6 + q^3)(q^5 - q^2) = q^{11} - q^5$.  
\end{proof}

\begin{lemma}
The number of special $(-1, 1)$ pairs in $\Mat(2, F)$ equals $q^3 - q$.
\end{lemma}
\begin{proof}
We can just repeat the proof of the previous lemma in the subalgebra of $O$ with vector-space basis $p_1, p_2, q_1, r_1$. Alternatively, for a more high-level perspective, we could realize $\Mat(2, F)$ as $D_{-1}(D_1(F))$ which singles out a `special' special $(-1, 1)$-pair consisting of the elements used in the role of $i$ in the first and second doubling step. (Cf. Cor. \ref{plusminpaar}.) Now every automorphism of $\Mat(2, F)$ sends this `special' special $(-1, 1)$-pair to some (possibly different) special $(-1, 1)$ pair and conversely every linear map that does this generates a unique automorphism of $D_{-1}(D_1(F))$ and hence $\Mat(2, F)$. We thus see that the number of special $(-1, 1)$-pairs equals the order of the automorphism group of $\Mat(2, F)$. By the Skolem-Noether theorem (see e.g. \cite{Pierce}) every automorphism is of the form $x \mapsto y^{-1}xy$ for some invertible element $y \in Mat(2, F)$, that is: for some $y \in \GL(n, F)$. Since elements $y_1$, $y_2$ in $\GL(2, F)$ define the same conjugation if and only if they are scalar multiples of each other we find that $\Aut(Mat(2, F))= \PGL(2, F)$. The orders of groups of this type are well known. In particular they are twice the orders of the even better known groups $\PSL(2, q)$.
\end{proof}
%remark: hoewel SL(2, q) en PGL(2, q) beide even veel elementen hebben (q^3 - q) zijn het heel verschilllende groepen. In het algemeen  is |SL(n, q)| = |GL(n, q)|/(q-1) = |PGL(n, q)|$ dus altijd gelijkheid van aantal elementen. Maar gelijkheid als groep geldt denk ik desda n oneven. (Denk ik. Niet goed nagegaan)

\begin{corollary}\label{NM4}
The number $N_{\textrm{M4}}$ of quaternion subalgebras of $O$ equals $\frac{q^{11} - q^5}{q^3 - q} = q^8 + q^6 + q^4$.
\end{corollary}

\begin{lemma}
Let $U$ be a type M space  Then $U$ contains $2(q-1)$ special $(-1, 1)$-pairs.
\end{lemma}
\begin{proof}
Let $q, r$ be a basis of $U$, as in Lemma \ref{typeMbasis}. An element $e_+ = \alpha q + \beta r$ satisfies $e_+^2 = 1$ if and only if $\alpha\beta = 1$, yielding $q - 1$ possible choices of $e_+$. We fix one such choice: $e_+ = q + r$. Now let $e_- = \gamma q + \delta r$. Then $e_+ e_- \in \im(O)$ means $2\tau(e_+e_-) = e_+e_- + e_-e_+ = 0$. But we compute that $e_+e_- + e_-e_+ = (\gamma + \delta)2\tau(qr) = \gamma + \delta$. It follows that $e_-$ is of the form $\gamma(q - r)$. The condition $e_-^2 = -1$ then yields $\gamma^2 = 1$, leaving two possibilities open for $e_-$.
\end{proof}

\begin{corollary}\label{typeMcount}
There are $\Theta_\textrm{M} = \frac{q^{11} - q^5}{2(q - 1)} = \frac{1}{2} (q^{10} + q^9 + q^8 + q^7 + q^6 + q^5)$ type M spaces in $\im(O)$ and $T_{M, M4} = \frac{q^3 - q}{2(q - 1)} = \frac{1}{2}(q^2 + q)$ type M subspaces in $\Mat(2, F)$.
\end{corollary}

Having the value of $N_{M4}$ enables us to compute the values of $\Theta_\textrm{Q}$, $\Theta_\textrm{U}$, $\Theta_\textrm{M}$ from the numbers $T_{\textrm{Q}, \textrm{M4}}$, $T_{\textrm{U}, \textrm{M4}}$, $T_{\textrm{M}, \textrm{M4}}$ and $H_{\textrm{Q}, \textrm{M4}}$, $H_{\textrm{U}, \textrm{M4}}$, $H_{\textrm{M}, \textrm{M4}}$ and vice versa. Thm. \ref{transitive} tells us that the numbers $T_{X, \textrm{M4}}$ do not depend on the particular $M4$-subalgebra and we can compute these numbers in the very explicit model of such an algebra as the set of two by two matrices over $F$. This makes these computations rather straightforward and the results are given in Proposition \ref{TXM4} below. 

For the numbers $H_{X, M4}$ we notice that the results of Section \ref{split} give $H_{\textrm{Q}, \textrm{M4}} = H_{\textrm{M}, \textrm{M4}} = 1$, $H_{\textrm{U}, \textrm{M4}} \geq q^2$, $H_{\textrm{Dn}, \textrm{M4}} = H_{\textrm{Ds}, \textrm{M4}} = H_{\textrm{J}, \textrm{M4}} = 0$. The fact that $H_{\textrm{Z}, \textrm{M4}} = 0$ follows from the fact, established in Proposition \ref{TXM4} below, that $T_{\textrm{Z}, \textrm{M4}} = 0$. So the only missing info on the numbers $H_{X, \textrm{M4}}$ is the precise value of $H_{\textrm{U}, \textrm{M4}}$. This will be the subject of Proposition \ref{Q2M4} below.

\begin{proposition}\label{TXM4}
Let $P = \{\begin{pmatrix}
\alpha & \beta \\ \gamma & -\alpha
\end{pmatrix} \colon \alpha, \beta, \gamma \in F \} \isom \im(D_{-1}(D_1(F)))$. 
Then $P$ contains $T_{\textrm{Q}, \textrm{M4}} = \frac{1}{2}(q^2 - q)$ type Q subspaces, $T_{\textrm{U}, \textrm{M4}} = q+1$ type U subspaces, $T_{\textrm{Dn}, \textrm{M4}} = T_{\textrm{Ds}, \textrm{M4}} = T_{\textrm{J},  \textrm{M4}} = 0$ type D and J subspaces, $T_{\textrm{M}, \textrm{M4}} = \frac{1}{2}(q^2 + q)$ type M subspaces and $T_{\textrm{Z}, \textrm{M4}} = 0$ type Z subspaces.
\end{proposition}
\begin{proof}
The statements about subspaces of type D, J, and M follow (respectively) from Prop. \ref{typeD}, Prop. \ref{typeJ} and Corollary \ref{typeMcount} above. For the remaining cases we first identify the nilpotents in $P$. From solving $\begin{pmatrix}
\alpha & \beta \\ \gamma & -\alpha
\end{pmatrix}^2 = 0$ we infer that $\begin{pmatrix}
\alpha & \beta \\ \gamma & -\alpha
\end{pmatrix}$ is nilpotent if and only if $\alpha^2 = -\beta\gamma$. 
It follows that there are $q^2 - 1$ non-zero nilpotent elements, together making up $(q^2 - 1)/(q - 1) = q + 1$ nilpotent lines. Moreover, by the theory of Jordan Normal Forms, for each such line $Fu$ there is a $g \in \GL(2, F)$ such that $g(Fu)g^{-1} = \begin{pmatrix}
0 & F \\ 0 & 0
\end{pmatrix}$. 
%(Explicitly: assuming wlog that $u \in Fu$ is the element with a 1 in the top right corner and a non-zero number $\alpha$ in the top left corner, we find that with $g = \begin{pmatrix} 0 & \alpha^{-1} \\ \alpha & 1$ we have that $gug^{-1} = \begin{pmatrix}0 & 1 \\ 1 & 0\end{pmatrix}$.
Since spaces of type U contain exactly one nilpotent line by definition, we can can compute the number $T_{\textrm{U}, \textrm{M4}}$ by counting the number of type U spaces containing the line $l \coloneqq \begin{pmatrix}
0 & F \\ 0 & 0
\end{pmatrix}$ and multiplying the outcomes by the total number $q+1$ of nilpotent lines.

It is easy to see that there is exactly one type U space in $P$ containing $l$ -- it consists of the traceless uppertriangular matrices. It follows that $T_{U, M4} = q + 1$. 

Finally let $Z$ be a hypothetical type Z subspace of $P$. By the $\GL(2, F)$ action we may assume without loss of generality that $Z$ contains $l$. However it is easy to see that every $p \in P$ satisfying $pl = \{0\}$ lies in $l$ itself, hence confirming that $T_{\textrm{Z}, \textrm{M4}} = 0$.

The value of $T_{Q, M4}$ then follows from subtracting the number of subspaces of the other types from the total number $\qbinom{3}{2} = q^2 + q + 1$ of two-dimensional spaces in $P$.
\end{proof}

\begin{corollary}
There are $\Theta_\textrm{Q} = N_{\textrm{M4}}T_{\textrm{Q}, \textrm{M4}}/H_{\textrm{Q}, \textrm{M4}} = (q^8 + q^6 + q^4)(\frac{1}{2}(q^2 - q))/1 = \frac{1}{2}(q^{10} - q^9 + q^8 - q^7 + q^6 - q^5)$ type Q subspaces in $\im(O)$ %and $\Theta_M = N_{M4}T_{M, M4}/Q_{M, M4} = (q^8 + q^6 + q^4)(\frac{1}{2}(q^2 + q)})/1 = \frac{1}{2}(q^{10} + q^9 + q^8 + q^7 + q^6 + q^5)$ type M subspaces in $\im(O)$. (laatste uit want die hadden we al)
\end{corollary}

\begin{proposition}\label{Q2M4}
Example \ref{problem} lists all the type M4 subalgebras of $O$ containing $U = \spam(p_1 - p_2, q_1)$ and hence (by Thm. \ref{transitive}) $H_{U, M4} = q^2$.
\end{proposition}
\begin{proof}
Let $H \subset O$ be a quaternion subalgebra containing $U$ and let $\psi_1 \colon H \to \Mat(2, F)$ be an isomorphism. We saw in the proof of Proposition \ref{TXM4} that there exist an element $g \in \Mat(2, F)$ such that $g(\psi_1(q_1))g^{-1} = \begin{pmatrix}
0 & 1 \\ 0 & 0 
\end{pmatrix}$.
Define an isomorphism $\psi_2 \colon H \to \Mat(2, F)$ by $\psi_2(x) = g\psi_1(x)g^{-1}$. From the algebraic relations $\psi_2(p_1 - p_2)\psi_2(u) = - \psi_2(u)\psi_2(p_1 - p_2) = \psi_2(u)$ we conclude that $\psi_2(p_1 - p_2) = \begin{pmatrix}
1 & \beta \\ 0 & -1 
\end{pmatrix}$ for some $\beta \in F$. Let $h = \begin{pmatrix}
1 & \beta/2 \\
0 & 1
\end{pmatrix}$ and define the isomorphism $\psi_3 \colon H \to \Mat(2, F)$ by $\psi_3(x) = h\psi_2(x)h^{-1}$. Then $\psi_3(q_1) = \begin{pmatrix}
0 & 1 \\ 0 & 0 
\end{pmatrix}$, $\psi_3(p_1 - p_2) = \begin{pmatrix}
1 & 0 \\ 0 & -1 
\end{pmatrix}$, and, since $p_1 + p_2 = 1$, $\psi_3(p_1) = \begin{pmatrix}
1 & 0 \\ 0 & 0 
\end{pmatrix}$ and $\psi_3(p_2) = \begin{pmatrix}
0 & 0 \\ 0 & 1 
\end{pmatrix}$.

Now let $w = \psi_3^{-1}(\begin{pmatrix}
0 & 0 \\
1 & 0 
\end{pmatrix}) \in H \subset O$. It is clear that by construction that $w$ satisfies 
\begin{equation} \label{wrelations}
q_1w = p_1; \qquad wq_1 = p_2; \qquad p_1w = wp_2 = 0
\end{equation}
and that $\psi_3$ equals the isomorphism $\psi_{u, w}$ of Example \ref{problem}. It is clear from Table \ref{multtab} that the same equalities (\ref{wrelations}) hold with $r_1$ in the role of $w$. Hence it follows that $w - r_1 \in \ker L_{q_1} \cap \ker R_{q_1} \cap \ker L_{p_1} \cap \ker R_{p_2}$ where $L_x, R_x$ denote left and right multiplication with $x$ respectively as in Lemma \ref{division}. Hence, in order to verify that the algebra $H$ appears in the list of Example \ref{problem} it suffices to show that $\ker L_{q_1} \cap \ker R_{q_1} \cap \ker L_{p_1} \cap \ker R_{p_2} = \spam(r_2, r_3)$. As we can explicitly write down the matrix representations of $L_{q_1}, R_{q_1}, L_{p_1}, R_{p_2}$ with respect to the basis of Table \ref{multtab} this is a straightforward application of Gaussian elimination.

Since moreover the above shows that $p_1, p_2, q_1, w$ form a vector space basis of $H$ it is clear that $w$ is the \emph{only} element of $H \cap W$ (with $W = r_1 + Fr_2 + Fr_3$ as in Example \ref{problem}), which means that two of the $q^2$ elements of $W$ define the same algebra. We hence find that the number of type $M4$-algebras containing $U$ equals $q^2$. Finally let $U' \subset \im(O)$ be an arbitrary type $U$ space with unique nilpotent line $l'$. Then $F1 \oplus U'$ is a subalgebra of $O$ by Proposition \ref{typeU} and any linear bijection $\phi \colon F1 \oplus U' \to F1 \oplus U$ sending $1$ to $1$ and $l'$ onto $Fq_1$ is an algebra isomorphism between these subalgebras that then, by Theorem \ref{transitive} exends to an automorphism $\overline{\phi}$ of $O$. Since $\overline{\phi}$ bijectively maps type M4-algebras containing $U'$ to M4-algebras containing $U$, we see that the number of type M4-algebras containing $U'$ equals $q^2$ as well and hence that the statement $Q_{U, M4} = q^2$ of the lemma is true and well defined.
\end{proof}
\begin{corollary}
There are $\Theta_\textrm{U} = N_{\textrm{M4}} T_{\textrm{U}, \textrm{M4}}/Q_{\textrm{U}, \textrm{M4}} = \frac{(q^8 + q^6 + q^4)(q+1)}{q^2} = q^7 + q^6 + q^5 + q^4 + q^3 + q^2$ type U spaces.
\end{corollary}

\subsection{Completing Table \ref{Tvalues}}
At this point we have established values in the rows labeled Q and M in Tables \ref{Thetavalues} and \ref{Tvalues} and the values in the columns labeled M4 in Tables \ref{Nvalues} and \ref{Qvalues}, as well as the value in row U of Table \ref{Thetavalues} and those in rows Dn, Ds and J in Table \ref{Qvalues}. 

We now move on to completing the verification of Table \ref{Tvalues}. The elements $T_{X, M4}$ have been computed in proposition \ref{TXM4}, the number %$T_{X, F2J2}$ and 
$T_{X, F1J3}$ follow from %Proposition %\ref{typeD} and 
Corollary \ref{F1J3subspaces}. Concretely: 
%in Proposition \ref{typeD} it was shown that a type $D2J2$ subalgebra $D$ of $O$ contains exactly one type $Z$ subspace and that all other two-dimensional subspaces of $\im(D)$ are of type D. [LET OP: hier moet dus iets mis zijn. Gaan dat later na. Hmmm of nu.] It follows that $T_{Q, D2J2} = T_{U, D2J2} = T_{J, D2J2} = T_{M, D2J2} = 0$, $T_{Z, D2J2} = 1$ and $T_{D, D2J2} = \qbinom{3}{2} - 1 = q^2 + q$. Similarly 
Corollary \ref{F1J3subspaces} tells us that for every type F1J3 subalgebra $J$ of $O$ the space $\im(J)$ contains a unique line $l$ such that all two-dimensional subspaces of $\im(J)$ containing $l$ are of type Z while those not containing $l$ are of type J. It follows that $T_{Q, F1J3} = T_{U, F1J3} = T_{M, F1J3} = T_{Dn, F1J3} = T_{Ds, F1J3} = 0$, $T_{Z, F1J3} = (\qbinom{3}{1} - 1)/(\qbinom{2}{1} - 1) = q + 1$ and $T_{J, F1J3} = \qbinom{3}{2} - T_{Z, F1J3} = q^2$. %sanity check: berekenen $T_{J, F1J3}$ direct als volgt: eerste basis vector van vlak moet buiten lijn l liggen: q^3 - q keuzes. Tweede basis vector moet buiten hele vlak opgespannen door l en eerste basisvector liggen: q^3 - q^2 keuzes. Dit telt vlakken met een geordende basis (die l niet bevatten). Ieder vlak heeft (q^2 - 1)(q^2 - q) geordende bases: eerste basisvector mag niet nul zijn maar verder alles en tweede mag niet op lijn opgespannen door eerste basisvector liggen. Vinden dus T_{J, F1J3} = \frac{(q^3 - q)(q^3 - q^2}{(q^2 - 1)(q^2 - q)} = q^2$. Ok.

It remains to verify the numbers $T_{X, Y}$ where $Y$ is either $F2J2$ or $S2J2$. Let $H$ be an algebra of either type (so in particular $H = \langle D \rangle$ for some type D space $D$). We know from Proposition \ref{typeD} that $J(H)$ is of type Z and that every nilpotent element of $H$ is contained in $J(H)$. This latter fact implies that on one hand that $J(H)$ is the only type Z space in $H$ so that $T_{\textrm{Z}, \textrm{F2J2}} = T_{\textrm{Z}, \textrm{S2J2}} = 1$ and on the other hand that every two-dimensional subspace of $\im(H)$ unequal to $J(H)$ contains exactly one nilpotent line (its intersection with $J(H)$) and hence is of type either D or U. We also know from Lemma \ref{Dns} that the type D spaces are of type Dn if and only if $H$ is of type F2J2 and of type Ds if and only if $H$ is of type S2J2. This establishes two of the three remaining zeroes in Table \ref{Tvalues}. We proceed by counting the number of type U subspaces of $H$ for $H$ (still) of type either $F2J2$ or $S2J2$. 

\begin{lemma}\label{gamma2}
Let $H$ be the associative algebra generated by a type D subspace of $\im(O)$, so (by Corollary \ref{curious}), $\im(H)$ has basis $u, v, uv$ satisfying $u^2 = 0$, $v^2 = \beta \in F \backslash \{0\}$, $uv = -vu$.  Let $U$ be a type U subspace of $\im(H)$ and $l$ be the unique nilpotent line in $U$. Then $l$ is of the form $F(\gamma u + uv)$ with $\gamma$ satisfying $\gamma^2 = \beta$. 
\end{lemma}
\begin{proof}
Since every nilpotent element in $H$ lies in $J(H) = \spam(u, uv)$ (Prop. \ref{typeD}) we have that $l$ is of the form $F(\gamma u + uv)$ for \emph{some} $\gamma \in F$. Let $x = \delta u + \epsilon v + \zeta uv$ be an element of $U \backslash l$. Since $U \cap J(H) = l$ we see that $\epsilon \neq 0$. Now from $lx \subset l$ (the definition of type U space) we find that there exists an $\eta \in F$ such that $(\gamma u + uv)x = \eta(\gamma u + uv)$ yielding the equalities $\epsilon \beta = \eta \gamma$ and $\gamma \epsilon = \eta$ between elements of $F$. Substituting the second equation into the first and using that $\epsilon \neq 0$ we find that $\beta = \gamma^2$  
\end{proof}
This completes our verification of the column F2J2 of Table \ref{Tvalues}:
\begin{corollary}\label{TXF2J2}
$T_{\textrm{U}, \textrm{F2J2}} = 0$ and $T_{\textrm{D}, \textrm{F2J2}} = q^2 + q$.
\end{corollary}
\begin{proof}
The number $\beta$ in Lemma \ref{gamma2} was defined as the square of a non-nilpotent element of $\im(H)$, which by Lemma \ref{Dns} means that if $H$ is of type F2J2, $\beta$ is a non-square in $F$. In particular the equation $\gamma^2 = \beta$ does not have solutions in $F$ when $H$ is of type F2J2. Lemma \ref{gamma2} thus implies that in that case $H$ contains no type U subspaces and $T_{\textrm{U}, \textrm{F2J2}} = 0$. 

The second equality in the Corollary then follows by subtracting the number of type Z subspaces ($1$) from the total number of two-dimensional subspaces ($\qbinom{3}{2} = q^2 + q + 1$) of $\im(H)$. 
\end{proof}
What remains is computing the number of type U and type D subspaces of $H$ in case $H$ is of type S2J2.
\begin{lemma}\label{TXS2J2}
Let $H$ be of type S2J2. Then there are two lines $l_1$, $l_2$ in $J(H)$ such that a two-dimensional subspace $U$ of $\im(H)$ is of type U if and only if it contains exactly one of $l_1, l_2$. 
\end{lemma}
\begin{proof}
Let $u, v, uv$ be a basis of $\im(H)$ as in Lemma \ref{gamma2}, and let $\beta = v^2$ as in that lemma. By Lemma \ref{Dns} we have that there exist a number $\gamma$ such that $\gamma^2 = \beta$. Let $l_1 = F(\gamma u + uv)$ and $l_2 = F(-\gamma u + uv)$. It is easy to verify that every element of $l_1$ is nilpotent, as is every element of $l_2$. Lemma \ref{gamma2} tells us that every type U subspace of $\im(H)$ contains at least one of $l_1, l_2$, while the definition of type U space as a space containing exactly one nilpotent line prevents it from containing both. %Similarly there cannot be a third line $l_3$ with the described property since any type U space containing $l_3$ must contain one of $l_1, l_2$ as well. 

It remains to show that every two-dimensional subspace $U \subset \im(H)$ containing exactly one of $l_1, l_2$ is of type U. Let $l$ be the line from $\{l_1, l_2\}$ contained in $U$ and $l'$ be the other one. Since every nilpotent element of $H$ lies in $J(H) = \spam(l_1, l_2)$, we see that $U$ cannot contain any nilpotent element outside $l$ since otherwise $U$ would contain $l'$ as well. 

This reduces the possible types of $U$ to U and D and in order to show that $U$ is of type U we only need to verify that $lx \subset l$ for any $x \in U \backslash l$. This is exactly the same computation as we did in the proof of Lemma \ref{gamma2}, only with a different interpretation.
\end{proof}
\begin{corollary}
$T_{\textrm{U}, \textrm{S2J2}} = 2q$ and $T_{\textrm{Ds}, \textrm{S2J2}} = q^2 - q$. 
\end{corollary}
\begin{proof}
Let $H, l_1, l_2$ be as in Lemma \ref{TXS2J2}. The two-dimensional space $J(H)$ contains $q^2$ elements, leaving $q^3 - q^2$ elements in $\im(H) \backslash J(H)$. For each element $x$ among these $q^3 - q^2$ we have that $\spam(l_1, x)$ is a type U space containing $l_1$ by Lemma \ref{TXS2J2}. Since every such space contains $q^2 - q$ elements not on $l_1$ we find that there are $\frac{q^3 - q^2}{q^2 - q} = q$ type U spaces containing $l_1$. An identical computation yields that there are $q$ type U spaces containing $l_2$ and by Lemma \ref{TXS2J2} together these are all type U spaces, bringing the total to $2q$. The number $q^2 - q$ of type D subpaces then follows from subtracting the numbers of type U spaces ($2q$) and type Z spaces ($1$) from the total number of $\qbinom{3}{2} = q^2 + q + 1$ of two-dimensional subspaces of $\im(H)$.
\end{proof}

This completes the verification of Table \ref{Tvalues}. 

\subsection{Subalgebras containing subspaces of type Z}
Our next goal is to verify the numbers $H_{X, Y}$ listed in Table \ref{Qvalues}. We note that the values of $H_{\textrm{Q}, Y}$ listed there follow from Proposition \ref{main}, the values of $H_{\textrm{Dn}, \textrm{Y}}$ and $H_{\textrm{Ds}, \textrm{Y}}$ follow from Proposition \ref{typeD} and Lemma \ref{Dns},
the values of $H_{\textrm{M}, Y}$ follow from Proposition \ref{typeM} and the values of $H_{\textrm{J}, Y}$ follow from proposition \ref{typeJ}. The value of $H_{\textrm{U}, \textrm{M4}}$ has been established in Proposition \ref{Q2M4}, while the equalities $H_{\textrm{U}, \textrm{F2J2}} = 0$ and $H_{\textrm{U}, \textrm{F1J3}} = 0$ follow respectively from the equalities $T_{\textrm{U}, \textrm{F2J2}} = 0$ and $T_{\textrm{U}, \textrm{F1J3}} = 0$ established above. Similarly we see that $H_{\textrm{Z}, \textrm{M4}} = 0$ from the equality $T_{\textrm{Z}, \textrm{M4}} = 0$ derived in Proposition \ref{TXM4}. It remains to compute the values of $H_{\textrm{U}, \textrm{S2J2}}$, and $H_{\textrm{Z}, Y}$ for $Y$ in F2J2, S2J2 and F1J3. We postpone computation of the number $H_{\textrm{U}, \textrm{S2J2}}$ until after we confirmed the values in Tables \ref{Thetavalues} and \ref{Nvalues} and take a more in depth look at the type Z subspaces of $O$ in order to compute the numbers $H_{\textrm{Z}, Y}$.

\begin{lemma}\label{Zideal}
Let $H$ be a four-dimensional associative subalgebra of $O$ containing a type Z subspace $Z$. Then $Z$ is a two-sided ideal of $H$. That is: $xZ \subset Z$ and $Zx \subset Z$ for all $x \in H$.
\end{lemma}
\begin{proof}
Since we established that all four-dimensional associative subalgebras of $O$ are of type M4, F2J2, S2J2 or F1J3 (see the text preceding Notation \ref{subalgtypes}) and no type M4 algebra can contain a type Z space (Prop. \ref{TXM4}), it suffices to verify the claim for algebras of type either F2J2 or S2J2 and for algebras of type F1J3. In the first two cases, Proposition \ref{typeD} tells us that $Z$ is the unique type Z subspace of $H$ and that it equals the Jacobson radical of $H$, which by definition is an ideal. 

In the F1J3-case we pick a basis $u, v, uv$ of $\im(H)$ as in Proposition \ref{typeJ}, that is: satisfying $u^2 = v^2 = (uv)^2 = 0$, $vu = -uv$, $x(uv) = (uv)x = 0$ for all $x \in \im(H)$. We recall from Corollary \ref{F1J3subspaces} that $uv$ is contained in every type Z subspace of $H$ and that hence in particular $uv \in Z$.

Now let $x = \alpha + \beta u + \gamma v + \delta uv$ be a generic element of $H$ and $z = \epsilon u + \zeta v + \eta uv$ be an element of $Z$. Then we compute that $xz = \alpha z + (\beta \zeta - \gamma \epsilon) uv$ which is a sum of two elements of $Z$ and hence an element of $Z$ itself while $zx = \alpha z + (\gamma \epsilon - \beta \zeta) uv$ which is an element of $Z$ for similar reasons.
\end{proof}

We recall from Chapter 5:
\begin{lemma*}[Cor. \ref{ZUtrans}]
The automorphism group $\Aut(O)$ of $O$ acts transitively on the set of type Z subspaces of $O$.
\end{lemma*}
%\begin{proof}%staat al in hoofdstuk 5
%Let $Z_1, Z_2$ be type $Z$ subspaces of $O$ and let $\phi \colon Z_1 \to Z_2$ be \emph{any} linear isomorphism between them. Then $\phi$ %extends to an algebra morphism between the algebras $F1 \oplus Z_1$ and $F1 \oplus Z_2$ by setting $\phi(1) = 1$ (cf Observation \ref{3D}) and hence to an automorphism of all of $O$ by Theorem \ref{transitive}.
%\end{proof}
%[idee achter dit zo expliciet maken is dat we dan later ook snel kunnen zeggen: aantal type Z spaces is $|Aut(O)|/|\GL(2, q)|$. Maar zie nu dat dat veronderstelt dat ieder morphisme $\phi$ uitbreidt tot \emph{uniek} automorphisme. Als in plaats daarvan ieder morphisme uitbreidt tot, zeg, 4 automorphimen dan wordt het aantal type Z spaces kleiner, namelijk $|Aut(O)|/(4 |\GL(2, q)|)$. Hmmm. Onverwachte tegenvaller.]

%Update: nu we het goede antwoord weten kunnen we wel zien hoeveel auto's per auto. Rekenen later een keer uit, zijn nu even bezig met iets anders.

The upshot of Corollary \ref{ZUtrans} is that the number $H_{\textrm{Z}, Y}$ equals the number of type $Y$ subalgebras containing the very explicit type Z space $Z_0 \coloneqq \spam(q_1, r_2)$ (with $q_1, r_2$ as in Table \ref{multtab}), for any automorphism mapping a given type Z space $Z$ to $Z_0$ will provide a bijection between the type $Y$ algebras containing $Z$ to those containing $Z_0$. We will count the latter starting from the following observation.

\begin{lemma}\label{typeZex4}
Let $Z_0 = \spam(q_1, r_2)$ (with notation as in Table \ref{multtab}) and let $w$ in $\im(O) \backslash Z_0$. Then $\spam(1, q_1, r_2, w)$ is a four-dimensional associative algebra containing $Z_0$ if and only if $wZ_0 \subset Z_0$ and $Z_0w \subset Z_0$.
\end{lemma}
\begin{proof}
The `only if' direction follows from Lemma \ref{Zideal}. In the `if' direction, it follows trivially from the right hand side that $H$ is a subalgebra. What remains to be shown is that it is associative, but this was covered in Corollary \ref{4ass}.
\end{proof}
\begin{lemma}\label{Z0idealizer}
Let $Z_0$ be as above and let $w \in \im(O) \backslash Z_0$. Then $wZ_0 \subset Z_0$ if and only if $Z_0w \subset Z_0$ if and only if $w \in W_0 \coloneqq \spam(p_1 - p_2, q_1, r_2, q_3, r_3)$.
\end{lemma}
\begin{proof}
The first equivalence follows from the fact that $Z_0$ is closed under the $*$-operator and the observation that, since $Z_0 \subset \im(O)$ and $w \in \im(O)$, we have that $wz = (-w)(-z) = w^*z^* = (zw)^*$ for any $z \in Z_0$. The second, more interesting, equivalence can be verified directly from Table \ref{multtab} after realizing that the statement $wZ_0 \subset Z_0$ is equivalent to the more down-to-earth statement that both $wq_1 \in Z_0$ and $wr_2 \in Z_0$.
\end{proof}
\begin{lemma}
Let $Z_0, W_0$ be as above and let $w \in W_0 \backslash Z_0$. By the last two lemmas the space $H \coloneqq \spam(1, q_1, r_2, w)$ is a four-dimensional associative algebra containing $Z_0$. We have:
\begin{enumerate}
\item $H$ is of type F2J2 if and only if $w^2$ is a non-square in $F$.
\item $H$ is of type S2J2 if and only if $w^2$ is a non-zero square in $F$.
\item $H$ is of type F1J3 if and only if $w^2 = 0$.
\end{enumerate}
\end{lemma}
\begin{proof}
It is not surprising that F2J2, S2J2 and F1J3 are the only types of algebra appearing in the lemma: by Thm \ref{QMDJ} we know that these, together with type M4, are the only four possibilities and type M4 is ruled out by the fact that $Z_0 \subset H$ while $T_{\textrm{Z}, \textrm{M4}} = 0$. (Prop. \ref{TXM4}.)

First, assume that $H$ is of type either F2J2 or S2J2. The space $T \coloneqq \spam(q_1, w)$ is not equal to $Z_0$ and hence Table \ref{Tvalues} implies that it is of type either D or U. But that means in particular that it contains exactly one nilpotent line and since $q_1^2 = 0$ we cannot have that $w^2 = 0$ as well. Since we already know that the three types listed here are the only possibilities, this establishes the `if' direction of statement (3). Also knowing that $w^2 \neq 0$, having that $H$ is of type F2J2 implies that $w^2$ is a non-square by Lemma \ref{Dns} and having that $H$ is of type S2J2 implies by the same lemma that $w^2$ is a non-zero square. This establishes the `only if' directions of statements (1) and (2).

For the remaining three implications, suppose that $w^2 \neq 0$ and look again at the space $T = \spam(q_1, w)$. Since not every element of $T$ is nilpotent $T$ cannot be of type J or type Z. On the other hand, since not every element is non-nilpotent either it can also not be of type Q. By Theorem \ref{classification} this means that $T$ is of type either U, D, or M. From the equalities $H_{\textrm{U}, \textrm{F1J3}} = H_{\textrm{D}, \textrm{F1J3}} = H_{\textrm{M}, \textrm{F1J3}} = 0$ established above we find that $H$ is not of type F1J3. This establishes the `only if' direction of statement 3. The possibility of $H$ being of type M4 was ruled out above. It follows that $H$ being of type F2J2 or S2J2 are the only remaining options. Finding ourselves once again in the situation where we know that $w^2 \neq 0$ and $H$ is of type either F2J2 or S2J2, we can deduce the remaining implications from Lemma \ref{Dns}.
\end{proof}
\begin{lemma}\label{W0nilpotents}
The five-dimensional space $W_0$ contains $q^4$ nilpotent elements $w$, all satisfying $w^2 = 0$, $q^2$ of which lie in the subspace $Z_0$.
\end{lemma}
\begin{proof}
The third statement is obvious. The second statements follows from Lemma \ref{nilpotent}. The first follows by direct computation. Let $w = \alpha(p_1 - p_2) + \beta_1 q_1 + \beta_2 r_2 + \gamma q_3 + \delta r_3$ be an element of $W_0$. Then $w^2 = \alpha^2 - \gamma \delta$. It follows that every choice of $\alpha, \beta_1, \beta_2 \in F, \gamma \in F \backslash \{0\}$ produces a unique nilpotent element by setting $\delta = \frac{\alpha^2}{\gamma}$ and every choice of $\beta_1, \beta_2, \delta$ gives us a further nilpotent not listed before by setting $\alpha = \gamma = 0$. Together these exhaust all possibilities and one verifies that there are $q^3(q - 1) + q^3 = q^4$ possible choices.
\end{proof}
The last three lemmas enable us to compute the number of elements $w \in O$ such that $H \coloneqq \spam(1, q_1, r_2, w)$ is an algebra of type F1J3. Since in \emph{any} four-dimensional subspace $H$ of $O$ containing $F1 \oplus Z_0$ we have that each of the $q^3 - q^2$ elements $w \in \im(H) \backslash Z_0$ satisfies $H = \spam(1, q_1, r_2, w)$ we find the following result. 
\begin{corollary}
$Z_0$ is contained in $\frac{q^4 - q^2}{q^3 - q^2} = q + 1$ type F1J3 algebras and in $\frac{q^5 - q^4}{q^3 - q^2} = q^2$ algebras of type either F2J2 or D2J2. Hence (by Corollary \ref{ZUtrans}) $H_{\textrm{Z}, \textrm{F1J3}} = q + 1$ and $H_{\textrm{Z}, \textrm{F2J2}} + H_{\textrm{Z}, \textrm{S2J2}} = q^2$.
\end{corollary}.
\begin{lemma}
Of the $q^5 - q^4$ non-nilpotent elements $w$ in $W_0$, we have that $w^2$ is a square in $F$ for $\frac{q^5 - q^3}{2}$ elements and non-square in $F$ for $\frac{q^3(q-1)^2}{2}$ elements.
\end{lemma}
\begin{proof}
As in the proof of Lemma \ref{W0nilpotents} we can write $w = \alpha(p_1 - p_2) + \beta_1 q_1 + \beta_2 r_2 + \gamma q_3 + \delta r_3$ and find $w^2 = \alpha^2 - \gamma \delta$. Now suppose that $w^2 = \epsilon^2$ for some $\epsilon \in F \backslash \{0\}$. Then $\gamma \delta = (\alpha - \epsilon)(\alpha + \epsilon)$. For fixed $\epsilon$ there is a unique value of $\delta$ for each choice of $\alpha \in F$ and $\gamma \in F \backslash \{0\}$. Moreover, there are $q$ choices of $\delta$ when $\alpha \in \{\epsilon, -\epsilon\}$ and $\gamma = 0$. Hence for a given value of $\epsilon$ the equation $\gamma \delta = (\alpha - \epsilon)(\alpha + \epsilon)$ has $q(q-1) + 2q = q(q + 1)$ solutions $(\alpha, \gamma, \delta) \in F^3$. Letting $\epsilon$ range over all non-zero elements of $F$ we see that $\epsilon_1, \epsilon_2$ share a solution $(\alpha, \gamma, \delta)$ only if $\epsilon_1 = \pm \epsilon_2$ and conversely that the solution sets for $\epsilon$ and $-\epsilon$ are completely identical. It follows that the number of triples $(\alpha, \beta, \gamma)$ such that $\alpha^2 - \gamma\delta$ is a non-zero square in $F$ equals $(q-1)q(q+1)/2 = (q^3 - q)/2$. The number of elements $w = \alpha(p_1 - p_2) + \beta_1 q_1 + \beta_2 r_2 + \gamma q_3 + \delta r_3 \in W_0$ such that $w^2$ is a non-zero square in $F$ then equals $q^2$ times this number (the extra factor coming from the choices for $\beta_1$ and $\beta_2$), yielding the number $\frac{q^5 - q^3}{2}$ from the statement of the lemma.

The number of elements $w$ such that $w^2$ is a non-square in $F$ then equals $q^5 - q^4$ minus this number, hence $\frac{q^5 - 2q^4 + q^3}{2} = \frac{q^3(q-1)^2}{2}$.
\end{proof}

\begin{corollary}
$Z_0$ is contained in $\frac{q^5 - q^3}{2(q^3 - q^2)} = \frac{q^2 + q}{2}$ type S2J2 algebras and in $\frac{q^3(q-1)^2}{2(q^3 - q^2)} = \frac{q^2 - q}{2}$ algebras of type F2J2. Hence (by Corollary \ref{ZUtrans}) $H_{\textrm{Z}, \textrm{S2J2}} = \frac{q^2 + q}{2}$ and $H_{\textrm{Z}, \textrm{F2J2}} = \frac{q^2 - q}{2}$.
\end{corollary}
%Moeten dit nog wel verifieren in F_3 case...

\subsection{Final computations}
All that remains is the verification of the eigth numbers $N_{\textrm{F2J2}}, N_{\textrm{S2J2}}, N_\textrm{F1J3}, \Theta_\textrm{Dn}, \Theta_\textrm{Ds}, \Theta_\textrm{J}, \Theta_\textrm{Z}$ and $H_{\textrm{U}, \textrm{S2J2}}$. Moreover, since we know all the relevant numbers $T_{X, Y}$, $H_{X, Y}$ and $\Theta_\textrm{U}$, knowledge of only one of them will allow a quick computation of the other seven using the double counting identity $\Theta_X H_{X, Y} = T_{X, Y}N_Y$. We opt to explicitly compute the number $\Theta_Z$ and derive the other seven values from it.

%Ik heb dit [tweede lemma] al eens uitgerekend maar ben te nieuwsgierig en doe het nog een keer.

\begin{lemma} 
Every non-zero nilpotent $u \in \im(O)$ is contained in exactly $q + 1$ type Z subspaces. 
\end{lemma} 

\begin{proof}
Let $\phi \colon F1 \oplus Fu \to F1 \oplus Fq_1$ be the unique linear map sending $u$ to $q_1$ and $1$ to $1$. Then $\phi$ is actually an algebra isomorphism between the \emph{algebras} $F1 \oplus Fu$ and $F1 \oplus Fq_1$ and hence, by Theorem \ref{transitive} extends to an automorphism of all of $O$. Since automorphisms map type Z spaces to type Z spaces, we see that the number of type Z spaces containing $u$ equals the number type Z spaces containing $q_1$. Since $zq_1 = q_1z = 0$ for every $z$ in such a space, we see that every type Z space containing $q_1$ is contained in the linear space $\ker L_{q_1} \cap \ker R_{q_1}$.

Using Table \ref{multtab} and a little linear algebra we see that $$\ker L_{q_1} \cap \ker R_{q_1} = \spam(q_1, r_2, r_3)$$ and it is easy to verify that $z^2 = 0$ for every $z$ in this space. It follows that there are $q^3 - q$ elements $z \in \im(O)$ such that $\spam(q_1, z)$ is a type Z space (i. e. the $q^3 - q$ elements in $(\ker L_{q_1} \cap \ker R_{q_1}) \backslash Fq_1$) and since each such space contains $q^2 - q$ elements not on the line $Fq_1$ we find that $q_1$ is contained in $\frac{q^3 - q}{q^2 - q} = q + 1$ type Z subspaces.
\end{proof}

\begin{lemma}\label{nilpelts}
$O$ contains $q^6 - 1$ non-zero nilpotent elements.
\end{lemma}
\begin{proof}
We know by Lemma \ref{nilpotent} that every nilpotent $u$ satisfies $u^2 = 0$ and hence is contained in $\im(O)$. In the notation of Table \ref{multtab}, let $u = \alpha (p_1 - p_2) + \sum_{i = 1}^3 (\beta_i q_i + \gamma_i r_i)$ be an element of $\im(O)$. Recalling that $p_1 + p_2 = 1$, Table \ref{multtab} tells us that $u^2 = \alpha^2 - \sum_{i=1}^3 \beta_i \gamma_i \in F1$. It follows that there are $q^5(q - 1)$ nilpotents for which $\gamma_3 \neq 0$, $q^4(q - 1)$ nilpotents for which $\gamma_3 = 0$ but $\gamma_2 \neq 0$, $q^3(q-1)$ nilpotents for which $\gamma_3 = \gamma_2 = 0$ but $\gamma_1 \neq 0$ and $q^3$ nilpotents with $\gamma_1 = \gamma_2 = \gamma_3$ bringing the total number of nilpotents (including 0) to $q^6$.   
\end{proof}

Now since every type Z space contains $q^2 - 1$ non-zero nilpotents we find:

\begin{corollary}
$\im(O)$ contains $\Theta_Z = \frac{(q^6 - 1)(q + 1)}{q^2 - 1} = q^5 + q^4 + q^3 + q^2 + q + 1 = [6]_q$ type Z subspaces.
\end{corollary}

This result enables us to zig-zag our way through the remaining unverified entries in tables \ref{Thetavalues},  \ref{Nvalues} and \ref{Qvalues}.

\begin{corollary}
$O$ has 
\begin{itemize}
\item $N_{\textrm{F2J2}} = H_{\textrm{Z}, \textrm{F2J2}}\Theta_\textrm{Z}/T_{Z, \textrm{F2J2}} = \frac{(q^2 - q)[6]_q}{2\cdot 1} = \frac{q^7 - q}{2}$ subalgebras of type F2J2,
\item $N_{\textrm{S2J2}} = H_{\textrm{Z}, \textrm{S2J2}}\Theta_\textrm{Z}/T_{Z, \textrm{S2J2}} = \frac{(q^2 + q)[6]_q}{2\cdot 1} = \frac{q^7 + 2q^6 + 2q^5 + 2q^4 + 2q^3 + 2q^2 + q}{2}$ subalgebras of type S2J2, and
\item $N_{\textrm{F1J3}} = H_{\textrm{Z}, \textrm{F1J3}}\Theta_\textrm{Z}/T_{Z, \textrm{F1J3}} = \frac{(q+1)[6]_q}{q+1} = [6]_q$ subalgebras of type F1J3.
\end{itemize}
\end{corollary}

\begin{corollary}
$\im(O)$ has 
\begin{itemize}
\item $\Theta_{\textrm{Dn}} = T_{\textrm{Dn}, \textrm{F2J2}}N_{\textrm{F2J2}}/H_{\textrm{Dn}, \textrm{F2J2}} = \frac{(q^2 + q)(q^7 - q)}{2\cdot 1} = \frac{q^9 + q^8 - q^3 - q^2}{2}$ subspaces of type Dn, 
\item $\Theta_{\textrm{Ds}} = T_{\textrm{Ds}, \textrm{S2J2}}N_{\textrm{S2J2}}/H_{\textrm{Ds}, \textrm{S2J2}} = \frac{(q^2 - q)(q^7 + 2q^6 + 2q^5 + 2q^4 + 2q^3 + 2q^2 + q)}{2\cdot 1} = \frac{q^9 + q^8 - q^3 - q^2}{2}$ subspaces of type Ds, hence %opmerkelijk dat die precies hetzelfde zijn.
\item $\Theta_\textrm{D} = q^9 + q^8 - q^3- q^2$ subspaces of type D in total, and 
\item $\Theta_{\textrm{J}} = T_{\textrm{J}, \textrm{F1J3}}N_{\textrm{F1J3}}/Q_{\textrm{J}, \textrm{F1J3}} = \frac{q^2[6]_q}{1} = [8]_q - [2]_q$ subspaces of type J.
\end{itemize}
\end{corollary}

\begin{remark}
By Lemma \ref{complete} every two-dimensional subspace of $\im(O)$ is of one of the six types Q, U, D, M, J, Z. It follows that the numbers in Table $\ref{Thetavalues}$ should add up to the total number $\qbinom{7}{2}$. Of course, since both are polynomials in $q$ we can verify the validity of this statement without knowing its interpretation. Doing so (for checking purposes), we see that the sum of the $\Theta_X$ equals 
%$\frac{1}{2}(q^{10} - q^9 + q^8 - q^7 + q^6 - q^5) + \frac{1}{2} (q^{10} + q^9 + q^8 + q^7 + q^6 + q^5) + (q^9 + q^8 - q^3 - q^2) + 2(q^7 + q^6 + q^5 + q^4 + q^3 + q^2) + (q^5 + q^4 + q^3 + q^2 + 1) = (q^{10} + q^8 + q^6) + (q^9 + q^8 - q^3 - q^2) + 2(q^7 + q^6 + q^5 + q^4 + q^3 + q^2) + (q^5 + q^4 + q^3 + q^2 + q + 1) =$
$q^{10} + q^9 + 2q^8 + 2q^7 + 3q^6 + 3q^5 + 3q^4 + 2q^3 + 2q^2 + q + 1$ while $\qbinom{7}{2}$ expands to $q^{10} + q^9 + 2q^8 + 2q^7 + 3q^6 + 3q^5 + 3q^4 + 2q^3 + 2q^2 + q + 1$ as well.
\end{remark}

\begin{corollary}
Every type U subspace of $\im(O)$ is contained in 
\begin{itemize}
\item $H_{\textrm{U}, \textrm{S2J2}} = \frac{T_{\textrm{U}, \textrm{S2J2}}N_{\textrm{S2J2}}}{\Theta_\textrm{U}} = \frac{q(q^7 + 2q^6 + 2q^5 + 2q^4 + 2q^3 + 2q^2 + q)}{q^7 + q^6 + q^5 + q^4 + q^3 + q^2} = q + 1$ %nice, maar moeten ook nog verifieren over F3.
subalgebras of type $S2J2$ 
\end{itemize}
in addition to the $H_{\textrm{U}, \textrm{M4}} = q^2$ subalagbras of type M4 we already verified it is contained in.
\end{corollary}

\begin{remark}\label{UZmystery} In Theorem \ref{classification} the spaces of types U and Z were introduced as the only two-dimensional spaces $T$ such that $F1 \oplus T$ is closed under multiplication and hence does not generate its `own' block $\im(\langle T \rangle)$ in the collection $\mathcal{B}$ of Thm. \ref{main1}. It follows that if and when this collection $\mathcal{B}$ fails to be a $q$-Fano plane this is due to the type U and type Z subspaces. However, from this `a priori' standpoint one would expect that the failure would consist of type U and Z spaces not being contained in any block. It is somewhat remarkable that in reality the opposite happens and that the failure of $\mathcal{B}$ to be a $q$-Fano plane is due to the type U and type Z spaces being contained in more than one block rather than zero. Even more remarkable is that this number, $[3]_q$, is the same for both types of subspaces. It would be desirable to have a conceptual explanation for this latter fact to complement the computational proof above.
\end{remark}

In conclusion, we see that
\begin{corollary}
The number of blocks in the $q$-covering design of Thm. \ref{main2} in the special case that $F$ is the field $\mathbb{F}_q$ of $q$-elements (where $q$ is an odd prime power) equals $N_\textrm{M4} + N_\textrm{F2J2} + N_\textrm{S2J2} + N_\textrm{F1J3} = q^8 + q^7 + 2q^6 + 2q^5 + 3q^4 + 2q^3 + 2q^2 + q + 1$. 
\end{corollary}

To complete the proof of Theorem \ref{mainintro} it only remains to check that

\begin{equation}\label{sixtwo}
q^8 + q^7 + 2q^6 + 2q^5 + 3q^4 + 2q^3 + 2q^2 + q + 1 = \qbinom{6}{2},
\end{equation}
which is left to the reader.

\subsection{The mysterious equality (\ref{sixtwo})}
The remarkable equality (\ref{sixtwo}) raises the following question, which makes sense for infinite as well as finite $F$:

\begin{question}
Is the set $\mathcal{H}$ of all 4-dimensional, (hence) associative subalgebras of the split Cayley algebra $O$ over $F$ in bijection with set of all two-dimensional subspaces of a fixed 6-dimensional vector space $S$ over $F$ in a natural %or otherwise cannonical 
way?
\end{question} 
Unsurprisingly, I don't know the answer to this question. If I had, I would have short-cut the current section by deriving the number $\qbinom{6}{2}$ in Theorem \ref{mainintro} directly from such a bijection. Now I will instead make the section even longer by showing that the most `obvious' approach to finding such a bijection doesn't work.

\begin{lemma}
Let $O$ be a Cayley algebra over $F$ and let $\mathcal{H}$ be the set of all its four-dimensional subalgebras. Then there exist no six-dimensional subspace $S \subset O$ satisfying:
\begin{enumerate}
\item Every two-dimensional subspace of $S$ is contained in exactly one element of $\mathcal{H}$
\item Every element of $\mathcal{H}$ contains exactly one two-dimensional subspace of $O$.
\end{enumerate}
\end{lemma}
\begin{proof}
We assume that such an $S$ is given and derive a contradition. Let $S' \coloneqq S \cap \im(O)$ so that $\dim S' \geq 5$. By Ex. \ref{problem}, Lemmas \ref{typeZex4}, \ref{Z0idealizer} (which do not depend on $F$ being finite) and Cor. \ref{ZUtrans}, the first of the two conditions implies that $S'$ contains no two-dimensional subspaces of type U or Z (cf also remark \ref{UZmystery}). In fact, combining this with Cor. \ref{almost} we see that the condition is equivalent to $S'$ containing no type U or Z subspaces. The second condition on $S$ implies that for each $H \in \mathcal{H}$ we have that $\dim (H \cap S) = 2$. Together these conditions imply

\begin{equation}\label{outsideS}
ab \nin S \textrm{ for all linearly independent } a, b \in S'.
\end{equation}
After all: when $ab \in \spam(a, b) \subset S'$ we have, by Lemma \ref{uvinU} and Proposition \ref{3dim}, that $\spam(a, b)$ is of type either U or Z, contradicting the first condition in the lemma. When $ab \nin \spam(a, b)$ on the other hand, this means that the algebra $\langle a, b \rangle \supseteq \spam(1, a, b, ab)$ is at least four dimensional, hence exactly four dimensional by Lemma \ref{3of4} and hence an element of $\mathcal{H}$. The second condition then implies that $\spam(1, a, b, ab) \cap S = \spam(a, b)$ and hence $ab \nin S$.

Now let $a \in S'$ be non-zero and consider left multiplication by $a$ as a linear map: $L_a: S' \to O$. Since $0 \in S$, the kernel of this map is contained in the linear subspace $Fa$ of $S'$ by (\ref{outsideS}) and hence the image $L_a(S')$ of the at least five-dimensional space $S'$ under this map is at least 4-dimensional. It follows that $L_a(S') \cap S$ is at least two-dimensional and so in particular contains a non-zero element $c$. In more down to earth terms this means that there are $a, b \in S'$, $c \in S$ such that $ab = c$. By (\ref{outsideS}) this means that $b$ is a scalar multiple of $a$, but as $a$ is imaginary, this in turn implies that $c$ is a non-zero scalar multiple of 1. 

On one hand this means that $c \nin \im(O)$ and hence $c \nin S'$. But on the other hand this means that $c$ is contained in $S \cap H$ for every $H \in \mathcal{H}$. Let $U \subset S'$ be any two-dimensional subspace and let $H \in \mathcal{H}$ be the unique four-dimensional subalgebra containing $U$, which exists by the first condition from the lemma. Then for every $d \in U$ we have on one hand that $\spam(c, d) \neq U$ and on the other hand that $\spam(c, d)$ is a two-dimensional subpace of $S$ contained in $H$. This contradicts the second condition.
\end{proof}
%Lukt het nu met die $\qbinom{a}{b}$?

\bibliographystyle{alpha}
\bibliography{ref}
\end{document}